\documentclass[11pt,a4paper,reqno]{amsart}%article}%
\usepackage{amsfonts}
\usepackage{amsmath,amssymb,amsthm,amsxtra}
\usepackage{float}
\usepackage[colorlinks, linkcolor=blue,anchorcolor=Periwinkle,
    citecolor=red,urlcolor=Emerald]{hyperref}
\usepackage[usenames,dvipsnames]{xcolor}
\usepackage{enumitem}
\usepackage{geometry,array} 
\usepackage{graphicx}
\usepackage{subfigure}
\usepackage{bookmark}
\usepackage{tikz}\usetikzlibrary{matrix}
\usepackage{url}

\usetikzlibrary{decorations.markings}
\tikzset{->-/.style={decoration={  markings,  mark=at position #1 with
    {\arrow{>}}},postaction={decorate}}}
\tikzset{-<-/.style={decoration={  markings,  mark=at position #1 with
    {\arrow{<}}},postaction={decorate}}}
\usepackage[all]{xy}
\usetikzlibrary{arrows}
\usetikzlibrary{decorations.pathreplacing,decorations.pathmorphing,shadings,fadings,calc}
\usepackage{chemarrow}
\let\rightarrow\chemarrow
\let\to\chemarrow

\usepackage{extarrows}
\theoremstyle{plain}
\newtheorem{theorem}{Theorem}[section]

\newtheorem{lemma}[theorem]{Lemma}
\newtheorem{corollary}[theorem]{Corollary}
\newtheorem{proposition}[theorem]{Proposition}

\theoremstyle{definition}
\newtheorem{definition}[theorem]{Definition}

\newtheorem{remark}[theorem]{Remark}
\newtheorem{notations}[theorem]{Notations}
\numberwithin{equation}{section}

\def\hua{\mathcal}

\def\<{\langle}
\def\>{\rangle}
\def\Aut{\operatorname{Aut}}
\def\Autp{\operatorname{Aut}^\circ}

\def\Sim{\operatorname{Sim}}
\def\Hom{\operatorname{Hom}}

\def\Stap{\operatorname{Stab}^\circ}

\renewcommand{\k}{\mathbf{k}}

\def\numbers{\begin{enumerate}[label=\arabic*{$^\circ$}.]}
\def\ends{\end{enumerate}}

\newcommand{\EG}{\operatorname{EG}}       %exchange graph of heart (oriented)
\newcommand{\D}{\operatorname{\hua{D}}}

\newcommand\Sph{\operatorname{Sph}}%^\circ}
\def\zero{\hua{H}_\Gamma}
\newcommand{\Tri}{\bigtriangleup}
\def\arrow{red}
\def\surf{\mathbf{S}}                       %FST's surface
\def\ts{\Sigma_{\TT}}

\newcommand{\ST}{\operatorname{ST}}        %spherical twists
\newcommand{\BT}{\operatorname{BT}}        %braid twists
\newcommand{\MCG}{\operatorname{MCG}}

\newcommand{\Int}{\operatorname{Int}}
\def\TT{\mathbf{T}}

\def\M{\mathbf{M}}

\def\surfo{{\mathbf{S}}_\Tri}
\def\surfO{{\operatorname{S}}_\aleph}
\def\cA{\operatorname{CA}}

\newcommand\Bt[1]{\operatorname{B}_{#1}}

\newcommand{\Quad}{\operatorname{Quad}}
\newcommand{\FQuad}[2]{\operatorname{FQuad}^{#1}(#2)}

\newcommand{\SBG}{\operatorname{SBr}}
\newcommand\iv[1]{\underline{#1}}
\newcommand\Br{\operatorname{Br}}
\newcommand\Co{\operatorname{Co}}
\newcommand\SCo{\operatorname{SCo}}

\newcommand\Tr{\operatorname{Tr}}

\def\tx{\widetilde{x}}
\def\ty{\widetilde{y}}
\def\tt{\widetilde{\tau}}
\def\ts{\widetilde{\sigma}}

\usepackage{braids}
\title[Finite presentations for spherical/braid twist groups]
{Finite presentations for spherical/braid twist groups from decorated marked surfaces}

\author{Yu Qiu}
\address{YQ:
Yau Mathematical Sciences Center,
Tsinghua University,
Beijing, China}
\email{yu.qiu@bath.edu}

\author{Yu Zhou}
\address{YZ:
Yau Mathematical Sciences Center,
Tsinghua University,
Beijing, China}
\email{yuzhoumath@gmail.com}

\date{\today}
\begin{document}
%=========================================================
%=========================================================

%=========================================================
\begin{abstract}
%We give a finite presentation for the braid twist group $\BT(\surfO)$ of a decorated surface $\surfO$. If the surface $\surfO$ is obtained from a triangulation $\TT$ of a marked surface, we obtain a finite presentation for the spherical twist group of the 3-Calabi-Yau category $\D_{fd}(\Gamma_\TT)$ associated to $\TT$. In the sequel, the result will be used to prove that the (principal component of) space of stability conditions on $\D_{fd}(\Gamma_\TT)$ is simply connected.

We give a finite presentation for the braid twist group
%a generalization of braid group
of a decorated surface.
If the decorated surface arises from a triangulated marked surface without punctures, we obtain a finite presentation for the spherical twist group of the associated 3-Calabi-Yau triangulated category.
The motivation/application is that the result will be used to show that the (principal component of) space of stability conditions on the 3-Calabi-Yau category is simply connected
in the sequel \cite{KQ1}.
\end{abstract}


\maketitle
\setcounter{tocdepth}{1}
\tableofcontents\addtocontents{toc}{\setcounter{tocdepth}{1}}

%=========================================================

\def\dgen{2\mathbb{N}_{\leq g}-1}
\def\dgene{2\mathbb{N}_{\leq g}}

%=========================================================
\section{Introduction}
%=========================================================
Artin's braid group is a classical object in low dimensional topology,
which links to many areas in mathematics: e. g. knot theory, representation theory of algebras, monodromy invariants of algebraic geometry, cf. the survey \cite{BB} and the textbook \cite{KT}.
We are interested in a generalization of braid group, the braid twist groups of decorated surfaces,
which fits into the theory of cluster algebras and stability conditions.

%=========================================================
\subsection{Motivation}\label{subsec:mot}
%=========================================================
In Bridgeland-Smith's seminal work \cite{BS},
they established a connection between Teichm\"{u}ller theory and theory of stability conditions on triangulated categories.
More precisely, let $\surf$ be a marked surface (i.e. an oriented compact surface with or without non-empty boundary, equipped with a collection of marked points in the sense of Fomin-Shapiro-Thurston \cite{FST}). There is an associated 3-Calabi-Yau triangulated category $\D(\surf)$  (see Section~2.3 for the construction). The authors proved that \cite[Theorem~1.2 and Theorem~1.3]{BS} there is an isomorphism of complex orbifolds
%for a marked surface $\surf$, there is an isomorphism of complex orbifolds
\begin{gather}\label{eq:motivation}
    \Stap\D(\surf)/\Autp\D(\surf)\cong\Quad(\surf),
\end{gather}
where $\Quad(\surf)$ is the moduli space of quadratic differentials
on $\surf$,
%$\D(\surf)$ the 3-Calabi-Yau triangulated category associated to $\surf$ (see Section~2.3 for the definition),
$\Stap\D(\surf)$ the principal component of space of stability conditions on $\D(\surf)$ and $\Autp\D(\surf)$ the group of auto-equivalences of $\D(\surf)$ which preserve this component.
Their motivations are coming from string theory in physics,
Donaldson-Thomas theory and homological mirror symmetry (cf. \cite{GMN}, \cite{S} and \cite{QQ}).

The main aim of this series of works (\cite{QQ, QZ2, KQ1})
is to show that $\Stap\D(\surf)$ is simply connected.
Thus we are interested in the symmetry groups in the formula \eqref{eq:motivation},
in particular, the spherical twist group (see Section~2.3 for the definition)
$\ST\D(\surf)\subset\Autp\D(\surf)$ that sits in the short exact sequence of groups  \cite[Theorem~9.9]{BS}
\[
    1\to\ST\D(\surf)\to\Autp\D(\surf)\to\MCG(\surf)\to1,
\]
where $\MCG(\surf)$ is the mapping class group of $\surf$, which consists of the isotopy classes of orientation-preserving homeomorphisms of $\surf$ that fix the set of marked points.
Such spherical twist groups were first studied by Khovanov, Seidel and Thomas \cite{KS,ST} from the two sides of the homological mirror symmetry
in the case when $\surf$ is a disk.

In the prequel \cite{QQ},
we introduced the decorated marked surface $\surfo$ (for $\surf$ unpunctured, i.e. all marked points are on the boundary of $\surf$), which is obtained from $\surf$ by decorating a certain set $\Tri$ of points in the interior of $\surf$ (see Definition~2.7). Let $\MCG(\surfo)$ be the mapping class group of $\surfo$, i.e. the group consisting of the isotopy classes of orientation-preserving homeomorphisms of $\surfo$ that fix the boundary of $\surfo$ pointwise and preserve the decoration set $\Tri$ setwise. We showed that (\cite[Theorem~1]{QQ}) $\ST\D(\surf)$ is isomorphic to a subgroup
%of the mapping class group $\MCG(\surfo)$,
\[\BT(\surfo)\subset\MCG(\surfo),\]
%the braid twist group
the braid twist group of $\surfo$ (see Definition~\ref{def:bt} and Section~\ref{sec:MS}).
%\[\BT(\surfo)\subset\MCG(\surfo),\]
%where $\MCG(\surfo)$ is the mapping class group of $\surfo$ consisting of
In this paper, we give finite presentations for these twist groups,
which will play a key role in achieving the aim,
i.e. proving the simply connectedness of $\Stap\D(\surf)$
(in the sequel \cite[Theorem~4.16]{KQ1}).

%=========================================================
\subsection{Braid groups via quivers with potential}
%=========================================================
The key ingredient in the construction of the 3-Calabi-Yau triangulated category $\D(\surf)$ associated to a marked surface $\surf$ is the so-called quiver with potential in the cluster theory.
%A quiver with potential is a pair $(Q,W)$ of a direct graph $Q$ and a linear combination $W$ of cycles in $Q$.
This notion was introduced by Derksen, Weyman and Zelevinsky \cite{DWZ} in the general case and was developed by Fomin, Shapiro and Thurston \cite{FST} and Labardini-Fragoso \cite{LF} for the surface case. Given a triangulation $\TT$ of a marked surface $\surf$, there is an associated quiver with potential $(Q_\TT,W_\TT)$ whose vertices are indexed by the arcs in $\TT$, whose arrows are indexed by the oriented angles of triangles of $\TT$ and whose potential $W_\TT$ is the sum of 3-cycles arising from triangles of $\TT$. The category $\D(\surf)$ is defined to be the finite dimensional derived category of the Ginzburg dg algebra associated to $(Q_\TT,W_\TT)$. See Section~\ref{sec:MS} for more details.
%(see Section~\ref{sec:MS} for more details).
%And $\D(\surf)\simeq\D(\Gamma_\TT)$ is
% and in our presentations for $\BT(\surfo)\cong\ST\D(\surf)$
%is quivers with potential in the cluster theory.

There is a preferable set of generators for the corresponding spherical twist group
$\ST\D(\surf)$, which is indexed by the vertices of the quiver $Q_\TT$ (see Section~\ref{sec:CY}).
%Correspondingly, the braid twists along dual (closed) arcs of $\TT$ in $\surfo$ form a preferable set of generators for $\BT(\surfo)$.
It is natural to try to find a set of (generating) relations for this set of generators.
Then our work fits into a larger program:
to introduce (algebraic) braid group $\Br(Q,W)$ associated to a quiver with potential $(Q,W)$.

For the (classical) braid group on $\aleph$ strands, we have the following well-known
%In the classical case (i.e. the quiver is of type $A$ with zero potential), the following
Artin presentation
%is well-known
\cite{A}:
%%\begin{itemize}
%\item The braid group $B_{\aleph}$ admits a presentation,
%whose
\begin{itemize}
	\item Generators: $\sigma_{i}$, $1\leq i\leq \aleph-1$.
	\item Relations: (cf. the notation in Section~\ref{sec:notation}).
	
	$\begin{array}{lll}
	&\Co(\sigma_i,\sigma_j)&  \text{if $|i-j|\neq1$;}\\
	&\Br(\sigma_i,\sigma_{j})&\text{if  $|i-j|=1$.}
	\end{array}$

\end{itemize}
%and
%whose
%\end{itemize}
This naturally generalizes to the Artin group (or generalized braid group) associated to a Dynkin diagram $\nabla$:
\begin{itemize}
	\item Generators: $\sigma_i$, $i$ vertices of $\nabla$.
	\item Relations:
	
	$\begin{array}{lll}
	&\Co(\sigma_i,\sigma_j)&  \text{if there is no edge between $i$ and $j$ in $\nabla$;}\\
	&\Br(\sigma_i,\sigma_{j})&\text{if there is exactly one edge between $i$ and $j$ in $\nabla$.}
	\end{array}$
	%$\Co(\sigma_i,\sigma_j)$ when there is no edge between $i$ and $j$ in $\nabla$ and $\Br(\sigma_i,\sigma_j)$ when there is exactly one edge between $i$ and $j$ in $\nabla$.
\end{itemize}% (cf. the notation in Section~\ref{sec:notation}).
Recently, Grant and Marsh \cite{GM} generalized this to the quivers with potential %In the case when a quiver with potential $(Q,W)$
which are mutation equivalent to a Dynkin quiver.
% (i.e. a Dynkin diagram endowed with an orientation). %$\nabla$ with zero potential,
%Grant-Marsh \cite{GM} has calculated the corresponding presentation for the braid group $\Br(Q,W)\cong\Br(\nabla)$,
See also \cite[Proposition~10.3]{QQ}.
Qiu and Woolf \cite{QW} showed that the generalized braid group $\Br(Q,W)$ in this case is indeed isomorphic to the corresponding spherical twist group.

In this paper, we will introduce the braid groups associated to the
quivers with potential which comes from marked surfaces
and show that they are isomorphic to the corresponding braid/spherical twist groups.
Note that the main difficulty lies in showing the predicted relations
in \cite{QQ} are enough (at least in some good cases, see Proposition~\ref{prop:simple}).
The result is in the same line of the faithfulness of spherical twist actions in \cite{KS,ST,QW}, which all imply that the corresponding spaces of stability conditions
are simply connected (cf. the survey \cite{Qs});
so does ours in \cite{KQ1}.

%=========================================================
\subsection{Relation with surface braid groups}
%=========================================================

The (classical) braid group $B_{\aleph}$ on $\aleph$ strands can be realized
(topologically) as $\MCG(\mathbf{D}_\aleph)$,
where $\mathbf{D}_\aleph$ is a disk $\mathbf{D}$ with $\aleph$ decorating points.
One can also realize $B_{\aleph}$ as collections of $\aleph$ strands on
%the disk
$\mathbf{D}_{\aleph}$ %with $\aleph$ points
(or as the fundamental group of the corresponding configuration space, cf. Figure~\ref{fig:braids}).
\begin{figure}[ht]
	\begin{tikzpicture}[xscale=.7,yscale=.7,rotate=90]
	\draw[dashed, fill=gray!11] (3,0) ellipse (2.3 and .8);
	\draw[dashed, fill=gray!11] (3,-9.5) ellipse (2.3 and .8);
	\braid[thick, style strands={1}{ForestGreen},
	style strands={2}{blue},
	style strands={3}{violet},
	style strands={4}{orange},
	style strands={5}{green}]
	s_1 s_3 s_4 s_2^{-1} s_1 s_3 s_2^{-1} s_1 s_2^{-1};
	\foreach \j in {3,4,5,2,1}
	{\draw (\j,0)node[white]{$\bullet$}node[red]{$\circ$}
		(\j,-9.5)node[white]{$\bullet$}node[red]{$\circ$};}
	\end{tikzpicture}\caption{Classical braids}\label{fig:braids}
\end{figure}
This realization leads to the following natural generalization of braid group.
Given a decorated surface $\surfO$, which is an oriented compact surface $\text{S}$
with non-empty boundary and a finite set of $\aleph$ decorating points in its interior,
the surface braid group
$\SBG(\surfO)$ consists of the collections of strands on $\surfO$ (cf. \cite[Section~2.1]{GJP}). The surface braid group $\SBG(\surfO)$ can be realized as a subgroup of the mapping class group $\MCG(\surfO)$ (cf. Definition~2.5) and sits in the following short exact sequence of groups
%the subgroup of the mapping class group $\MCG(\surfO)$ of $\surfO$
%that sits in the short exact sequence
\[1\to\SBG(\surfO)\to\MCG(\surfO)\to\MCG(\text{S})\to1,\]
cf. \cite[\S~2.4 (5)]{GJP}.

%There are various other generalizations of braid groups, cf. the survey \cite{Qs}.
%In particular,

For our purpose, we also want to consider the marking $\M$ on the surface $\text{S}$,
which is a set of marked points on the boundary of $\text{S}$.
Denote by $\surf$ the pair $($\text{S}$,\M)$, known as a marked surface,
and by $\surfo$ the associated decorated marked surface as in Section~\ref{subsec:mot},
where $\Tri$ is the decoration set of $\aleph$ points.
Our braid twist group $\BT(\surfo)$ is in fact a subgroup of
the surface braid group $\SBG(\surfo)$.
% for $\surfO=\surfo$.
%More precisely,
%which is a decorated surface with extra marked points on its boundary.
We denote by $\FQuad{}{\surf}$ the moduli space of $\surf$-framed quadratic differentials
and by $\FQuad{}{\surfo}$ the moduli space of $\surfo$-framed quadratic differentials.
Then we have the following commutative diagram of coverings
\begin{equation}\label{eq:quad1}
\begin{tikzpicture}[xscale=1.8,yscale=1.2]
\draw (-1,1.5) node (s) {$\FQuad{}{\surfo}$}
  (1,1.5) node (s1) {$\FQuad{}{\surf}$}
  (0,0) node (s2) {$\Quad(\surf)$};
\draw [->, font=\scriptsize]
  (s1) edge[->,>=stealth] node [right] (m1) {$\;\MCG(\surf)$} (s2);

\draw[font=\scriptsize]
     (s)edge[->,>=stealth] node[above] {$\SBG(\surfo)$} (s1)
     (s)edge[->,>=stealth] node[below left] {$\MCG(\surfo)$} (s2);
\end{tikzpicture}
\end{equation}
The space $\FQuad{}{\surfo}$ is not connected in general (unless $\surf$ is a disk).
For any connected component $\FQuad{\circ}{\surfo}$ of $\FQuad{}{\surfo}$,
the isomorphism \eqref{eq:motivation} can be upgraded to (\cite[Theorem~4.14]{KQ1})
\begin{gather}\label{eq:stab=fquad}
    \kappa^\circ\colon\Stap\D(\surf)\cong\FQuad{\circ}{\surfo},
\end{gather}
and the covering group of the covering
%we are interested in the covering group $\BT(\surfo)$ of
\[
    \FQuad{\circ}{\surfo}\rightarrow{} \FQuad{}{\surf},
\]
is the braid twist group $\BT(\surfo)$, which hence is a subgroup of $\SBG(\surfo)$.

Note that finite presentations of the mapping class group $\MCG(\surfO)$ of a general decorated surface $\surfO$ have been heavily studied (cf. \cite{FM})
and finite presentations of the surface braid group $\SBG(\surfO)$ have
been calculated by Bellingeri \cite{Be}. Our main strategy is to use the well-studied presentations of $\SBG(\surfo)$
(by many works \cite{Be, BG, GJP})
to find finite presentations of $\BT(\surfo)$ that fit into our motivation.
%In particular, in the sequel, we apply our result to show that the spaces in \eqref{eq:stab=fquad}
%are indeed simply connected (\cite[Theorem~4.16]{KQ1}).

%=========================================================
\subsection{Context and notations}\label{sec:notation}
%=========================================================
The paper is organized as follows.
In Section~\ref{sec:bg}, we review the background of braid/spherical twist groups
and surface braid groups.
In Section~\ref{app:pf sbg}, we give an alternative presentation of surface braid group.
In Section~\ref{sec:bt}, we find a first finite presentation of the braid twist group
of a general decorated surface.
In Section~\ref{sec:st}, we calculate finite presentations for the braid/spherical twist group
of a decorated marked surface via quivers with potential. The main results/presentations are
\begin{description}
\item[Theorem~\ref{thm:pre}] A finite presentation for the braid twist group of a decorating surface.
% $\BT(\surfO)$.
\item[Theorem~\ref{thm:1}] Finite presentations for the braid twist group of a decorated marked surface via quivers with potential.
%$\BT(\surfo)$ with respect to a given triangulation $\TT$.
\item[Corollary~\ref{Cor:1}] Finite presentations for the spherical twist group of the 3-Calabi-Yau triangulated category from a marked surface via quivers with potential.
%$\ST(\Gamma_\TT)$
%with respect to the quiver with potential $(Q_\TT,W_\TT)$.
\item[Theorem~\ref{thm:2}] (Infinite) presentations for braid/spherical twist groups %$\BT(\surfo)$ and $\ST(\Gamma_\TT)$
with only conjugation relations.
%commutation and braid relations.
\end{description}

In mapping class groups, autoequivalence groups or their subgroups, we have the following conventions.
\begin{itemize}
\item \textbf{Multiplication:} the multiplication $ab$ stands for the composition $a\circ b$, that is, first $b$ then $a$.% Our convention of composition is from left to right (as product).
%So for two (spherical) functors, we will write $\phi_1\phi_2$ for $\phi_2\circ\phi_1$.
%Note that this is different from the composition convention in \cite{QQ}.
\item \textbf{Inverse} and \textbf{Conjugation:} for simplifying notation in calculations, we will use $\iv{s}$ to denote the inverse $s^{-1}$ of an element $s$ and use $a^b$ to denote the conjugation $\iv{b}ab$ of $a$ by $b$. The easy formula $\iv{(a^b)}=\iv{a}^b$ will be used frequently.
\item \textbf{Relation:} we will use the following notation for relations throughout this paper.
\[\begin{array}{llll}
\text{Commutation relation}&\Co(a,b)&\colon& ab=ba\\
\text{Braid relation}&\Br(a,b)&\colon& aba=bab\\
\text{Skewed commutation relation}&\SCo(x;a,b)&\colon& xaxbx=bxa\\
\text{Triangle relation}&\Tr(a,b,c)&\colon& abca=bcab=cabc.
%&\SCoi(x;a,b)=
%&\Co(a^{\iv{x}},b)\colon& xa\iv{x}b=bxa\iv{x}\\
%&\Ss(a;b,c)\colon& (abc)^2=(bca)^2\\
%&\Nr(x,y;a,b)\colon& xaybxa=aybxay   \\
%&\Nrr(x,y;a,b)=\Nr(y,x;b,a)\colon& ybxayb=bxaybx   \\
%&\Nl(x,y;a,b)=\Co(a,xyabxy)\colon& axyabxy=xyabxya \\
%&\Nll(x,y;a,b)=\Co(b,xyabxy)\colon& bxyabxy=xyabxyb \\
\end{array}\]
\end{itemize}
The following two subsets of the set $\mathbb{N}$ of natural numbers will be used in the paper:
%\[\mathbb{N}_{\leq g}=\{h\in\mathbb{Z}\mid 1\leq h\leq g \};\]
\[\dgene=\{2h\mid h\in\mathbb{N}\text{ with }1\leq h\leq g \};\]
\[\dgen=\{2h-1\mid h\in\mathbb{N}\text{ with }1\leq h\leq g \}.\]
%\[\dgen=\{2h-1\mid 1\leq h\leq g \};\]
%and $\dgene=\{2h\mid 1\leq h\leq g \}$.

%\todo[inline,color=green!40]{ever $\Nl(x,y;a,b)\colon axyabxyb=yabxyabx$, and this with now implies $xabxyab=abxyabx$.}
%$\Ss(a,b,c)\xLongleftrightarrow[\Br(a,b)]{\Br(a,c)}\Br(a^b,c)$
%=========================================================
\subsection*{Acknowledgments}
%=========================================================
This work is supported by National Natural Science Foundation of China (Grants No.11801297), Beijing Natural Science Foundation (Z180003), Tsinghua University Initiative Scientific Research Program (2019Z07L01006), Hong Kong RGC 14300817 and Direct Grant 4053293 (from Chinese University of Hong Kong).

%=========================================================
\section{Preliminaries}\label{sec:bg}
%=========================================================
\subsection{Decorated surfaces and two braid groups}\label{sec:DS}
%=========================================================

A \emph{decorated surface} $\surfO$ is a compact connected oriented surface $\mathrm{S}$ with non-empty boundary $\partial\mathrm{S}$, endowed with a set $\Delta=\{Z_1,\cdots,Z_\aleph\}$ of $\aleph$ points in the interior of $\mathrm{S}$. The points $Z_i$, $1\leq i\leq \aleph$, are called \emph{decorating points} in $\surfO$.
%the surface obtained from $\mathrm{S}$ by decorating $\aleph$ points in the interior of $\mathrm{S}$.
%Let $\mathrm{S}$ be . Up to homeomorphism, $\mathrm{S}$ is determined by the following data
Denote by
\begin{itemize}
\item $g$ the genus of $\mathrm{S}$, and
\item $b$ the number of connected components of $\partial\mathrm{S}$.
\end{itemize}

%We denote by  the set of decorating points in $\surfO$.

The \emph{mapping class group} $\MCG(\surfO)$ of $\surfO$ is the group of isotopy classes of
orientation-preserving homeomorphisms of $\mathrm{S}$,
where all homeomorphisms and isotopies are required to:
i) fix $\partial\mathrm{S}$ pointwise;
ii) fix the decoration set $\Tri$ setwise.
%We are interested in two types of twists in $\MCG(\surfO)$.
%Note that the mapping class group $\MCG(\surf)$ of $\surf$ require only the first condition and thus there is a canonical map $F_*\colon\MCG(\surfo)\twoheadrightarrow\MCG(\surf)$ induced by the forgetful map $F$.

A \emph{closed arc} in $\surfO$ is (the homotopy class of) a continuous function $\eta:[0,1]\to\mathrm{S}$ such that
\begin{enumerate}
	\item both $\eta(0)$ and $\eta(1)$ are in $\Tri$ with $\eta(0)\neq\eta(1)$, and
	\item for any $0<t<1$, $\eta(t)\notin \Tri$.
\end{enumerate}
A closed arc $\eta$ is called \emph{simple} if $\eta(t_1)\neq\eta(t_2)$ for any $t_1\neq t_2$. Denote by $\cA(\surfO)$ the set of simple closed arcs in $\surfO$.
%whose interior lies in $\mathrm{S}-\Tri$ and whose endpoints are different decorating points in $\Tri$.

\begin{definition}[Braid twists]\label{def:bt}
For any simple closed arc $\eta\in\cA(\surfO)$, the \emph{braid twist}
$\Bt{\eta}\in\MCG(\surfO)$ along $\eta$ is shown in Figure~\ref{fig:Braid twist}.
%Throughout the paper, we will also use $\eta$ to denote its negative braid twist $\iv{\Bt{\eta}}$ when there is no confusion.
\begin{figure}[ht]\centering
\begin{tikzpicture}[scale=.2]
  \draw[very thick,NavyBlue](0,0)circle(6)node[above,black]{$_\eta$};
  %\draw(-120:5)node{+};
  \draw(-2,0)edge[red, very thick](2,0)  edge[cyan,very thick, dashed](-6,0);
  \draw(2,0)edge[cyan,very thick,dashed](6,0);
  \draw(-2,0)node[white] {$\bullet$} node[Emerald] {$\circ$};
  \draw(2,0)node[white] {$\bullet$} node[Emerald] {$\circ$};
  \draw(0:7.5)edge[very thick,->,>=latex](0:11);\draw(0:9)node[above]{$\Bt{\eta}$};
\end{tikzpicture}\;
%=======================================================
\begin{tikzpicture}[scale=.2]
  \draw[very thick, NavyBlue](0,0)circle(6)node[above,black]{$_\eta$};
  \draw[red, very thick](-2,0)to(2,0);
  \draw[cyan,very thick, dashed](2,0).. controls +(0:2) and +(0:2) ..(0,-2.5)
    .. controls +(180:1.5) and +(0:1.5) ..(-6,0);
  \draw[cyan,very thick,dashed](-2,0).. controls +(180:2) and +(180:2) ..(0,2.5)
    .. controls +(0:1.5) and +(180:1.5) ..(6,0);
  \draw(-2,0)node[white] {$\bullet$} node[Emerald] {$\circ$};
  \draw(2,0)node[white] {$\bullet$} node[Emerald] {$\circ$};
\end{tikzpicture}
\caption{The braid twist along a simple closed arc $\eta$}
\label{fig:Braid twist}
\end{figure}
The \emph{braid twist group} $\BT(\surfO)$ of a decorated surface $\surfO$ is the subgroup of $\MCG(\surfO)$ generated by the braid twists.
\end{definition}

Note that the braid twist $B_\eta$ does not depend on the orientation of $\eta$ in the sense that if we define a closed arc $\eta'$ by $\eta'(t)=\eta(1-t)$ for $t\in[0,1]$ then we have $B_\eta=B_{\eta'}$.
We have the following easy observation for the action of $\MCG(\surfO)$ on $\BT(\surfO)$ by conjugation (cf. \cite[Equation~(3.3)]{QQ}).

\begin{lemma}[Conjugation]\label{rmk:conj}
	For an element $\Psi$ in $\MCG(\surfO)$ and a closed arc $\eta$ in $\cA(\surfO)$, we have
	%the following relation in $\MCG(\surfO)$:
	%denote by $\eta^\Psi$ the action of $\Psi$ on $\eta$, for some arc $\eta$.
	%This is compatible with the conjugation notation in the sense that
	%the braid twist $\eta^\Psi$ equals the conjugation:
	\begin{gather}\label{eq:formulaB}
	B_{\Psi(\eta)}=(B_\eta)^{\iv{\Psi}}.
	\end{gather}
	This implies that the group $\BT(\surfO)$ is a normal subgroup of $\MCG(\surfO)$.
\end{lemma}

%For two simple closed arcs $\alpha,\beta$ in $\cA(\surfo)$, their intersection number is a half integer in $\tfrac{1}{2}\ZZ$ and defined as follows (following \cite{KS}):
%\[
%\Int(\alpha,\beta)=\tfrac{1}{2}\Int_{\Tri}(\alpha,\beta)+\Int_{\surf-\Tri}(\alpha,\beta),
%\]
%where
%\begin{gather*}
%\Int_{\surf-\Tri}(\alpha,\beta)=\min\{ |\alpha'\cap\beta'\cap (\surf-\Tri)| \ff \alpha'\sim\alpha,\beta'\sim\beta \},\\
%\Int_{\Tri}(\alpha,\beta)=\sum_{Z\in\Tri} |\{t\mid\alpha(t)=Z\}|\cdot|\{r\mid\beta(r)=Z\}|.
%\end{gather*}
%It is well-known that braid twists satisfy the following three types of relations.

It is easy to check the following relations on braid twists. %(see \cite{??}).

\begin{lemma}\label{lem:btrel}
For any closed arcs $a,b,c$ in $\cA(\surfO)$, the following hold in $\BT(\surfO)$:
\[\begin{array}{cl}
\Co(B_a,B_b)&\text{if $a$ and $b$ are disjoint;}\\
\Br(B_a,B_b)&\text{if $a$ and $b$ are disjoint except
	sharing a common endpoint;}\\
\Tr(B_a,B_b,B_c)&\text{if $a$, $b$ and $c$ are disjoint except sharing a common endpoint, and they}\\
&\text{are in clockwise order at that point (see the left picture of Figure~\ref{fig:ex0}).}
\end{array}\]
\end{lemma}

An \emph{L-arc} $\delta$ in $\surfO$ is (the homotopy class of) a continuous function $\delta:[0,1] \to\mathrm{S}$ such that
\begin{enumerate}
	\item $\delta(0)=\delta(1)$ is in $\Tri$,
	\item for any $0<t<1$, $\delta(t)\notin \Tri$, and
	\item $\delta$ is not homotopic to $\delta(0)$.
\end{enumerate}
An L-arc $\delta$ is called \emph{simple} if $\delta(t_1)\neq\delta(t_2)$ for any $t_1\neq t_2\in(0,1)$.
%Denote by $\cA(\surfO)$ the set of simple closed arcs in $\surfO$.

%whose interior lies in $\mathrm{S}-\Tri$, whose endpoints coincide at a decorating point $Z$ in $\Tri$ and which is not isotopic to $Z$.

\begin{definition}
For any simple L-arc $\delta$, the \emph{twist} $t_\delta$ along $\delta$ is an element in $\MCG(\surfO)$,
moving the point $Z=\delta(0)$ along $\delta$ to $\delta(1)$ as shown in Figure~\ref{fig:Tt}.
\begin{figure}[ht]\centering
	\begin{tikzpicture}[scale=.3]
	\draw[very thick, NavyBlue](0,0)circle(6);
	\draw[cyan,very thick, dashed](-6,0)to(-.7,0);
	
	\draw[very thick, NavyBlue] (0,0) circle (.7);
	\draw[red] (0,0) circle (3);
	\draw[red,->,>=stealth](3,0.1)to(3,-0.01);
	\draw(-3,0)node[white] {$\bullet$} node{$\circ$} node[below]{$_Z\quad$}
	(3,0)node[right,red]{$\delta$};
	\draw(0:7.5)edge[very thick,->,>=latex](0:11);\draw(0:9)node[above]{$t_{\delta}$};
	\end{tikzpicture}\;
	%=======================================================
	\begin{tikzpicture}[scale=.3]
	\draw[very thick, NavyBlue](0,0)circle(6);
	\draw[cyan,very thick, dashed](-6,0).. controls +(60:9) and +(90:3) ..(5,0)
	.. controls +(-90:5) and +(-120:7) ..(-3,0)
	.. controls +(-60:4) and +(-90:2) ..(2,0)
	.. controls +(90:2) and +(120:3) ..(-.7,0);
	
	\draw[very thick, NavyBlue] (0,0) circle (.7);
	\draw[red] (0,0) circle (3);
	\draw[red,->,>=stealth](3,0.1)to(3,-0.01);
	\draw(-3,0)node[white] {$\bullet$} node{$\circ$} node[]{$_Z\quad$}
	(3,0)node[right,red]{$\delta$};
	\end{tikzpicture}
	\caption{The twist along a simple L-arc $\delta$}
	\label{fig:Tt}
\end{figure}
\end{definition}

Note that the twist $t_\delta$ along an L-arc $\delta$ depends on the orientation of $\delta$ in the way that $t_{\delta}=\underline{t_{\delta'}}$ where $\delta'$ is the L-arc given by $\delta'(t)=\delta(1-t)$ for $t\in[0,1]$.

\begin{notations}\label{nts:arc}
	To simplify the notation, we will denote by $\eta$ the braid twist $B_\eta$ and denote by $\delta$ the twist $t_\delta$.
\end{notations}

There are several equivalent definitions of
surface braid group (the most common one is via configuration space, cf. e.g. \cite{GJP}).
We will take the following one as it suits our purpose better.

\begin{definition}[Surface braid groups]
The surface braid group $\SBG(\surfO)$ of a decorated surface $\surfO$ is the subgroup of $\MCG(\surfO)$ generated by the twists of simple L-arcs and the braid twists.
\end{definition}

It follows directly from the definitions that the braid twist group $\BT(\surfO)$ is the subgroup of $\SBG(\surfO)$ generated by the braid twists. We shall use the following known presentation of $\SBG(\surfO)$ to obtain a presentation of $\BT(\surfO)$ later.
\begin{proposition}[\cite{BG}]\label{prop:sbg}
The group $\SBG(\surfO)$ admits the following presentation.
\begin{itemize}
	\item Generators: $\sigma_1,\cdots,\sigma_{\aleph-1},\delta_1,\cdots,\delta_{2g+b-1}$ (see Figure~\ref{fig:B's}).
	
	\item Relations: for $1\leq i,j\leq \aleph-1$ and $1\leq r,s\leq 2g+b-1$,
	
	$\begin{array}{lll}
	&\Co(\sigma_i,\sigma_j)&  \text{if $|i-j|>1$;}\\
	&\Br(\sigma_i,\sigma_j)&\text{if $|i-j|=1$;}\\
	& \Co(\sigma_i,\delta_r)& \text{if $i\neq 1$;}\\
	&\Co(\delta_r,\sigma_1\delta_r\sigma_1)&\\
	&\Co(\delta_r^{\iv{\sigma_1}},\delta_s)&\text{if $s<r$, with $s\notin\dgen$ or with $s\neq r-1$;}\\
	&\SCo(\sigma_1;\delta_{s+1},\delta_{s})
	&\text{if $s\in\dgen$.}
	\end{array}$
\end{itemize}
\end{proposition}

\begin{figure}[ht]\centering
	\begin{tikzpicture}[xscale=.3,yscale=.3]
	\draw[red](10,-7)to(14,-9)(24,-9)to(20,-9);\draw[dashed,red](14,-9)to(24,-9);
	\draw[red]plot [smooth,tension=.5] coordinates {(10,-7)(-.2,0)(1,1)(10,-7)};
	\draw[ultra thin,red,->-=1,>=stealth](-.2-.02,0+.02)to(-.2,0);
    \draw[ultra thin,red,->-=1,>=stealth](5-.1,.31)to(5-.1,0.3);
    \draw[ultra thin,red,->-=1,>=stealth](10-.27,0+.02)to(10-.27,0);
    \draw[ultra thin,red,-<-=1,>=stealth](16-.3+.02,0+.02)to(16-.3,0);
    \draw[ultra thin,red,-<-=1,>=stealth](30-.5+.015,0+.02)to(30-.5,0);

    \foreach \j in {28}{
    \draw[ultra thin,red,-<-=0,>=stealth](\j-.2-.15,0+.02)to(\j-.2-.15,0);}

	\draw[red]plot [smooth,tension=.5] coordinates {(10,-7)(5,0)(6.6,1)(10,-7)};
	\draw[red]plot [smooth,tension=.5] coordinates {(10,-7)(15,.5)(10,1.5)(10,-7)};
	\draw[red]plot [smooth,tension=.5] coordinates {(10,-7)(27,-1)(24,1.5)(10,-7)};
	\draw[red](10,-7)to(8+2+4-.5,0);
	\draw[red,dashed](8+2+4-.5,0).. controls +(60:1) and +(180:1) ..(17,5);
	\draw[red](10,-7).. controls +(45:5) and +(0:3) ..(17,5);
	\draw[red](10,-7).. controls +(35:5) and +(-100:2) ..(20+2+4-.3,0.3);
	\draw[red,dashed](20+2+4-.3,0.3).. controls +(80:1) and +(180:1) ..(28,5);
	\draw[red](10,-7).. controls +(8:27) and +(0:2) ..(28,5);
	
	\foreach \j in {1,2.5} {
		\draw[thick](8*\j-2+4+.5,0).. controls +(45:1) and +(135:1) ..(8*\j+2+4-.5,0);
		\draw[very thick](8*\j-2+4+.3,0.3).. controls +(-60:1) and +(-120:1) ..(8*\j+2+4-.3,0.3);
	}
	\draw[very thick]
	(0,5)to(32,5)(32,-11)(0,-11)to(32,-11)
	(0,5)to[bend right=90](0,-11);
	\draw[very thick,fill=gray!14] (32,-3) ellipse (1 and 8) 	node{\tiny{$\partial_{1}$}};
	\draw[very thick,fill=gray!14](1,0)circle(.5) node{\tiny{$\partial_{b}$}};
	\draw[very thick,fill=gray!14](6,0)circle(.5) node{\tiny{$\partial_{2}$}};
	\node at(3.5,0) {$\cdots$};
	\node at(18,0) {$\cdots$};
	\foreach \x/\y in {14/-9,10/-7,24/-9,20/-9}
	{\draw(\x,\y)node[white]{{$\bullet$}}node{{$\circ$}};}
	\draw (10,-7)node[below]{\small{$Z_{1}$}}
	(14,-9)node[above]{\small{$Z_2$}}
	(20,-9)node[above]{\small{$Z_{\aleph-1}$}}
	(24,-9)node[above]{\small{$Z_{\aleph}$}};
	\draw[red](-.5,2.2)node{$\delta_{2g+b-1}$} (6,2.2)node{$\delta_{2g+1}$}
	(12,2.8)node{$\delta_{2g}$} (19,2.8)node{$\delta_{2g-1}$}
	(23.5,2.4)node{$\delta_2$} (28,-4)node{$\delta_{1}$}
	(11.5,-8.5)node{$\sigma_1$} (22,-10)node{$\sigma_{\aleph-1}$};
	%\draw[very thick,fill=gray!14] (38,-3) ellipse (1 and 8)
	%	node{\tiny{$\partial_{1}$}};
	\end{tikzpicture}
	\caption{Generators for $\SBG(\surfO)$}
	\label{fig:B's}
\end{figure}

%=========================================================
%\subsection{Quivers with potential and mutations}
%=========================================================

%=========================================================
%\subsection{Spherical twists on 3-Calabi-Yau categories}
%=========================================================

%Denote by $S_i$ the simple $\Gamma$-module associated to a vertex $i$ of $\overline{Q}$.

% is the graded quiver whose arrows of degree 0 are the arrows of $Q$, whose arrows of degree -1 are the

%The Jacobian algebra $\mathcal{P}(Q,W)$ is defined to be the quotient of the completion path algebra by the closure of the ideal generated by the partial of $W$.

%Let $\Gamma=\Gamma(Q,W)$ be the corresponding \emph{Ginzburg dg algebra (of degree 3)} \cite{G}. Then the \emph{finite dimensional derived category} $\D_{fd}(\Gamma)$ of $\Gamma$, which is the full subcategory of the derived category of $\Gamma$ consisting of the finite dimensional total homology dg modules, is 3-Calabi-Yau \cite{Keller}. The $0$-th cohomology $H^0(\Gamma)$ of $\Gamma$ is called the \emph{Jacobian algebra} of the quiver with potential $(Q,W)$. Its module category consists of representations of the quiver $Q$ satisfying certain conditions and can embed into the heart of the canonical $t$-structure of $\D_{fd}(\Gamma)$.

%=========================================================
\subsection{Decorated marked surfaces and quivers with potential}\label{sec:MS}
%=========================================================
A \emph{marked surface} $\surf$ without punctures in the sense of \cite{FST} is a pair $(\mathrm{S},\mathbf{M})$ of a compact connected oriented surface $\mathrm{S}$ with non-empty boundary and a finite set $\mathbf{M}$ of marked points on the boundary $\partial\surf$
satisfying that each connected component of $\partial\mathrm{S}$ contains at least one marked point.
%As mentioned in the previous subsection, $\mathrm{S}$ is determined by its genus $g$ and the number $b$ of its boundary components.
%Up to homeomorphism, $\surf$ is determined by the following data
%\begin{itemize}
%\item the genus $g$;
%\item the number $|\partial\surf|$ of boundary components;
%\item the integer partition of $|\M|$ into $|\partial\surf|$ parts describing the number of marked points on each boundary component.
%\end{itemize}

An (open) \emph{arc} in $\surf$ is a curve (up to isotopy) on $\mathrm{S}$ whose interior lies in $\mathrm{S}-\partial\mathrm{S}$, whose endpoints are marked points in $\mathbf{M}$, and which is neither homotopic to a boundary segment nor to a point. A triangulation of $\surf$
is a maximal collection of simple arcs in $\surf$ which do not cross each other in the interior of $\surf$.
Any triangulation of $\surf$ consists of $n=6g+3b+|\mathbf{M}|-6$ arcs and divides $\surf$ into
$\frac{2n+|\mathbf{M}|}{3}$ triangles (cf. e.g.  \cite[Proposition~2.10]{FST} and \cite[Equation~(2.9)]{QQ}).

Putting the notion of decorated surface and marked surface together,
we have the following.
\begin{definition}[{\cite[Definition~3.1]{QQ}}]
A \emph{decorated marked surface} $\surfo$ is a marked surface $\surf$
with a set $\Tri$ of $\aleph=\frac{2n+|\mathbf{M}|}{3}$ decorating points in
the interior of $\surf$.
A triangulation $\TT$ of $\surfo$ is induced by a triangulation of $\surf$
such that each triangle contains exactly one decorating point.
\end{definition}

For a decorated marked surface $\surfo=(\mathrm{S},\M,\Tri)$, when forgetting the marked points,
it becomes a decorated surface $\surfO$ (where $\aleph=\frac{2n+|\mathbf{M}|}{3}$).
In this case, let
\[\BT(\surfo)\colon=\BT(\surfO) \ \text{and }\MCG(\surfo)\colon=\MCG(\surfO).\]
On the other hand,
one can turn a decorated surface $\surfO$ into a decorated marked surface $\surfo$,
by adding certain number of marked points, when
%This is because $\surfO$ is a decorated marked surface if and only if $\aleph=4g+2b+|\M|-4$ for some $|\M|\geq b$. The latter is equivalent to that
$\aleph\geq 4g+3b-4$.
%This is because the number of decorating points of $\surfO$ might be too small.

Let $\TT$ be a triangulation of $\surf$. For any arc $\gamma\in\TT$, the dual $\gamma^\ast=\gamma^\ast_{\TT}$ of $\gamma$ with respect to $\TT$ is the unique closed arc in $\cA(\surfo)$ which intersects $\gamma$ once and does not intersect any other arcs in $\TT$. Let $\TT^\ast$ be the dual of $\TT$, that is, $\TT^\ast$ consists of the duals of the arcs in $\TT$. Let $\BT(\TT)$ be the subgroup of $\BT(\surfo)$ generated by the braid twists $B_{\gamma^\ast}$, $\gamma^\ast\in\TT^\ast$.

%Then the vertices of $Q_\TT$ are also indexed by the closed arcs in $\TT^\ast$.

\begin{lemma}[{\cite[Proposition~4.13]{QQ},\cite[Proposition~2.3]{QZ2}}]\label{lem:4.2}
	When specifying a triangulation $\TT$ of $\surfo$,
	the group $\BT(\TT)$ equals $\BT(\surfo)$.
\end{lemma}

We recall the notion of quiver with potential from \cite{DWZ}. A \emph{quiver} is a quadruple $Q=(Q_0,Q_1,s,t)$, where $Q_0$ is the set of \emph{vertices} of $Q$, $Q_1$ is the set of \emph{arrows} of $Q$, and $s,t:Q_1\to Q_0$ send an arrow $a$ of $Q$ to its starting vertex $s(a)$ and its ending vertex $t(a)$, respectively. The notation $a:i\to j$ denotes that $a$ is an arrow of $Q$ with $i=s(a)$ and $j=t(a)$. A \emph{path of length $d$} is a sequence $a_1a_2\cdots a_d$ with $t(a_i)=s(a_{i+1})$. A path $a_1a_2\cdots a_d$ is called a \emph{cycle} if $t(a_d)=s(a_1)$. A cycle of length $d$ is called an $d$-cycle.
%Let $Q$ be a finite quiver without loops and
Let $\k$ an algebraically closed field. %For convenience, we always assume that $Q$ does not have loops or 2-cycles and assume that $\k$ is algebraically closed.
%without loops or 2-cycles and $\k$ an algebraically closed field.
%A cycle in $Q$ is a path which starts and ends at the same vertex.
A \emph{potential} $W$ is a linear combination of finite cycles in $Q$ up to cyclic permutation. We call the pair $(Q,W)$ a quiver with potential. % (see \cite{DWZ}).
%A quiver with potential $(Q,W)$ is called \emph{trivial} if $W$ is a linear combination of 2-cycles in $Q$.
%A right equivalence from a quiver with potential $(Q,W)$ to another $(Q',W')$ is

Let $i$ be a vertex of $Q$ such that there are no 2-cycles through it. The \emph{pre-mutation} $\widetilde{\mu}_i(Q,W)$ of $(Q,W)$ at $i$ is a new quiver with potential $(Q',W')$ constructed as follows.
\begin{itemize}
	\item The new quiver $Q'$ is obtained from $Q$ by
	\begin{enumerate}
		\item[Step 1] For any pair of arrows $a:j\to i$ and $b:i\to l$, add a new arrow $[ab]:j\to l$.
		\item[Step 2] Reverse each arrow $a$ starting or ending at $i$, i.e. replace $a$ with a new arrow $a^\star$ with $s(a^\star)=t(a)$ and $t(a^\star)=s(a)$.
	\end{enumerate}
	\item The new potential $W'=[W]+\sum b^\star a^\star [ab]$, where $[W]$ is obtained from $W$ by replacing each composition $ab$ of arrows $a$ and $b$ with $ab$ going through $i$ by $[ab]$, and the sum runs over all composition $ab$ going through $i$.
\end{itemize}
%Note that there might be 2-cycles in the expression of $W'$.
We assume that any arrow in a 2-cycle in $W'$ does not occur in any other item in $W'$. Then the \emph{mutation} of $(Q,W)$ at the vertex $i$, denoted by $\mu_i(Q,W)$, is obtained from $(Q',W')$ by removing all 2-cycles from $W'$ and removing all arrows in these 2-cycles from $Q'$. We remark that in the general case (i.e. without the above assumption) the mutation of quiver with potential need to be defined via the notion of right equivalence. However, in the marked surface with punctures case, this assumption always holds (cf. Case 1 in the proof of \cite[Theorem~30]{LF}), which makes the definition simpler. %The above assumption does not hold in general. However it holds in our setting.

Let $\TT$ be a triangulation of $\surfo$.
There is an associated quiver with potential $(Q_\TT,W_\TT)$ \cite{FST,LF}, constructed as follows.
%(see, e.g. \cite{FST} and \cite{LF} for the precise definition):
\begin{itemize}
\item The vertices of $Q_\TT$ are (indexed by) the arcs in $\TT$.
\item There is an arrow from $i$ to $j$ whenever there is a triangle in $\TT$ having $i$ and $j$ as edges with $j$ following $i$ in the clockwise orientation (which is induced by the orientation of $\surf$). For instance, the quiver for a triangle is shown in
%for each triangle in $\TT$, there are three arrows between the corresponding vertices as shown in
Figure~\ref{fig:quiver}.
\item Each triangle in $\TT$ yields a unique 3-cycle up to cyclic permutation. The potential $W_\TT$ is the sum of all such 3-cycles.
\end{itemize}
\begin{figure}[h]\centering
  \begin{tikzpicture}[scale=.7]
  \foreach \j in {1,...,3}  { \draw (120*\j-30:2) coordinate (v\j);}
    \path (v1)--(v2) coordinate[pos=0.5] (x3)
              --(v3) coordinate[pos=0.5] (x1)
              --(v1) coordinate[pos=0.5] (x2);
    \foreach \j in {1,...,3}{\draw (x\j) node[red] (x\j){$\bullet$};}
    \draw[->,>=stealth,red,thick] (x1) to (x3);
    \draw[->,>=stealth,red,thick] (x3) to (x2);
    \draw[->,>=stealth,red,thick] (x2) to (x1);
    \draw[gray,thin] (v1)--(v2)--(v3)--cycle;
  \end{tikzpicture}
\caption{The (sub-)quiver (with potential) associated to a triangle}
\label{fig:quiver}
\end{figure}

\begin{definition}
Let $\gamma$ be an arc in a triangulation $\TT$ of $\surfo$. The arc $\gamma^\flat=\gamma^{\flat}(\TT)$ is obtained from $\gamma$ by clockwise moving its endpoints along the quadrilateral in $\TT$ whose diagonal is $\gamma$ (cf. Figure~\ref{fig:ex0}), to the next marked points. The \emph{backward flip} of a triangulation $\TT$ of $\surfo$ at $\gamma\in\TT$ is the triangulation $\mu_\gamma^\flat(\TT)$ obtained from $\TT$ by replacing the arc $\gamma$ with $\gamma^\flat$.
Similarly, we have the notion of forward flip,
which is the inverse of backward flip, i.e. $\TT=\mu_{\gamma^\flat}^\sharp\left(\mu_\gamma^\flat(\TT)\right)$.
\end{definition}

\begin{figure}[ht]\centering
	\begin{tikzpicture}[scale=.35,rotate=-120]
	\foreach \j in {1,...,6}{\draw[very thick](60*\j+30:6)to(60*\j-30:6);}
	\foreach \j in {1,...,3}{\draw[NavyBlue,thin](120*\j-30:6)node{$\bullet$}to(120*\j+90:6)
		(120*\j-90:6)node{$\bullet$};}
	\foreach \j in {1,...,3}{  \draw[red,thick](30+120*\j:4)to(0,0);}
	\foreach \j in {1,...,3}{  \draw(30-120*\j:4)node[white]{$\bullet$}node[Emerald]{$\circ$};}
	\draw(30-120*1+14:2)node[red]{\text{\footnotesize{$a$}}};
	\draw(30-120*2+14:2)node[red]{\text{\footnotesize{$b$}}};
	\draw(30-120*3+14:2)node[red]{\text{\footnotesize{$c$}}};
	\draw(0,0)node[white]{$\bullet$}node[Emerald]{$\circ$};
	\end{tikzpicture}
	\qquad
	\begin{tikzpicture}[xscale=-.35,yscale=.35]
	\foreach \j in {1,...,6}{\draw[very thick](60*\j+30:6)to(60*\j-30:6);}
	\foreach \j in {1,...,2}{\draw[NavyBlue,thin](120*\j-30:6)node{$\bullet$}
		to(120*\j+90:6)(120*\j-90:6)node{$\bullet$};}
	\foreach \j in {2,3}{  \draw[red,thick](30+120*\j:4)to(0,0);}
	\draw[red,thick](30+120:4)to[bend left](30+120*3:4);
	\foreach \j in {1,3}{  \draw(30-120*\j:4)node[white]{$\bullet$}node[Emerald]{$\circ$};}
	\draw[](30-120+14:2)node[red]{\text{\footnotesize{$c$}}};
	\draw[](90:3.5)node[red]{\text{\footnotesize{$d$}}};
	\draw[](30-120*3-14:1.5)node[red]{\text{\footnotesize{$a$}}};
	\draw[NavyBlue,thin](30:6).. controls +(-120:2) and +(-60:2) ..(30:2)
	.. controls +(120:3) and +(40:2) ..(-150:6);
	\foreach \j in {1,...,3}{  \draw(30+120*\j:4)node[white]{$\bullet$}node[Emerald]{$\circ$};}
	\draw(0,0)node[white]{$\bullet$}node[red]{$\circ$};
	\end{tikzpicture}
	\caption{A backward flip (where $d=a^b=b^{\iv{a}}$)}
	\label{fig:ex0}
\end{figure}

Flip of triangulations is compatible with mutation of quivers with potential in the following sense.

\begin{proposition}[{\cite[Theorem~30]{LF}}]\label{prop:LF}
Let $\gamma$ be an arc in a triangulation $\TT$ of $\surfo$. Then
the quivers with potential $(Q_{\mu_\gamma^\flat(\TT)},W_{\mu_\gamma^\flat(\TT)})$ and $(Q_{\mu_\gamma^\sharp(\TT)},W_{\mu_\gamma^\sharp(\TT)})$ coincide with the quiver with potential $\mu_\gamma(Q_\TT,W_\TT)$.
\end{proposition}

%Due to \cite{FST,LF},  in the sense of \cite{DWZ}.

%For simplicity, we define mutation of quivers with potential via (backward) flip of triangulations, which is compatible with the original one.

%\begin{definition}
%Let $(Q_\TT,W_\TT)$ be the quiver with potential associated to a triangulation $\TT$ of $\surfo$ and $\gamma\in\TT$ a vertex of $Q_\TT$. The mutation $\mu_\gamma(Q_\TT,W_\TT)$ is defined to be $(Q_{\TT'},W_{\TT'}$, where $\TT'$ is the backward flip of $\TT$ w.r.t. $\gamma$ (cf. \cite{FST, LF}).
%\end{definition}

%=========================================================
\subsection{Spherical twists on 3-Calabi-Yau categories}\label{sec:CY}
%=========================================================

Let $(Q,W)$ be a quiver with potential. For an arrow $a$ of $Q$ and a cycle $a_1a_2\cdots a_s$ in $Q$, define $\partial_a(a_1a_2\cdots a_s)=\sum_{a_i=a}a_{i+1}\cdots a_sa_1\cdots a_{i-1}$. This extends linearly to $\partial_a W$.

%The graded quiver $\overline{Q}$ associated to $Q$ has the same vertex set as $Q$ and

\begin{definition}
The \emph{complete Ginzburg dg algebra} $\Gamma=\Gamma(Q,W)$ is constructed as follows \cite{G}.
Let $\widetilde{Q}$ be the graded quiver with the same vertices as $Q$ and whose arrows are
\begin{itemize}
	\item the arrows of $Q$ (with degree 0),
	\item an arrow $a^\ast: j\to i$ of degree -1 for each arrow $a:i\to j$ of $Q$,
	\item a loop $t_i: i\to i$ of degree -2 for each vertex $i$ of $Q$.
\end{itemize}
The underlying graded algebra of $\Gamma$ is the completion of the graded path algebra $\k\widetilde{Q}$ and the differential $d$ of $\Gamma$
% where $\overline{Q}=Q\cup\{a^\ast:j\to i\mid a:i\to j\text{ in }Q\}\cup\{t_i:i\to i \mid i\text{ a vertex in }Q\}$ with the arrows in $Q$ in degree 0, the arrows $a^\ast$ in degree -1 and the arrows $t_i$ in degree -2. The differential
is linearly determined by the formulas $d(a^\ast)=\partial_a W$ and $d(\sum_{i:\text{ a vertex of }Q} t_i)=\sum_{a:\text{ an arrow of }Q} (aa^\ast-a^\ast a)$.
\end{definition}

Denote by $\D(\Gamma)$ the derived category of $\Gamma$. The \emph{finite dimensional derived category} $\D_{fd}(\Gamma)$ is the full subcategory of $\D(\Gamma)$ consisting of those dg $\Gamma$-modules whose homology is of finite total dimension.
%Let $(Q',W')$ be the mutation of $(Q,W)$ at a vertex $i$. Then by the main result in \cite{KY}, there is an equivalence
The category $\D_{fd}(\Gamma)$ is a 3-Calabi-Yau triangulated category \cite{Ke} in the sense that for any objects $X,Y$ of $\D_{fd}(\Gamma)$, there is a functorial isomorphism
\[\Hom(X,Y)\cong D\Hom(Y,X[3])\]
where $D=\Hom_\k(-,\k)$.

\begin{definition}[spherical twists \cite{ST}]
An object $S$ of $\D_{fd}(\Gamma)$ is called \emph{(3-)spherical} provided that
\[\Hom_{\D_{fd}(\Gamma)}(S,S[r])\cong\begin{cases}
\k&\text{if $r=0$ or $3$;}\\0&\text{otherwise.}
\end{cases} \]
For a spherical object $S$, there is an induced auto-equivalence $\phi_S$,
called spherical twist,
 of $\D_{fd}(\Gamma)$ defined by the triangle \cite{ST}:
%Recall from \cite{ST} that any spherical object $S$ defines a \emph{twist functor} $\phi_S\in\Aut(\D_{fd}(\Gamma))$, such that
\[
    \phi_S(X)[-1]\to S\otimes\Hom^\bullet(S,X)\to X\to \phi_S(X)
\]
for $X\in\D_{fd}(\Gamma)$.
% is an  induced by a spherical object $S$ such that
%with inverse
%\[
%\phi_S^{-1}(X)=\Cone\left(X\to S\otimes\Hom^\bullet(X,S)^\vee \right)[-1]\]
\end{definition}

For any vertex $i$ of $Q$, the corresponding simple b$\Gamma$-module $S_i$ is a spherical object in $\D_{fd}(\Gamma)$ (cf. \cite[Lemma~2.15]{KY}). Denote by $\Sim\zero$ the set of simple $\Gamma$-modules. The spherical twist group $\ST(\Gamma)$ of $\D_{fd}(\Gamma)$ is defined to be the subgroup of the auto-equivalence group $\Aut\D_{fd}(\Gamma)$ generated by $\phi_S$, $S\in \Sim\zero$.   Set $\Sph(\Gamma)=\ST(\Gamma)\cdot \Sim\zero$. Then we have that $\ST(\Gamma)$ is also generated by all $\phi_X$, $X\in\Sph(\Gamma)$.

%Fix an algebraically closed field $\k$. Denote by $\Gamma=\Gamma(Q,W)$ the \emph{Ginzburg dg $\k$-algebra (of degree 3)} \cite{G} associated to a quiver with potential $(Q,W)$ and $\D_{fd}(\Gamma)$ the finite-dimensional derived category of $\Gamma$.
%Then $\D_{fd}(\Gamma)$ is a 3-Calabi-Yau, Hom-finite, Krull-Schmidt, $\k$-linear triangulated category.
%We also know that $\D_{fd}(\Gamma)$ admits a canonical heart $\zero$ generated by the simple $\Gamma$-modules $S_e$ corresponding to the vertices $e$ of $Q$. Each $S_e$ is

%Known results.

For the quiver with potential $(Q_\TT,W_\TT)$ associated to a triangulation $\TT$ of $\surfo$, denote the corresponding Ginzburg dg algebra by $\Gamma_\TT$.

\begin{theorem}[{\cite[Theorem~1]{QQ}}]\label{thmQQ}
	There is an isomorphism
	\begin{gather}\label{1}
	\iota\colon\BT(\TT)\to\ST(\Gamma_\TT),
	\end{gather}
	sending the standard generators (i.e. the braid twists of closed arcs in $\TT^*$)
	to the standard generators (i.e. the spherical twists of the simple $\Gamma_\TT$-modules).
\end{theorem}

Let $\TT'$ be a (forward or backward) flip of $\TT$. Since $(Q_{\TT'},W_{\TT'})$ is the mutation of $(Q_\TT,W_\TT)$ at some vertex (Proposition~\ref{prop:LF}), by the main theorem of \cite{KY}, there is a triangle equivalence between $\D(\Gamma_{\TT'})$ and $\D(\Gamma_\TT)$, which restricts to a triangle equivalence $\D_{fd}(\Gamma_{\TT'})$ and $\D_{fd}(\Gamma_{\TT})$. Then %by \cite[Proposition]{FST},
%Due to \cite{LF,KY},
the category $\D_{fd}(\Gamma_\TT)$ is independent of the chosen triangulation $\TT$ up to triangle equivalence. Hence one can use $\D(\surf)$ to denote $\D_{fd}(\Gamma_\TT)$. %See \cite{QQ} and the references there for more details.

%=========================================================
\section{An alternative presentation of surface braid group}\label{app:pf sbg}
%=========================================================

This section devotes to give an alternative (positive) presentation of $\SBG(\surfO)$,
which is derived from the presentation in Proposition~\ref{prop:sbg} and will be used in the next section.
Define $\varepsilon_r$, $1\leq r\leq 2g+b-1$, recursively by
%$\varepsilon_{1}=\delta_{1}$, $\varepsilon_{2}=\delta_{2}$ and for $r\geq 3$,
\[\varepsilon_r=\begin{cases}
\delta_r\varepsilon_{r-1}&\text{if $r\notin \dgene$,}\\
\delta_r\varepsilon_{r-2}&\text{if $r\in \dgene$,}
\end{cases} \]
where for convenience, $\varepsilon_0$ is taken to be the identity. Conversely, we have %$\delta_1=\varepsilon_1$, $\delta_2=\varepsilon_2$ and for $r\geq3$,
\[\delta_r=\begin{cases}
\varepsilon_r\iv{\varepsilon_{r-1}}&\text{if $r\notin \dgene$;}\\
\varepsilon_r\iv{\varepsilon_{r-2}}&\text{if $r\in \dgene$.}
\end{cases} \]
Clearly, $\sigma_i$, $1\leq i\leq \aleph-1$, and $\varepsilon_r$, $1\leq r\leq 2g+b-1$, form new generators for $\SBG(\surfO)$, which are illustrated in Figure~\ref{fig:QZ's}, where, to make it reader friendly, we use the following notation:
$$\xi_{-r}=\varepsilon_{2r-1},\,
\xi_r=\varepsilon_{2r},\,
\zeta_l=\varepsilon_{2g+l},\quad
\text{for $1\leq r\leq g$ and $1\leq l\leq b-1$}$$

\begin{figure}[ht]\centering
	\begin{tikzpicture}[xscale=.28,yscale=.28]
	% sigma
	\draw[purple](10,-7)to(14,-9)(24,-9)to(20,-9);\draw[dashed,purple](14,-9)to(24,-9);
	\draw[purple](12,-8.8)node{\tiny{$\sigma_1$}} (22,-9.8)node{\tiny{$\sigma_{\aleph-1}$}};
	%\zeta
	\draw[ForestGreen]plot [smooth,tension=.8] coordinates {(10,-7)(-.2,2)(30,3)(28,-4)(10,-7)};
	\draw[ForestGreen]plot [smooth,tension=.8] coordinates {(10,-7)(5.5,1.5)(29,2.5)(26,-3)(10,-7)};
	\draw[ForestGreen](3.1,1.8)node{\tiny{$\zeta_{b-1}\cdots\zeta_1$}};
	%xi plus
	\draw[red]plot [smooth,tension=.8] coordinates {(10,-7)(24,-3)(29,1.5)(11,1.5)(10,-7)};
	\draw[red]plot [smooth,tension=.5] coordinates {(10,-7)(27,-1)(24,1.5)(10,-7)};
	\draw[red](9.6,-2)node{\tiny{$\xi_g$}} (20.5,-1)node{\tiny{$\xi_1$}};
	%xi minus
	\draw[red](10,-7)to(8+2+4-.5,0);
	\draw[red,dashed](8+2+4-.5,0).. controls +(60:2) and +(180:3) ..(20,5);
	\draw[red](10,-7).. controls +(-5:25) and +(0:25) ..(20,5);
	\draw[red](10,-7).. controls +(35:5) and +(-100:2) ..(20+2+4-.3,0.3);
	\draw[red,dashed](20+2+4-.3,0.3).. controls +(80:3) and +(180:1) ..(32,5);
	\draw[red](10,-7).. controls +(-10:30) and +(0:7) ..(32,5);
	\draw[red](11.5,-2)node{\tiny{$\xi_{-g}$}} (23,-1)node{\tiny{$\xi_{-1}$}};
	% surface
	\foreach \j in {1,2.5} {
		\draw[thick](8*\j-2+4+.5,0).. controls +(45:1) and +(135:1) ..(8*\j+2+4-.5,0);
		\draw[very thick](8*\j-2+4+.3,0.3).. controls +(-60:1) and +(-120:1) ..(8*\j+2+4-.3,0.3);
	}
	\draw[very thick] (0,5)to(38,5)(38,-11)(0,-11)to(38,-11) (0,5)to[bend right=90](0,-11);
	\draw[very thick,fill=gray!14] (38,-3) ellipse (1 and 8)
	node{\tiny{$\partial_{1}$}};
	\draw[very thick,fill=gray!14](6,0)circle(.7)node{\tiny{$\partial_2$}};
	\draw[very thick,fill=gray!14](1,0)circle(.7)node{\tiny{$\partial_{b}$}};
	\node at(3.5,0) {$\cdots$};
	\node at(18,0) {$\cdots$};
	\foreach \x/\y in {14/-9,10/-7,24/-9,20/-9}
	{\draw(\x,\y)node[white]{{$\bullet$}}node{$\circ$};}
	\draw (10,-7)node[below]{$Z_{1}$}
	(14,-9)node[below]{$Z_2$}
	(24,-9)node[below]{$Z_{\aleph}$};

    \foreach \j in {34.45,35.9}{
    \draw[ultra thin,red,->-=1,>=stealth](\j,0)to(\j,0.001);}

    \draw[ultra thin,red,->-=1,>=stealth](27.7,0.1)to(27.7,0.11);
    \draw[ultra thin,red,->-=1,>=stealth](29.77,0.1)to(29.77+.01,0.11);
    \draw[ultra thin,ForestGreen,->-=1,>=stealth](30.64,0)to(30.64+.001,0.001);
    \draw[ultra thin,ForestGreen,->-=1,>=stealth](33+.11,0)to(33+.11+.001,0.002);

	\end{tikzpicture}
	\caption{Alternative generators for $\SBG(\surfO)$}
	\label{fig:QZ's}
\end{figure}

\begin{proposition}\label{lem:sbg}
The group $\SBG(\surfO)$ admits the following presentation.
\begin{itemize}
	\item Generators: $\sigma_i$, $1\leq i\leq \aleph-1$, $\varepsilon_r$, $1\leq r\leq 2g+b-1$.
	\item Relations: for $1\leq i,j\leq \aleph-2$ and $1\leq r,s\leq 2g+b-1$,
    \begin{align}
	&\Co(\sigma_i,\sigma_j)\label{ap:01}&&  \text{if $|i-j|>1$;}\\
	&\Br(\sigma_i,\sigma_{j})\label{ap:02}&&\text{if $|i-j|=1$}.\\
	&\Co(\sigma_i,\varepsilon_r)\label{ap:03}&& \text{if $i\neq 1$;}\\
	&\Co(\varepsilon_r,\sigma_1\varepsilon_r\sigma_1)\label{ap:04}&&\\
	&\Co(\varepsilon_s,\sigma_1\varepsilon_r\sigma_1)\label{ap:05}&&\text{if $s<r$ with $s\notin \dgen$;}\\
	&\SCo(\sigma_1;\varepsilon_r,\varepsilon_{s})\label{ap:06}
	&&\text{if $s<r$ with $s\in\dgen$.}
	\end{align}
\end{itemize}
\end{proposition}
%\todo[inline,color=green!40]{A remark: there are several alternative relations for $\SCoi(\sigma_1;\delta_r,\delta_s)$: $\Co(\delta_r,\sigma_1\delta_r\delta_s\sigma_1)$,  $\Co(\delta_s,\sigma_1\delta_r\delta_s\sigma_1)$, $\Co(\delta_s,\sigma_1\delta_r\iv{\sigma_1})$ and $\Co(\delta_r,\iv{\sigma_1}\delta_s\sigma_1)$}
\begin{proof}
We need to prove that the relations in Proposition~\ref{prop:sbg} are equivalent to those
in Proposition~\ref{lem:sbg}.
The relations common to both presentations are
\[\begin{array}{lll}
\Co(\sigma_i,\sigma_j)&  \text{if $|i-j|>1$;}\\
\Br(\sigma_i,\sigma_{j})&\text{if $|i-j|=1$}.
\end{array}\]
By the construction of $\varepsilon_r$, it is easy to see the equivalence between the relations $\Co(\sigma_i,\delta_r)$ and $\Co(\sigma_i,\varepsilon_r)$ for any $i\neq 1$.% and vice versa.

Now we prove that the relations in Proposition~\ref{prop:sbg} imply the other relations in Proposition~\ref{lem:sbg}. First, we show a useful relation
\[\begin{array}{ll}
\Co(\delta_r^{\iv{\sigma_1}},\varepsilon_s)& \text{if $s<r$ with $s\notin \dgen$ or with $s\neq r-1$,}
\end{array}\]
which will be used for many times. Indeed, by construction, we have $\varepsilon_s=\delta_s\delta_{r_1}\cdots\delta_{r_m}$ with $s> r_1> r_2>\cdots > r_m$. So each $r_i<r-1$. Then we have the following relations from Proposition~\ref{prop:sbg}: $\Co(\delta_r^{\iv{\sigma_1}},\delta_s)$ and $\Co(\delta_r^{\iv{\sigma_1}},\delta_{r_i})$, $1\leq i\leq m$, which imply the required relation.

%for , which follows inductively from $\Co(\delta_r^{\iv{\sigma_1}},\delta_s)$.

To show $\Co(\varepsilon_r,\sigma_1\varepsilon_r\sigma_1)$, use induction on $r$, starting with the trivial case $r=0$.
%$r=1$ and $r=2$, where $\varepsilon_r=\delta_r$.
Assume $\Co(\varepsilon_r,\sigma_1\varepsilon_r\sigma_1)$ holds for any $r<t$ with some $t>0$. Consider the case $r=t$. By construction, $\varepsilon_t=\delta_t\varepsilon_{t-x}$, where $x=1$ if $t\notin\dgene$ or $x=2$ if $t\in\dgene$. So $t-x\notin\dgen$, which implies that we have the useful relation  $\Co(\delta_t^{\iv{\sigma_1}},\delta_{t-x})$. Hence we have
\[\begin{array}{rl}
\varepsilon_t\sigma_1\varepsilon_t\sigma_1
=&\delta_t\varepsilon_{t-x}\sigma_1\delta_t\varepsilon_{t-x}\sigma_1\\
=
&\delta_t\sigma_1\delta_t\iv{\sigma_1}\varepsilon_{t-x}\sigma_1\varepsilon_{t-x}\sigma_1\\
=
&\delta_t\sigma_1\delta_t\varepsilon_{t-x}\sigma_1\varepsilon_{t-x}\\
=
&\sigma_1\delta_t\sigma_1\delta_t\iv{\sigma_1}\varepsilon_{t-x}\sigma_1\varepsilon_{t-x}\\
=
&\sigma_1\delta_t\varepsilon_{t-x}\sigma_1\delta_t\varepsilon_{t-x}\\
=&\sigma_1\varepsilon_t\sigma_1\varepsilon_t
\end{array}\]
where the first and the last equalities are due to $\varepsilon_t=\delta_t\varepsilon_{t-x}$, the second and the fifth equalities use the relation $\Co(\delta_t^{\iv{\sigma_1}},\varepsilon_{t-x})$, the third equality uses the inductive assumption $\Co(\varepsilon_{t-x},\sigma_1\varepsilon_{t-x}\sigma_1)$ and the fourth equality uses the relation $\Co(\delta_t,\sigma_1\delta_t\sigma_1)$ from Proposition~\ref{prop:sbg}.

To show $\Co(\varepsilon_s,\sigma_1\varepsilon_r\sigma_1)$ for $s<r$ with $s\notin \dgen$, fix $s$ and use induction on $r$, starting with the extreme case $r=s$ which was proved above. Assume $\Co(\varepsilon_s,\sigma_1\varepsilon_r\sigma_1)$ holds for any $r<t$ with some $t>s$. Consider the case $r=t$. By construction, $\varepsilon_t=\delta_t\varepsilon_{t-x}$, where $x=1$ if $t\notin\dgene$ or $x=2$ otherwise. We claim $t-x\geq s$.%, which implies$\Co(\varepsilon_s,\sigma_1\varepsilon_{t-x}\sigma_1)$ by the assumption.
Indeed, if $t=s+1$ then $t\notin\dgene$. So $x=1$ and $t-x=s$; if $t\geq s+2$ then $t-x\geq t-2\geq s$. %In both cases, we have  $t-x\geq s$.
Then we have
\[\begin{array}{rl}
\varepsilon_s\sigma_1\varepsilon_t\sigma_1
=&\varepsilon_s\sigma_1\delta_t\varepsilon_{t-x}\sigma_1\\
=&\sigma_1\delta_t\iv{\sigma_1}\varepsilon_s\sigma_1\varepsilon_{t-x}\sigma_1\\
=&\sigma_1\delta_t\varepsilon_{t-x}\sigma_1\varepsilon_s\\
=&\sigma_1\varepsilon_t\sigma_1\varepsilon_s
\end{array}\]
where the first and the last equalities are due to $\varepsilon_t=\delta_t\varepsilon_{t-x}$, the second equality uses the useful relation $\Co(\delta_t^{\iv{\sigma_1}},\varepsilon_s)$ and the third equality uses the inductive assumption $\Co(\varepsilon_s,\sigma_1\varepsilon_{t-x}\sigma_1)$ by $t-x\geq s$.

To show $\SCo(\sigma_1;\varepsilon_r,\varepsilon_s)$ for $s<r$ with $s\in\dgen$, fix $s$ and use induction on $r$. Write $s=2i-1$ for some $1\leq i\leq g$. We first use induction on $i$ to prove the case $r=s+1$, starting with the case $i=1$, where by construction, $\varepsilon_s=\delta_1$ and $\varepsilon_r=\delta_2$, so $\SCo(\sigma_1;\varepsilon_r,\varepsilon_s)$ becomes $\SCo(\sigma_1;\delta_2,\delta_1)$ from Proposition~\ref{prop:sbg}. Assume  $\SCo(\sigma_1;\varepsilon_{2i},\varepsilon_{2i-1})$ holds for any $i<j$ with some $1<j\leq g$. Consider the case $i=j$. By construction, $\varepsilon_{2j-1}=\delta_{2j-1}\varepsilon_{2j-2}$ and $\varepsilon_{2j}=\delta_{2j}\varepsilon_{2j-2}$.  Then we have
\[\begin{array}{rl}
\varepsilon_{2j-1}\sigma_1\varepsilon_{2j}
=&\delta_{2j-1}\varepsilon_{2j-2}\sigma_1\delta_{2j}\varepsilon_{2j-2}\\
=&\delta_{2j-1}\sigma_1\delta_{2j}\iv{\sigma_1}\varepsilon_{2j-2}\sigma_1\varepsilon_{2j-2}\\
=&\sigma_1\delta_{2j}\sigma_1\delta_{2j-1}\varepsilon_{2j-2}\sigma_1\varepsilon_{2j-2}\\
=&\sigma_1\delta_{2j}\sigma_1\delta_{2j-1}\iv{\sigma_1}\varepsilon_{2j-2}\sigma_1\varepsilon_{2j-2}\sigma_1\\
=&\sigma_1\delta_{2j}\varepsilon_{2j-2}\sigma_1\delta_{2j-1}\varepsilon_{2j-2}\sigma_1\\
=&\sigma_1\varepsilon_{2j}\sigma_1\varepsilon_{2j-1}\sigma_1
\end{array}\]
where the first and the last equalities are due to both $\varepsilon_{2j-1}=\delta_{2j-1}\varepsilon_{2j-2}$ and $\varepsilon_{2j}=\delta_{2j}\varepsilon_{2j-2}$, the second equality uses the useful relation $\Co(\delta_{2j}^{\iv{\sigma_1}},\varepsilon_{2j-2})$, the third equality uses the relation $\SCo(\sigma_1;\delta_{2j},\delta_{2j-1})$ from Proposition~\ref{prop:sbg}, the fourth equality uses the relation $\Co(\varepsilon_{2j-2},\sigma_1\varepsilon_{2j-2}\sigma_1)$ proved above and the fifth equality uses the useful relation $\Co(\delta_{2j-1}^{\iv{\sigma_1}},\varepsilon_{2j-2})$.
Thus, the proof for the case $r=s+1$ is complete. Assume now $\SCo(\sigma_1;\varepsilon_r,\varepsilon_s)$ holds for any $r<t$ with some $t>s+1$. Consider the case $r=t$. By construction, $\varepsilon_t=\delta_t\varepsilon_{t-x}$, where $x=1$ if $t\notin\dgene$ or $x=2$ otherwise. If $t=s+2$ then $t\notin\dgene$, which implies $x=1$ and $t-x>s$; if $t>s+2$ then $t-x\geq t-2>s$. Hence we always have $t-x>s$. Then we have
\[\begin{array}{rl}
\varepsilon_{s}\sigma_1\varepsilon_{t}
=&\varepsilon_{s}\sigma_1\delta_t\varepsilon_{t-x}\\
=&\sigma_1\delta_t\iv{\sigma_1}\varepsilon_s\sigma_1\varepsilon_{t-x}\\
=&\sigma_1\delta_t\varepsilon_{t-x}\sigma_1\varepsilon_s\sigma_1\\
=&\sigma_1\varepsilon_t\sigma_1\varepsilon_s\sigma_1
\end{array}\]
where the first and the last equalities are due to $\varepsilon_t=\delta_t\varepsilon_{t-x}$, the second equality uses the useful relation $\Co(\delta_t^{\iv{\sigma_1}},\varepsilon_s)$ by $t>s+1$ and the third equality uses the inductive assumption $\SCo(\sigma_1;\varepsilon_{t-x},\varepsilon_s)$ by $t-x>s$.

Conversely, we prove that the relations in Proposition~\ref{lem:sbg} imply the other relations in Proposition~\ref{prop:sbg}. To show
$\Co(\delta_r,\sigma_1\delta_r\sigma_1)$, by construction, $\delta_r=\varepsilon_r\iv{\varepsilon_{r-x}}$, where $x=1$ if $r\notin\dgene$ or $x=2$ if $r\in\dgene$. In any case, we always have $r-x\notin\dgen$, which implies $\Co(\varepsilon_{r-x},\sigma_1\varepsilon_r\sigma_1)$ from \eqref{ap:05}. From the relation \eqref{ap:04}, we have $\Co(\varepsilon_{r-x},\sigma_1\varepsilon_{r-x}\sigma_1)$ and $\Co(\varepsilon_{r},\sigma_1\varepsilon_r\sigma_1)$, which imply $\iv{\varepsilon_{r-x}\sigma_1\varepsilon_{r-x}}\sigma_1=\sigma_1\iv{\varepsilon_{r-x}\sigma_1\varepsilon_{r-x}}$ and $\varepsilon_r\sigma_1\varepsilon_r\sigma_1=\sigma_1\varepsilon_r\sigma_1\varepsilon_r$, respectively. Combining these two equalities, we have \[\varepsilon_r\sigma_1\varepsilon_r\sigma_1\iv{\varepsilon_{r-x}\sigma_1\varepsilon_{r-x}}\sigma_1=\sigma_1\varepsilon_r\sigma_1\varepsilon_r\sigma_1\iv{\varepsilon_{r-x}\sigma_1\varepsilon_{r-x}}.\] Now using $\Co(\varepsilon_{r-x},\sigma_1\varepsilon_r\sigma_1)$, we have $\varepsilon_r\iv{\varepsilon_{r-x}}\sigma_1\varepsilon_r\iv{\varepsilon_{r-x}}\sigma_1=\sigma_1\varepsilon_r\iv{\varepsilon_{r-x}}\sigma_1\varepsilon_r\iv{\varepsilon_{r-x}}$. Then due to $\delta_r=\varepsilon_r\iv{\varepsilon_{r-x}}$, we have $\delta_r\sigma_1\delta_r\sigma_1=\sigma_1\delta_r\sigma_1\delta_r$, which is $\Co(\delta_r,\sigma_1\delta_r\sigma_1)$.

%Hence we have
%\[\begin{array}{rl}
%&\Co(\delta_r,\sigma_1\delta_r\sigma_1)\\
%\xLongleftrightarrow{}&\delta_r\sigma_1\delta_r\sigma_1=\sigma_1\delta_r\sigma_1\delta_r\\
%\xLongleftrightarrow{}&\varepsilon_r\iv{\varepsilon_{r-x}}\sigma_1\varepsilon_r\iv{\varepsilon_{r-x}}\sigma_1=\sigma_1\varepsilon_r\iv{\varepsilon_{r-x}}\sigma_1\varepsilon_r\iv{\varepsilon_{r-x}}\\
%\xLongleftrightarrow{}&\varepsilon_r\sigma_1\varepsilon_r\sigma_1\iv{\varepsilon_{r-x}\sigma_1\varepsilon_{r-x}}\sigma_1=\sigma_1\varepsilon_r\sigma_1\varepsilon_r\sigma_1\iv{\varepsilon_{r-x}\sigma_1\varepsilon_{r-x}}\ (\text{by }\Co(\varepsilon_{r-x},\sigma_1\varepsilon_r\sigma_1))\\
%\xLongleftrightarrow{}&\iv{\varepsilon_{r-x}\sigma_1\varepsilon_{r-x}}\sigma_1=\sigma_1\iv{\varepsilon_{r-x}\sigma_1\varepsilon_{r-x}}\ (\text{by }\Co(\varepsilon_{r},\sigma_1\varepsilon_r\sigma_1))\\
%\xLongleftrightarrow&\Co(\varepsilon_{r-x},\sigma_1\varepsilon_{r-x}\sigma_1)
%\end{array}\]

We also need to prove the useful relation
\[\begin{array}{ll}
\Co(\delta_r^{\iv{\sigma_1}},\varepsilon_s) & \text{if $s<r$ with $s\notin \dgen$ or with $s\neq r-1$}.
\end{array}\]
By construction, $\delta_r=\varepsilon_r\iv{\varepsilon_{r-x}}$, where $x=1$ if $r\notin\dgene$ or $x=2$ if $r\in\dgene$. There are two cases depending on whether $s$ is in $\dgen$. If $s\notin\dgen$, it is easy to see that $r-x\geq s$. Then we have
\[\begin{array}{rl}
\sigma_1\delta_r\iv{\sigma_1}\varepsilon_s\sigma_1
\xlongequal&\sigma_1\varepsilon_r\iv{\varepsilon_{r-x}}\iv{\sigma_1}\varepsilon_s\sigma_1\\
\xlongequal{}&\sigma_1\varepsilon_r\sigma_1\varepsilon_s\iv{\sigma_1}\iv{\varepsilon_{r-x}}\\
\xlongequal{}&\varepsilon_s\sigma_1\varepsilon_r\iv{\varepsilon_{r-x}}\\
=&\varepsilon_s\sigma_1\delta_r
\end{array}\]
where the first and the last equalities are due to $\delta_r=\varepsilon_r\iv{\varepsilon_{r-x}}$, the second equality uses $\Co(\varepsilon_s,\sigma_1\varepsilon_{r-x}\sigma_1)$ from \eqref{ap:04} if $r-x=s$ or from \eqref{ap:05} if $r-x>s$, and the third equality uses $\Co(\varepsilon_s,\sigma_1\varepsilon_{r}\sigma_1)$ from \eqref{ap:05}.
If $s\in\dgen$, then $s\neq r-1$. In this case, it is easy to see that $r-x>s$. So we have both $\SCo(\sigma_1;\varepsilon_{r-x},\varepsilon_s)$ and $\SCo(\sigma_1;\varepsilon_{r},\varepsilon_s)$ from \eqref{ap:06}. Then we have
\[\begin{array}{rl}
\sigma_1\delta_r\iv{\sigma_1}\varepsilon_s\sigma_1
\xlongequal&\sigma_1\varepsilon_r\iv{\varepsilon_{r-x}}\iv{\sigma_1}\varepsilon_s\sigma_1\\
\xlongequal{}&\sigma_1\varepsilon_r\sigma_1\varepsilon_s\sigma_1\iv{\varepsilon_{r-x}}\\
\xlongequal{}&\varepsilon_s\sigma_1\varepsilon_r\iv{\varepsilon_{r-x}}\\
=&\varepsilon_s\sigma_1\delta_r
\end{array}\]
where the first and the last equalities are due to $\delta_r=\varepsilon_r\iv{\varepsilon_{r-x}}$, the second equality and the third equality use $\SCo(\sigma_1;\varepsilon_{r-x},\varepsilon_s)$ and $\SCo(\sigma_1;\varepsilon_{r},\varepsilon_s)$, respectively.

To show $\Co(\delta_r^{\iv{\sigma_1}},\delta_s)$ for $s<r$, with $s\notin\dgen$ or with $s\neq r-1$, by construction, we have $\delta_s=\varepsilon_s\iv{\varepsilon_{s-x}}$, where $x=1$ or $2$. It follows that $s-x< r-1$. Hence we have $\Co(\delta_r^{\iv{\sigma_1}},\varepsilon_s)$ and $\Co(\delta_r^{\iv{\sigma_1}},\varepsilon_{s-x})$ , which imply $\Co(\delta_r^{\iv{\sigma_1}},\delta_s)$.

To show $\SCo(\sigma_1;\delta_{s+1},\delta_s)$ for $s\in\dgen$, by construction, we have $\delta_s=\varepsilon_s\iv{\varepsilon_{s-1}}$ and $\delta_{s+1}=\varepsilon_{s+1}\iv{\varepsilon_{s-1}}$. Hence we have
\[\begin{array}{rl}
\sigma_1\delta_{s+1}\sigma_1\delta_s\sigma_1
\xlongequal{}&\sigma_1\varepsilon_{s+1}\iv{\varepsilon_{s-1}}\sigma_1\varepsilon_s\iv{\varepsilon_{s-1}}\sigma_1\\
\xlongequal{}&\sigma_1\varepsilon_{s+1}\sigma_1\varepsilon_s\sigma_1\iv{\varepsilon_{s-1}}\iv{\sigma_1}\iv{\varepsilon_{s-1}}\sigma_1\\
\xlongequal{}&\varepsilon_s\sigma_1\varepsilon_{s+1}\iv{\varepsilon_{s-1}}\iv{\sigma_1}\iv{\varepsilon_{s-1}}\sigma_1\\
\xlongequal{}&\varepsilon_s\sigma_1\varepsilon_{s+1}\sigma_1\iv{\varepsilon_{s-1}}\iv{\sigma_1}\iv{\varepsilon_{s-1}}\\
\xlongequal{}&\varepsilon_s\iv{\varepsilon_{s-1}}\sigma_1\varepsilon_{s+1}\iv{\varepsilon_{s-1}}\\
=&\delta_s\sigma_1\delta_{s+1}
\end{array}\]
where the first and the last equalities are due to both $\delta_s=\varepsilon_s\iv{\varepsilon_{s-1}}$ and $\delta_{s+1}=\varepsilon_{s+1}\iv{\varepsilon_{s-1}}$, the second equality uses $\Co(\varepsilon_{s-1},\sigma_1\varepsilon_s\sigma_1)$ from \eqref{ap:05}, the third equality uses $\SCo(\sigma_1;\varepsilon_{s+1},\varepsilon_s)$ from \eqref{ap:06}, the fourth equality uses $\Co(\varepsilon_{s-1},\sigma_1\varepsilon_{s-1}\sigma_1)$ from \eqref{ap:04}, and the fifth equality uses $\Co(\varepsilon_{s-1},\sigma_1\varepsilon_{s+1}\sigma_1)$ from \eqref{ap:05}.
\end{proof}

%=========================================================
\section{A finite presentation for braid twist group}\label{sec:bt}
%=========================================================

Recall that the braid twist group $\BT(\surfO)$ of a decorated surface $\surfO$ is the subgroup of the surface braid group $\SBG(\surfO)$ generated by the braid twists along closed arcs in $\cA(\surfO)$. In the presentation of $\SBG(\surfO)$ in Proposition~\ref{lem:sbg}, the generators $\sigma_i$, $1\leq i\leq \aleph-1$, are braid twists, while the generators $\varepsilon_r$, $1\leq r\leq 2g+b-1$, are not. So we introduce the following elements in $\BT(\surfO)$:
\begin{align}
\tau_r=\sigma_1^{\iv{\varepsilon_r}}\label{tau},\quad 1\leq r\leq 2g+b-1
\end{align}
which, by Lemma~\ref{rmk:conj}, are the braid twists along the closed arcs $t_{\varepsilon_r}(\sigma_1)$, $1\leq r\leq 2g+b-1$, respectively. So as in Notation~\ref{nts:arc}, we denote $t_{\varepsilon_r}(\sigma_1)$ also by $\tau_r$.
We illustrate these elements in Figure~\ref{fig:BT}, where similarly as in Figure~\ref{fig:QZ's}, we use the notation
\[\omega_{-j}=\tau_{2j-1},\,
\omega_j=\tau_{2j},\,
\nu_k=\tau_{2g+k},\
\text{for $1\leq j\leq g$ and $1\leq k\leq b-1$}.\]
\begin{figure}[ht]\centering
	\begin{tikzpicture}[xscale=.28,yscale=.28]
	% sigma
	\draw[purple](10,-7)to(14,-9)(24,-9)to(20,-9);\draw[dashed,purple](14,-9)to(24,-9);
	\draw[purple](12,-8.8)node{\tiny{$\sigma_1$}} (22,-9.8)node{\tiny{$\sigma_{\aleph-1}$}};
	%\zeta
	\draw[ForestGreen]plot [smooth,tension=.8] coordinates {(10,-7)(-.2,2)(30,3)(29,-4)(14,-9)};
	\draw[ForestGreen]plot [smooth,tension=.8] coordinates {(10,-7)(5.5,1.5)(29,2.5)(27,-3)(14,-9)};
	\draw[ForestGreen](3.1,1.8)node{\tiny{$\nu_{b-1}\cdots\nu_1$}};
	%xi plus
	\draw[red]plot [smooth,tension=.8] coordinates {(14,-9)(24,-4)(29,1.5)(11,1.5)(10,-7)};
	\draw[red]plot [smooth,tension=.5] coordinates {(14,-9)(26,-1)(24,1.5)(10,-7)};
	\draw[red](9.6,-2)node{\tiny{$\omega_g$}} (20.5,-1)node{\tiny{$\omega_1$}};
	%xi minus
	\draw[red](10,-7)to(8+2+4-.5,0);
	\draw[red,dashed](8+2+4-.5,0).. controls +(60:2) and +(180:3) ..(20,5);
	\draw[red](14,-9).. controls +(10:25) and +(0:25) ..(20,5);
	\draw[red](10,-7).. controls +(35:5) and +(-100:2) ..(20+2+4-.3,0.3);
	\draw[red,dashed](20+2+4-.3,0.3).. controls +(80:3) and +(180:1) ..(32,5);
	\draw[red](14,-9).. controls +(5:30) and +(0:6) ..(32,5);
	\draw[red](11.5,-2)node{\tiny{$\omega_{-g}$}} (23,-1)node{\tiny{$\omega_{-1}$}};
	% surface
	\foreach \j in {1,2.5} {
		\draw[thick](8*\j-2+4+.5,0).. controls +(45:1) and +(135:1) ..(8*\j+2+4-.5,0);
		\draw[very thick](8*\j-2+4+.3,0.3).. controls +(-60:1) and +(-120:1) ..(8*\j+2+4-.3,0.3);
	}
	\draw[very thick] (0,5)to(38,5)(38,-11)(0,-11)to(38,-11) (0,5)to[bend right=90](0,-11);
	\draw[very thick,fill=gray!14] (38,-3) ellipse (1 and 8)
	node{\tiny{$\partial_{1}$}};
	\draw[very thick,fill=gray!14](6,0)circle(.7)node{\tiny{$\partial_2$}};
	\draw[very thick,fill=gray!14](1,0)circle(.7)node{\tiny{$\partial_{b}$}};
	\node at(3.5,0) {$\cdots$};
	\node at(18,0) {$\cdots$};
	\foreach \x/\y in {14/-9,10/-7,24/-9,20/-9}
	{\draw(\x,\y)node[white]{{$\bullet$}}node{$\circ$};}
	\draw (10,-7)node[below]{$Z_{1}$}
	(14,-9)node[below]{$Z_2$}
	(24,-9)node[below]{$Z_{\aleph}$};
	\end{tikzpicture}
	\caption{Generators for $\BT(\surfO)$}
	\label{fig:BT}
\end{figure}

The main result in this section is the following presentation of $\BT(\surfO)$.

\begin{theorem}\label{thm:pre}
	Suppose that either $\aleph\geq 5$, or $\aleph=4$ and $2g+b-1\leq 2$. The braid twist group $\BT(\surfO)$ has the following finite presentation.
	\begin{itemize}
		\item Generators: $\sigma_i$, $1\leq i\leq \aleph-1$, $\tau_r$, $1\leq r\leq 2g+b-1$.
		\item Relations: for $1\leq i,j\leq \aleph-1$ and $1\leq r,s\leq 2g+b-1$,
		\begin{align}
		&\Co(\sigma_i,\sigma_j)\label{nr:01}&&  \text{if $|i-j|>1$}\\
		&\Br(\sigma_i,\sigma_{j})\label{nr:02}&&\text{if $|i-j|=1$}\\
		&\Co(\tau_r,\sigma_i)\label{nr:03}&&\text{if $i>2$}\\
		&\Br(\tau_r,x)\label{nr:04}&&\\
		&\Br(\tau_r,y)\label{nr:05}&&\\
		&\Co({\tau_r}^y,{\tau_s}^x)\label{nr:06}&&\text{if $s<r$ and $s\notin\dgen$}\\
		&\Co({\tau_r}^{\iv{y}},{\tau_s}^x)\label{nr:07}&&\text{if $s<r$ and $s\in\dgen$}
		\end{align}
		where
		\begin{align}
		x:=\sigma_1^{\iv{\sigma_2}}=\sigma_2^{\sigma_1}\label{x},
		\end{align}
		\begin{align}
		y:=\sigma_3^{\iv{\sigma_2}}=\sigma_2^{\sigma_3}\label{y}.
		\end{align}
	\end{itemize}
\end{theorem}

%\begin{proof}
%	This follows directly from Lemma~\ref{lem:key},
%	where the assumptions hold due to Lemma~\ref{lem:b2}, Lemma~\ref{lem:b3}, Lemma~\ref{lem:b4}, Lemma~\ref{lem:2cases} and Lemma~\ref{lem:last4} in Appendix~\ref{app:lemmas}.
%\end{proof}

Apply this result to the case of decorated marked surface, we have the following.

\begin{corollary}\label{rmk:n=4}
	Let $\surfo$ be a decorated marked surface with $|\Tri|\geq 4$. Then $\BT(\surfo)$ has the above presentation.
\end{corollary}

\begin{proof}
	We only need to show that $2g+b-1\leq 2$ when $|\Tri|=4$. For a decorated marked surface $\surfo$, if $|\Tri|=4$ then $4g+2b+|M|-4=4$. Note that $|M|\geq b>0$. So either $g=1$ and $b=1$, or $g=0$ and $b\leq 2$. In each case, $2g+b-1\leq 2$ holds.
\end{proof}

%This section devotes to give a finite presentation of the braid twist group $\BT(\surfO)$ of a decorated surface $\surfO$.

%=========================================================
%\subsection{Generators}
%=========================================================

%Through this subsection, we assume $\aleph\geq 3$. Set

%We

The remaining of the section devotes to prove Theorem~\ref{thm:pre}. First we show $\sigma_i$ and $\tau_r$ form a set of generators for $\BT(\surfO)$.

\begin{lemma}\label{lem:gen}
	The braid twist group $\BT(\surfO)$ is generated by $\sigma_i$, $1\leq i\leq \aleph-1$, and $\tau_r$, $1\leq r\leq 2g+b-1$.
\end{lemma}
\begin{proof}
Let $N$ be the subgroup of $\BT(\surfO)$ generated by $\sigma_i$, $1\leq i\leq \aleph-1$, and $\tau_r$, $1\leq r\leq 2g+b-1$. To complete the proof, we only need to show that the braid twist $\eta$ is in $N$
for any $\eta\in\cA(\surfO)$. Without loss of generality, suppose that $Z_1$ is an endpoint of $\eta$.
As in Figure~\ref{fig:BT}, we use the notation
\[\omega_{-j}=\tau_{2j-1},\,
\omega_j=\tau_{2j},\,
\nu_k=\tau_{2g+k},\
\text{for $1\leq j\leq g$ and $1\leq k\leq b-1$}.\]

Add a marked point on each boundary component of $\surfO$ and take the set $\TT$ of (blue/cyan) open arcs in the upper picture in Figure~\ref{fig:poly} on $\surfO$ that divide $\surfO$ into two polygons: one (denoted by $\mathbf{L}$) contains the decorating point $Z_1$ and the other (denoted by $\mathbf{R}$) contains other decorating points.
\begin{figure}[t]
\begin{tikzpicture}[xscale=.3,yscale=.3]
% surface
\draw[very thick,fill=gray!14] (0,-14).. controls +(45:4) and +(135:4) ..(0,-14);
\draw[very thick,fill=gray!14] (-120:14).. controls +(15:4) and +(105:4) ..(-120:14);
\draw[very thick,fill=gray!14] (-150:14).. controls +(-15:4) and +(75:4) ..(-150:14);
%omega plus
\draw[cyan](75:15.5)node{$\overrightarrow{A_1}$};
\draw[blue](15:15.5)node{$\overrightarrow{A_g}$};
\draw[blue](-135:15.5)node{$\overleftarrow{A_g}$};
\draw[cyan](135:15.5)node{$\overleftarrow{A_1}$};
% omega minus
\draw[cyan] (-45:15.5) node {$\overrightarrow{B_g}$};
\draw[cyan] (-75:15.5) node {$\overrightarrow{B_1}$};
\draw[cyan] (165:15.5) node {$\overleftarrow{B_g}$};
\draw[cyan] (105:15.5) node {$\overleftarrow{B_1}$};

\draw[cyan, very thick](0,0)circle(14) (0,14).. controls +(-90:5) and +(165:6) ..(0,-14);
\draw[ultra thick, blue] (-90:14) arc(-90:-180:14);
\draw[ultra thick, blue] (-30:14) arc(-30:60:14);
		
\draw (-5,1) coordinate (t1)
(5,1) coordinate (t2) node[gray!11]{\Huge$\mathbf{R}$}
(0,-5) coordinate (t3);
\draw[purple](t1)to(t2);
\draw[purple,dashed](t3)to(t2);
\foreach \j in {105,135,165,-165} {\draw[red](t1)to(\j:14);}
\foreach \j in {165,-165}
{\draw[red](t1)to(\j:14);}
\foreach \j in {-142.5,-105}
{\draw[ForestGreen](t1)to[bend left=15](\j:14);}
\foreach \j in {45,75,-45,-75} {\draw[red](t2)to(\j:14);}
\foreach \j in {22.5,-15} {\draw[ForestGreen](t2)to(\j:14);}
		
\foreach \j in {1,2,3}
{\draw(t\j)node[white]{$\bullet$}node{$\circ$};}
\draw(t3)node[right]{\tiny{$Z_{\aleph}$}};
\draw(t1)node[above right]{\tiny{$Z_{1}$}};
\draw(t2)node[above left]{\tiny{$Z_{2}$}};

\foreach \j in {1,...,12}
{\draw[very thick](30*\j:14)node{$\bullet$};}
\draw[white,fill=white](-135:14)circle(.7)node[blue,rotate=-45]{{$\cdots$}};
\draw[white,fill=white](15:14)circle(.7)
node[blue,rotate=105]{{$\cdots$}};
\foreach \j in {-1,2,-5} {\draw[white,fill=gray!0] (90+30*\j:14) circle (1);
\draw[cyan] (90+30*\j:14)node[rotate=30*\j]{\Large{$\cdots$}};}
		
\draw(-90:12.5)node{\small{$\partial_1$}};
\draw(-120:12.5)node{\small{$\partial_{\mathrm{b}}$}};
\draw(-150:12.5)node{\small{$\partial_2$}};
\draw[ForestGreen](-127:11)node[rotate=-35]{{$\cdots$}} (-127:12)node{{\tiny{$\nu_k$}}} (6:12)node{{\tiny{$\nu_k$}}};
\draw[ForestGreen](6:11)node[rotate=95]{{$\cdots$}}
(185:11)node[red]{{\tiny{$\omega_g$}}}
(170:11)node[red]{{\tiny{$\omega_{-g}$}}}
(135:11)node[red]{{\tiny{$\omega_1$}}}
(118:11)node[red]{{\tiny{$\omega_{-1}$}}}
(75:11)node[red]{{\tiny{$\omega_{1}$}}}
(35:11)node[red]{{\tiny{$\omega_{g}$}}}
(-33:11)node[red]{{\tiny{$\omega_{-g}$}}}
(-75:11)node[red]{{\tiny{$\omega_{-1}$}}}
(90:1.5)node[purple]{{\tiny{$\sigma_1$}}}
(-45:2.5)node[purple]{{\tiny{$\sigma_i$}}};
\foreach \j in {-1.2,2,-4.8} {\draw[red](90+30*\j:11)node[rotate=30*\j]{\Large{$\cdots$}};}
		
\end{tikzpicture}
%=================================================================
		
\begin{tikzpicture}[xscale=.27,yscale=.27]
		\draw[blue,very thick](6,.7)to[bend left=7](10.5,0);
		\draw[blue, very thick,dashed](10.5,0)to(38,-13);
		
		\draw[purple](10,-7)to(14,-9);\draw[dashed,purple](14,-9)to(20,-10);
		\draw[ForestGreen]plot [smooth,tension=.8] coordinates {(10,-7)(-.2,2)(30,3)(28.5,-4)(14,-9)};
		\draw[ForestGreen,->-=.5,>=stealth](-.2,2)to(-.2+.01,2+.005);
		\draw[ForestGreen]plot [smooth,tension=.8] coordinates {(10,-7)(5.5,1.5)(30,2)(26,-4)(14,-9)};
		\draw[red]plot [smooth,tension=.8] coordinates {(14,-9)(28,0)(11,1)(10,-7)};
		
		\draw[red](10,-7)to(8+2+4-.5,0);
		\draw[red,dashed](8+2+4-.5,0).. controls +(60:2) and +(180:3) ..(20,5);
		\draw[red](14,-9).. controls +(5:20) and +(0:25) ..(20,5);
		%  \draw[red,thick](10,-7).. controls +(35:5) and +(-100:2) ..(20+2+4-.3,0.3);
		%  \draw[red,thick,dashed](20+2+4-.3,0.3).. controls +(80:3) and +(180:1) ..(32,5);
		%  \draw[red,thick](10,-7).. controls +(-10:30) and +(0:7) ..(32,5);
		
		\foreach \j in {1,2.5} {
			\draw[thick](8*\j-2+4+.5,0).. controls +(45:1) and +(135:1) ..(8*\j+2+4-.5,0);
			\draw[very thick](8*\j-2+4+.3,0.3).. controls +(-60:1) and +(-120:1) ..(8*\j+2+4-.3,0.3);
		}
		\draw[very thick]
		(0,5)to(38,5)(38,-13)(0,-13)to(38,-13)
		(0,5)to[bend right=90](0,-13);
		\draw[very thick,fill=gray!14] (38,-4) ellipse (1 and 9)
		node{\tiny{$\partial_{1}$}};
		\draw[thick,fill=gray!14](6,0)circle(.7)node{\tiny{$\partial_2$}};;
		\draw[thick,fill=gray!14](1,0)circle(.7)node{\tiny{$\partial_{b}$}};
		
		\draw[blue, ultra thick] (6,.7)to(1,.7) .. controls +(180:10) and +(178:40) ..(38,-13);
		
		\draw (6,.7)node{$\bullet$} (1,.7)node{$\bullet$} (38,-13)node{$\bullet$};
		
		\node at(3.5,0) {$\cdots$};
		\node at(18,0) {$\cdots$};
		\foreach \x/\y in {14/-9,10/-7,20/-10}
		{\draw(\x,\y)node[white]{{$\bullet$}}node{{$\circ$}};}
		
		\draw[ForestGreen](3.1,2)node{\tiny{$\nu_k$}}node[below]{\tiny{$\cdots$}};
		\draw[red](10,-2)node{\tiny{$\omega_g\qquad\omega_{-g}$}};
		%  (20.5,-1)node{\tiny{$\xi_1$}} (22.5,-1)node{\tiny{$\xi_{-1}$}}
		\draw[purple](11.5,-8.5)node{\small{$\sigma_1$}};
		\end{tikzpicture}
		\caption{Dividing $\surfO$ into two polygons}
		\label{fig:poly}
	\end{figure}
	Then the closed arcs (cf. the upper picture in Figure~\ref{fig:poly})
	\[\{ \sigma_1, \omega_{\pm j}, \nu_k\mid 1\leq j\leq g, 1\leq k\leq b-1\}\]
	form a `dual' set $\TT^*$ of $\TT$ in the sense
	that each of which intersect exactly one of the open arcs in $\TT$.
	Note that $\overrightarrow{A_g}$ consists of $b$ open arcs (where $b$ is the number of boundary components of $\surfO$),
	which are dual to $\nu_k$, $1\leq k\leq b-1$, and $\omega_g$, respectively (cf. the lower picture in Figure~\ref{fig:poly}).
	
	Use induction on $\Int(\eta, \TT)$. The initial case is when $\Int(\eta,\TT)=1$.
	Without loss of generality, suppose that the intersection is on $\overrightarrow{A_1}$.
	Then the braid twist $\eta$ and $\omega_1$ (the dual of $\overrightarrow{A_1}$) only differ by an element of
	\[\MCG(\mathbf{R}_{\aleph-1})=\<\sigma_i\mid 2\leq i\leq \aleph-1\>(\cong B_{\aleph-1}),\]
	where $\mathbf{R}_{\aleph-1}$ is the right polygon $\mathbf{R}$
	with decorating points $Z_i, {2\leq i\leq\aleph}$.
	Therefore, $\eta\in N$.
	Now consider the case when $\Int(\eta,\TT)=m\geq2$.
	Let $Z_q$ be the endpoint of $\eta$ other than $Z_1$.
	The intersections divide $\eta$ into $m+1$ segments
	$L_0,\cdots,L_m$ in order (so $L_0\subset\mathbf{L}$ and $L_m\subset \mathbf{R}$).
	Note that except $L_0$ and $L_m$, $L_t$ has both endpoints on arcs in $\TT$.
	Choose a decorating point $Z\in\mathbf{R}$ that is not $Z_q$
	and connect a line $l$ from $Z$ in $\mathbf{R}-L_m$ (which is still a disk)
	to one of $L_t, 1\leq t\leq m-1$
	so that it does not interest $\eta$ except at the endpoint $Y\in L_t$
	(cf. Figure~\ref{fig:alphabeta}).
	Then we can decompose $\eta$ into $\alpha$ and $\beta$,
	where $\alpha$ is isotopy to $l\cup YZ_q$ and
	$\beta$ is isotopy to $l\cup YZ_1$ such that:
	\begin{itemize}
		\item $\alpha,\beta\in\cA(\surfO)$ whose intersection numbers with $\TT$
		are less than $m$.
		\item As braid twists, $\eta=\beta^{\iv{\alpha}}$.
	\end{itemize}
	\begin{figure}[ht]\centering
		\begin{tikzpicture}[scale=1]
		\draw[orange,thick](0,0).. controls +(90:1) and +(-90:1) ..(3,2);
		\draw[orange,thick](6,0).. controls +(180:1) and +(0:1) ..(3,2);
		
		\draw[orange,thick](0,0).. controls +(90:.7) and +(210:1) ..(3,1)
		node[below,black]{$^Y$}node[black]{\tiny{$\bullet$}}
		.. controls +(30:1) and +(180:1) ..(6,0);
		\draw[bend right=15,thick](3,1)to(3,2) (3.1,1.5)node[right]{$l$};
		
		\draw(0,0)node[white] {$\bullet$} node[orange]{$\circ$}node[left]{$Z_1$}
		(3,2)node[white] {$\bullet$} node[red]{$\circ$}node[above]{$Z$}
		(6,0)node[white] {$\bullet$} node[red]{$\circ$}node[right]{$Z_q$}
		(4.8,1)node {$\alpha$}(1.5,1.3)node {$\beta$}(3.7,0.9)node[below] {$\eta$};
		\end{tikzpicture}
		\caption{Decomposing $\eta$}
		\label{fig:alphabeta}
	\end{figure}
	By inductive assumption, we have $\alpha, \beta\in N$
	and hence $\eta\in N$ as required.
\end{proof}

%\begin{lemma}\label{lem:pre}
%	The following relations hold in $\BT(\surfO)$:
%\end{lemma}

\begin{figure}
	\begin{tikzpicture}[scale=.6]
	\draw[white]plot [smooth,tension=1] coordinates
	{(2,0)(3.5,3)(-3,3)(-2.5,-1.8)(2,-1.5)};
	\draw[white]plot [smooth,tension=1] coordinates
	{(-2,-1.5)(-2.5,1)(2.5,-1.2)(2,-3)};
	\draw[red,thick]plot [smooth,tension=5] coordinates {(-2,0)(0,3)(2,0)};
	\draw[red,thick]plot [smooth,tension=1] coordinates {(-2,0)(0,1.5)(2,0)};
	\draw[red,thick](-2,0)to(-2,-1.5)(2,0)to(2,-1.5);
	\draw[dashed,red,thin](-2,0)to(2,0)to(-2,-1.5)to(2,-1.5)to(2,-3);
	\draw[red,thick](3.8,1.2)node[above]{$\tau_r$}(0,1.4)node[above]{$\tau_s$};
	\foreach \j/\k in {2/0,-2/0,2/-1.5,-2/-1.5,2/-3}
	{\draw(\j,\k)node[white]{$\bullet$}node[red]{$\circ$};}
	\draw[red](-2,-.75)node[left]{$x$}(2,-.75)node[right]{$y$}
	(2,-2.25)node[right]{$\sigma_4$}
	(0,0)node[above]{$\sigma_1$}
	(0,-1.5)node[below]{$\sigma_3$}
	(0,-.85)node[right]{$\sigma_2$};
	\end{tikzpicture}
	\begin{tikzpicture}[scale=.6]
	\draw[red,thick]plot [smooth,tension=3] coordinates {(-2,0)(2,3)(2,0)};
	\draw[white, fill=white](0,2.6)circle(.2);
	\draw[red,thick]plot [smooth,tension=3] coordinates {(-2,0)(-2,3)(2,0)};
	\draw[red,thick](-2,0)to(-2,-1.5)(2,0)to(2,-1.5);
	\draw[dashed,red,thin](-2,0)to(2,0)to(-2,-1.5)to(2,-1.5)to(2,-3);
	\draw[red,thick](-4,1.2)node[above]{$\tau_r$}(4,1.2)node[above]{$\tau_s$};
	\foreach \j/\k in {2/0,-2/0,2/-1.5,-2/-1.5,2/-3}
	{\draw(\j,\k)node[white]{$\bullet$}node[red]{$\circ$};}
	\draw[red](-2,-.75)node[left]{$x$}(2,-.75)node[right]{$y$}
	(0,0)node[above]{$\sigma_1$}
	(2,-2.25)node[right]{$\sigma_4$}
	(0,-1.5)node[below]{$\sigma_3$}
	(0,-.85)node[right]{$\sigma_2$};
	\end{tikzpicture}
	\caption{$s<r$ and $s\notin\dgen$ on the left, while $s\in\dgen$ on the right}\label{fig:xy}
\end{figure}

By Lemma~\ref{lem:btrel} and Lemma~\ref{rmk:conj}, it is easy to see that all relations in Theorem~\ref{thm:pre} hold. Indeed, note that $\sigma_i$ and $\sigma_j$ are disjoint if $|i-j|>1$ or are disjoint except sharing a common endpoint if $|i-j|=1$, and note that $\tau_r$ and $\sigma_i$ are disjoint if $i>2$ (see Figure~\ref{fig:BT}), so we have the relations \eqref{nr:01}, \eqref{nr:02} and \eqref{nr:03}. By Lemma~\ref{rmk:conj}, the elements $x=\sigma_1^{\iv{\sigma_2}}$ and $y=\sigma_3^{\iv{\sigma_2}}$ are the braid twists along the closed arcs $B_{\sigma_2}(\sigma_1)$ and  $B_{\sigma_2}(\sigma_3)$, respectively. As in Notation~\ref{nts:arc}, we denote ${\sigma_2}(\sigma_1)$ and  ${\sigma_2}(\sigma_3)$ also by $x$ and $y$, respectively, see Figure~\ref{fig:xy}. Then $x$ (resp. $y$) and $\tau_r$ are disjoint except sharing a common endpoint. So by Lemma~\ref{lem:btrel}, we have the relations \eqref{nr:04} and \eqref{nr:05}. For $s<r$ and $s\notin\dgen$, the elements ${\tau_r}^y$ and ${\tau_s}^x$ are the braid twists along the closed arcs $\iv{B_{y}}(\tau_r)$ and $\iv{B_{x}}(\tau_s)$, respectively, which are disjoint (see the left picture in Figure~\ref{fig:xy}), so we have the relation \eqref{nr:06}. Similarly, for $s<r$ and $s\notin\dgen$, the elements ${\tau_r}^{\iv{y}}$ and ${\tau_s}^x$ are the braid twists along the closed arcs ${B_{y}}(\tau_r)$ and $\iv{B_{x}}(\tau_s)$, respectively, which are disjoint (see the right picture in Figure~\ref{fig:xy}), so we have the relation \eqref{nr:07}.

% (see Figure~\ref{fig:BT} and Figure~\ref{fig:xy})

To prove that these relations are sufficient, we need to show that a group $\widetilde{N}$ with the following presentation:
\begin{itemize}
	\item Generators: $\ts_i$, $1\leq i\leq \aleph-1$, $\tt_r$, $1\leq r\leq 2g+b-1$.
	\item Relations: for $1\leq i,j\leq \aleph-1$ and $1\leq r,s\leq 2g+b-1$,
	\begin{align}
	&\Co(\ts_i,\ts_j)\label{def:01}&& \text{if $|i-j|>1$}\\
	&\Br(\ts_i,\ts_{j})\label{def:02}&&\text{if $|i-j|=1$}\\
	&\Co(\tt_r,\ts_i)\label{def:03}&&\text{if $i>2$}\\
	&\Br(\tt_r,\tx)\label{def:04}&&\\
	&\Br(\tt_r,\ty)\label{def:05}&&\\
	&\Co({\tt_r}^{\ty},{\tt_s}^{\tx})\label{def:06}&&\text{if $s<r$ and $s\notin\dgen$}\\
	&\Co({\tt_r}^{\iv{\ty}},{\tt_s}^{\tx})\label{def:07}&&\text{if $s<r$ and $s\in\dgen$}
	\end{align}
	where
	\begin{gather}\label{eq:tx}
	\tx:=\ts_1^{\iv{\ts_2}}={\ts_2}^{\ts_1}\end{gather}
	and
	\begin{gather}\label{eq:tx ty}
	\ty:=\ts_3^{\iv{\ts_2}}={\ts_2}^{\ts_3}\end{gather}
\end{itemize}
is isomorphic to $\BT(\surfO)$ by sending $\widetilde{\sigma}_i$ to $\sigma_i$ and sending $\widetilde{\tau}_r$ to $\tau_r$.

The following relations in $\widetilde{N}$ derived from the above relations will be useful later.

\begin{lemma}
	The following hold in $\widetilde{N}$.
	\begin{align}
	%&{\blue \Br(\widetilde{\tau}_r,\widetilde{\sigma}_2)}\label{c1:01}&&\\
	&\Co(\widetilde{x},\widetilde{y})\label{c1:02}&&\\
	&\Br(\widetilde{x},\widetilde{\sigma}_3)\label{c1:03}&&\\
	&\Br(\widetilde{y},\widetilde{\sigma}_3)\label{c1:04}&&\\
	&\Br(\tt_r,{\tt_s}^{\tx})\label{c1:05}&&\text{if $s<r$}\\
	&\Br(\tt_r^{\ty},{\tt_s})\label{c1:06}&&\text{if $s<r$ and $s\notin \dgen$}\\
	&\Br(\tt_r^{\iv{\ty}},{\tt_s})\label{c1:07}&&\text{if $s<r$ and $s\in \dgen$}
	\end{align}
\end{lemma}

\begin{proof}
	The relations \eqref{c1:02}, \eqref{c1:03} and \eqref{c1:04} follows directly from \eqref{eq:tx}, \eqref{eq:tx ty}, \eqref{def:01} and \eqref{def:02}.
	
	To show \eqref{c1:05}, by \eqref{def:05}, we have $\Br(\ty,\tt_s)$. Conjugated by ${\tx}$, we get $\Br(\ty^{\tx},{\tt_s}^{\tx})$. Since by \eqref{c1:02}, $\ty^{\tx}=\ty$, we deduce that $\Br(\ty,{\tt_s}^{\tx})$ holds. It follows from $\Br(\tt_r,\ty)$ \eqref{def:05} that both $\Br({\tt_r}^{\iv{\tt_r}^{\ty}},{\tt_s}^{\tx})$ and $\Br({\tt_r}^{{\tt_r}^{\iv{\ty}}},{\tt_s}^{\tx})$ hold. For the case $s\notin\dgen$, conjugating $\Br({\tt_r}^{\iv{\tt_r}^{\ty}},{\tt_s}^{\tx})$ by ${\tt_r^{\ty}}$ and using $\Co({\tt_r}^{\ty},{\tt_s}^{\tx})$ \eqref{def:06}, we get $\Br({\tt_r},{\tt_s}^{\tx})$; for the case $s\in\dgen$, conjugating $\Br({\tt_r}^{{\tt_r}^{\iv{\ty}}},{\tt_s}^{\tx})$ by ${\iv{\tt_r}^{\iv{\ty}}}$ and using $\Co(\iv{\tt_r}^{\iv{\ty}},{\tt_s}^{\tx})$ from \eqref{def:07}, we also get $\Br({\tt_r},{\tt_s}^{\tx})$.
	
	%for $s\notin\dgen$, we have
	%\[
	%\eqref{def:05}\xLongrightarrow{?^{\tx}}\Br(\ty^{\tx},{\tt_s}^{\tx})
	%\xLongrightarrow{\eqref{c1:02}}\Br(\ty,{\tt_s}^{\tx})
	%\xLongrightarrow{\eqref{def:05}}
	%\Br({\tt_r}^{\iv{\tt_r}^{\ty}},{\tt_s}^{\tx})
	%\xLongrightarrow[\eqref{def:06}]{{?^{\tt_r^{\ty}}}}\Br({\tt_r},{\tt_s}^{\tx});
	%\]
	%for $s\in\dgen$, we have
	%\[
	%\eqref{def:05}\xLongrightarrow{?^{\tx}}\Br(\ty^{\tx},{\tt_s}^{\tx})
	%\xLongrightarrow{\eqref{c1:02}}\Br(\ty,{\tt_s}^{\tx})
	%\xLongrightarrow{\eqref{def:05}}
	%\Br({\tt_r}^{{\tt_r}^{\iv{\ty}}},{\tt_s}^{\tx})
	%\xLongrightarrow[\eqref{def:07}]{{?^{\iv{\tt_r}^{\iv{\ty}}}}}\Br({\tt_r},{\tt_s}^{\tx}).
	%\]
	
	The relations \eqref{c1:06} and \eqref{c1:07} can be proved similarly.
	
	%To show \eqref{c1:06}, we have
	%\[
	%\eqref{def:04}\xLongrightarrow{?^{\ty}}\Br({\tt_r}^{\ty},\tx^{\ty})\xLongrightarrow{\eqref{c1:02}}\Br({\tt_r}^{\ty},\tx)\xLongrightarrow{\eqref{def:04}}\Br({\tt_r}^{\ty},{\tt_s}^{{\iv{\tt_s}}^{\tx}})\xLongrightarrow[\eqref{def:06}]{{?^{\tt_s^{\tx}}}}\Br({\tt_r}^{\ty},{\tt_s});
	%\]
	
	%To show \eqref{c1:07}, we have
	%\[
	%\eqref{def:04}\xLongrightarrow{?^{\iv{\ty}}}\Br({\tt_r}^{\iv{\ty}},\tx^{\iv{\ty}})\xLongrightarrow{\eqref{c1:02}}\Br({\tt_r}^{\iv{\ty}},\tx)\xLongrightarrow{\eqref{def:04}}\Br({\tt_r}^{\iv{\ty}},{\tt_s}^{{\iv{\tt_s}}^{\tx}})\xLongrightarrow[\eqref{def:06}]{{?^{\tt_s^{\tx}}}}\Br({\tt_r}^{\iv{\ty}},{\tt_s})
	%\]
\end{proof}

To compare $\widetilde{N}$ with $\BT(\surfO)$, we shall consider the following group action. Let $H$ be the subgroup of  $\SBG(\surfO)$ generated by $\varepsilon_r$, $1\leq r\leq 2g+b-1$ (which are the generators in the presentation for $\SBG(\surfO)$ in Proposition~\ref{lem:sbg}, that are not in $\BT(\surfO)$). The group $H$ can be regarded as the fundamental group
$\pi_1(\mathrm{S},Z_1)$ of the underlying surface $\mathrm{S}$ at the base point $Z_1$.
%where $\surfO-\{Z_1\}$ is the surface obtained from $\surfO$ by forgetting $Z_1$ as a decorating point.
Thus $H$ is freely generated by $\varepsilon_r$, $1\leq r\leq 2g+b-1$.
%We see that $H$ is a free group, regarding $H$
Since $\BT(\surfO)$ is a normal subgroup of $\SBG(\surfO)$ (by Lemma~\ref{rmk:conj}), there is an action of $H$ on $\BT(\surfO)$ by conjugation.

\begin{lemma}\label{lem:relforbt}
The action of $H$ on $\BT(\surfO)$ by conjugation is given explicitly as follows.
\[\begin{array}{ccl}
	%\varepsilon_t \sigma_1\iv{\varepsilon_t}&=&\ \ \  \tau_t\\
	\sigma_i^{\iv{\varepsilon_t}}&=&\begin{cases}
	\tau_t&\text{if $i=1$}\\
	\sigma_i&\text{if $i\neq 1$}
	\end{cases}\\
	\tau_r^{\iv{\varepsilon_t}}&=&\begin{cases}
	x^{\sigma_2\tau_t}&\text{if $r=t$;}\\
	\tau_r^{x\tau_t\sigma_2\tau_t}&\text{if $r<t$ and $r\notin\dgen$;}\\
	\tau_r^{x\iv{\tau_t}\sigma_2\tau_t}&\text{if $r<t$ and $r\in\dgen$;}\\
	\tau_r^{\iv{x}\iv{\tau_t}\sigma_2\tau_t}&\text{if $r>t$ and $t\notin\dgen$;}\\
	\tau_r^{\iv{x}\tau_t\sigma_2\tau_t}&\text{if $r>t$ and $t\in\dgen$.}
	\end{cases}
	\end{array}
	\]
where $1\leq i\leq \aleph-1$ and $1\leq r,t\leq 2g+b-1$.
\end{lemma}

\begin{proof}
The first equality follows from \eqref{tau} and \eqref{ap:03}. To show the second equality, for $r=t$, we have
\[\tau_t^{\iv{\varepsilon_t}}
\xlongequal{\eqref{tau}}\sigma_1^{\iv{\varepsilon_t}\iv{\varepsilon_t}}
\xlongequal{\eqref{ap:04}}
\sigma_1^{\varepsilon_t\sigma_1\iv{\varepsilon_t}}
\xlongequal{\eqref{tau}}\sigma_1^{\tau_t}
\xlongequal{\eqref{x}}x^{\sigma_2\tau_t}
\]
(where each equality is labeled by the relation which is used there);
if $r<t$ and $r\notin\dgen$, we have
\[\begin{split}
&\tau_r^{x\tau_t\sigma_2\tau_t}
\xlongequal{\eqref{x}}\tau_r^{\sigma_2\sigma_1\iv{\sigma_2}\tau_t\sigma_2\tau_t}
\xlongequal{\eqref{tau}}
\sigma_1^{\iv{\varepsilon_r}\sigma_2\sigma_1\iv{\sigma_2}\varepsilon_t\sigma_1\iv{\varepsilon_t}\sigma_2\varepsilon_t\sigma_1\iv{\varepsilon_t}}
\xlongequal{\eqref{ap:03}}
\sigma_1^{\sigma_2\iv{\varepsilon_r}\sigma_1\varepsilon_t\iv{\sigma_2}\sigma_1\sigma_2\sigma_1\iv{\varepsilon_t}}\\
&\xlongequal{\eqref{ap:02}}
\sigma_2^{\iv{\sigma_1}\iv{\varepsilon_r}\sigma_1\varepsilon_t\sigma_1\sigma_2\iv{\varepsilon_t}}
\xlongequal{\eqref{ap:05}}
\sigma_2^{\varepsilon_t\sigma_1\iv{\varepsilon_r}\sigma_2\iv{\varepsilon_t}}
\xlongequal{\eqref{ap:03}}
\sigma_2^{\sigma_1\sigma_2\iv{\varepsilon_r}\iv{\varepsilon_t}}
\xlongequal{\eqref{ap:02}}
\sigma_1^{\iv{\varepsilon_r}\iv{\varepsilon_t}}
\xlongequal{\eqref{tau}}
\tau_r^{\iv{\varepsilon_t}};
\end{split}\]
the other three cases can be proved similarly.

\end{proof}

%=========================================================
%\subsection{Relations}
%=========================================================

%Through this subsection, we assume $\aleph\geq4$. Set

On the other hand, we are going to construct an action of $H$ on $\widetilde{N}$. Denote by $\Lambda$ the set $\{\ts_i,\ 1\leq i\leq\aleph-1;\ \tt_r,\ 1\leq r\leq 2g+b-1 \}$
of generators of $\widetilde{N}$. To any generator $\varepsilon_t$ of $H$, we associate a map
\[\begin{array}{rccc}
\rho(\varepsilon_t):& \Lambda&\to&\widetilde{N}\\
&\lambda&\mapsto&\rho(\varepsilon_t).\lambda
\end{array}\]
where (comparing with the formulas in Lemma~\ref{lem:relforbt}).
\begin{align}
%\rho(\varepsilon_t).\ts_1\label{eq:ss01}&=\tt_t&&\\
\rho(\varepsilon_t).\ts_i\label{eq:ss02}&=\begin{cases}
\tt_t&\text{\qquad\ if $i=1$}\\
\ts_i&\text{\qquad\ if $i\neq 1$}
\end{cases}\\
\rho(\varepsilon_t).\tt_r\label{eq:ss03}&=\begin{cases}
\tx^{\ts_2\tt_t}&\text{if $r=t$}\\
{\tt_r}^{\tx\tt_t\ts_2\tt_t}&\text{if $r<t$ and $r\notin\dgen$}\\
{\tt_r}^{\tx\iv{\tt_t}\ts_2\tt_t}&\text{if $r<t$ and $r\in\dgen$}\\
{\tt_r}^{\iv{\tx}\iv{\tt_t}\ts_2\tt_t}&\text{if $r>t$ and $t\notin\dgen$}\\
{\tt_r}^{\iv{\tx}\tt_t\ts_2\tt_t}&\text{if $r>t$ and $t\in\dgen$}
\end{cases}
\end{align}
%in which
To the inverse $\iv{\varepsilon_t}$, we also associate a map
\[\begin{array}{rccc}
\rho(\iv{\varepsilon_t}):& \Lambda&\to&\widetilde{N}\\
&\lambda&\mapsto&\rho(\iv{\varepsilon_t}).\lambda
\end{array}\]
where
\begin{align}
%\rho(\iv{\varepsilon_t}).\ts_1&={\tt_t}^{\iv{\ts_1}}\label{eq:ss04}&&\\
\rho(\iv{\varepsilon_t}).\ts_i\label{eq:ss05}&=\begin{cases}
{\tt_t}^{\iv{\ts_1}}&\qquad\ \text{if $i=1$}\\
\ts_i&\qquad\ \text{if $i\neq 1$}
\end{cases}\\
\rho(\iv{\varepsilon_t}).\tt_r&=\label{eq:ss06}\begin{cases}
\ts_1&\text{if $r=t$}\\
{\tt_r}^{\iv{\ts_2}\iv{\tt_t}\iv{\ts_1}\iv{\ts_2}}&\text{if $r<t$ and $r\notin\dgen$}\\
{\tt_r}^{\ts_2\iv{\tt_t}\iv{\ts_1}\iv{\ts_2}}&\text{if $r<t$ and $r\in\dgen$}\\
{\tt_r}^{\ts_2\tt_t\iv{\ts_1}\iv{\ts_2}}&\text{if $r>t$ and $t\notin\dgen$}\\
{\tt_r}^{\iv{\ts_2}\tt_t\iv{\ts_1}\iv{\ts_2}}&\text{if $r>t$ and $t\in\dgen$}
\end{cases}
\end{align}

%The following alternative description of $\rho(\iv{\varepsilon_t}).\tt_r$ will be useful later.

%\begin{lemma}\label{lem:ac}
%We have the following.
%\[\]
%\end{lemma}

Since $H$ is freely generated by $\varepsilon_1,\cdots,\varepsilon_{2g+b-1}$, we have that $\rho$ induces an action of $H$ on $\widetilde{N}$ if and only if each $\rho(\varepsilon_r)$ induces an automorphism of $\widetilde{N}$. In this case, one can identify $\widetilde{N}$ with $\BT(\surfO)$ as follows.

\begin{lemma}\label{lem:key}
If both $\rho(\varepsilon_t)$ and $\rho(\iv{\varepsilon_t})$ induce endomorphisms of the group $\widetilde{N}$ for any $1\leq t\leq 2g+b-1$, then there is a group isomorphism from $\widetilde{N}$ to $\BT(\surfO)$, sending $\ts_i$ to $\sigma_i$ and sending $\tt_r$ to $\tau_r$.
\end{lemma}

\begin{proof}
We use the same notation $\rho(\varepsilon_t)$ and $\rho(\iv{\varepsilon_t})$ to denote their induced endomorphisms of $\widetilde{N}$, respectively. We claim that $\rho(\varepsilon_t)$ is an automorphism of $\widetilde{N}$. Since it is easy to see that the elements $\rho(\varepsilon_t).\ts_i$, $1\leq i\leq \aleph-1$, and $\rho(\varepsilon_t).\tt_r$, $1\leq r\leq 2g+b-1$, generate $\widetilde{N}$, the map $\rho(\varepsilon_t)$ is surjective. To show it is injective, we only need to show that $\rho(\iv{\varepsilon_t})\circ\rho(\varepsilon_t)$ is the identity on the set $\Lambda$ of generators of $\widetilde{N}$. It is straightforward to see $\left(\rho(\iv{\varepsilon_t})\circ\rho(\varepsilon_t)\right).\ts_i=\ts_i$ for any $1\leq i\leq \aleph-1$; to prove $\left(\rho(\iv{\varepsilon_t})\circ\rho(\varepsilon_t)\right).\tt_r=\tt_r$,
%\[\left(\rho(\iv{\varepsilon_t})\circ\rho(\varepsilon_t)\right).\ts_1\xlongequal{\eqref{eq:ss02}}\rho(\iv{\varepsilon_t}).(\tt_t)\xlongequal{\text{\eqref{eq:ss06}}}\ts_1.\]
%For $i>1$, we have
%\[
%\left(\rho(\iv{\varepsilon_t})\circ\rho(\varepsilon_t)\right).\ts_i\xlongequal{\eqref{eq:ss02}}\rho(\iv{\varepsilon_t}).\ts_i\xlongequal{\eqref{eq:ss05}}\ts_i.\]
for $r=t$, we have $\rho(\varepsilon_t).\tt_r=\tx^{\ts_2\tt_t}$ \eqref{eq:ss03}, which, by $\tx={\ts_1}^{\iv{\ts_2}}$ \eqref{eq:tx}, equals ${\ts_1}^{\tt_t}$, so by $\rho(\iv{\varepsilon_t}).\ts_1={\tt_t}^{\ts_1}$ \eqref{eq:ss05} and $\rho(\iv{\varepsilon_t}).\tt_t=\ts_1$ \eqref{eq:ss06}, we have $\rho(\iv{\varepsilon_t}).\left(\rho(\varepsilon_t).\tt_r\right)=\rho(\iv{\varepsilon_t}).\left({\ts_1}^{\tt_t}\right)=\left({\tt_t}^{\iv{\ts_1}}\right)^{\ts_1}=\tt_t$;
%\[\left(\rho(\iv{\varepsilon_t})\circ\rho(\varepsilon_t)\right).\tt_r\xlongequal{\eqref{eq:ss03}}\rho(\iv{\varepsilon_t}).\left(\tx^{\ts_2\tt_t}\right)\xlongequal{\eqref{eq:tx}}\rho(\iv{\varepsilon_t}).\left(\ts_1^{\tt_t}\right)\xlongequal[\text{Lemma~\ref{lem:ac}}]{\eqref{eq:ss05}}\left({\tt_t}^{\iv{\ts_1}}\right)^{\ts_1}=\tt_t;\]
for $r<t$ and $r\notin\dgen$, we have $\rho(\varepsilon_t).\tt_r={\tt_r}^{\tx\tt_t\ts_2\tt_t}$ \eqref{eq:ss03}, so using \eqref{eq:ss05} and \eqref{eq:ss06}, we have $\rho(\iv{\varepsilon_t}).\left(\rho(\varepsilon_t).\tt_r\right)=\left({\tt_r}^{\iv{\ts_2}\iv{\tt_t}\iv{\ts_1}\iv{\ts_2}}\right)^{{\tt_t}^{\iv{\ts_1}\iv{\ts_2}}\ts_1\ts_2\ts_1}={\tt_r}^{\iv{\ts_2}\iv{\tt_t}\iv{\ts_1}\iv{\ts_2}\ts_2\ts_1{\tt_t}{\iv{\ts_1}\iv{\ts_2}}\ts_1\ts_2\ts_1}={\tt_r}^{\iv{\ts_2}{\iv{\ts_1}\iv{\ts_2}}\ts_1\ts_2\ts_1}=\tt_r$ where the last equality uses the braid relation $\Br(\ts_1,\ts_2)$ \eqref{def:02};
%$\rho(\iv{\varepsilon_t}).\tt_r={\tt_r}^{\iv{\ts_2}\iv{\tt_t}\iv{\ts_1}\iv{\ts_2}}$ \eqref{eq:ss06},
%\[
%\left(\rho(\iv{\varepsilon_t})\circ\rho(\varepsilon_t)\right).\tt_r\xlongequal{\eqref{eq:ss03}}\rho(\iv{\varepsilon_t}).\left({\tt_r}^{\tx\tt_t\ts_2\tt_t}\right)\xlongequal[\eqref{eq:ss05}]{\text{Lemma~\ref{lem:ac}}}\left({\tt_r}^{\iv{\ts_2}\iv{\tt_t}\iv{\ts_1}\iv{\ts_2}}\right)^{{\tt_t}^{\iv{\ts_1}\iv{\ts_2}}\ts_1\ts_2\ts_1}\xlongequal{\eqref{def:02}}\tt_r;
%\]
the other cases can be proved similarly.
%For $r<t$ and $r\in\dgen$, we have
%\[
%\left(\rho(\iv{\varepsilon_t})\circ\rho(\varepsilon_t)\right).\tt_r\xlongequal{\eqref{eq:ss03}}\rho(\iv{\varepsilon_t}).\left({\tt_r}^{\tx\iv{\tt_t}\ts_2\tt_t}\right)\xlongequal[\eqref{eq:ss05}]{\text{Lemma~\ref{lem:ac}}}\left({\tt_r}^{\ts_2\iv{\tt_t}\iv{\ts_1}\iv{\ts_2}}\right)^{{\tt_t}^{\iv{\ts_1}\iv{\ts_2}}\iv{\ts_1}\ts_2\ts_1}\xlongequal{\eqref{def:02}}\tt_r.
%\]
%For $r>t$ and $r\notin\dgen$, we have
%\[
%\left(\rho(\iv{\varepsilon_t})\circ\rho(\varepsilon_t)\right).\tt_r\xlongequal{\eqref{eq:ss03}}\rho(\iv{\varepsilon_t}).\left({\tt_r}^{\iv{\tx}\iv{\tt_t}\ts_2\tt_t}\right)\xlongequal[\eqref{eq:ss05}]{\text{Lemma~\ref{lem:ac}}}\left({\tt_r}^{\ts_2\tt_t\iv{\ts_1}\iv{\ts_2}}\right)^{{\iv{\tt_t}}^{\iv{\ts_1}\iv{\ts_2}}\iv{\ts_1}\ts_2\ts_1}\xlongequal{\eqref{def:02}}\tt_r.
%\]
%For $r>t$ and $r\in\dgen$, we have
%\[
%\left(\rho(\iv{\varepsilon_t})\circ\rho(\varepsilon_t)\right).\tt_r\xlongequal{\eqref{eq:ss03}}\rho(\iv{\varepsilon_t}).\left({\tt_r}^{\iv{\tx}\tt_t\ts_2\tt_t}\right)\xlongequal[\eqref{eq:ss05}]{\text{Lemma~\ref{lem:ac}}}\left({\tt_r}^{\iv{\ts_2}\tt_t\iv{\ts_1}\iv{\ts_2}}\right)^{{\iv{\tt_t}}^{\iv{\ts_1}\iv{\ts_2}}\ts_1\ts_2\ts_1}\xlongequal{\eqref{def:02}}\tt_r.
%\]
Hence $\rho(\iv{\varepsilon_t})\circ\rho(\varepsilon_t)$ is the identity on $\Lambda$. Thus, we complete the proof of the claim that $\rho(\varepsilon_t)$ is an isomorphism.

Since $H$ is generated freely by $\varepsilon_t$, it follows that $\rho$ induces an action of $H$ on $\widetilde{N}$. Let $\widetilde{G}=\widetilde{N}\rtimes_{\rho} H$ be the outer semidirect product of $\widetilde{N}$ and $H$ with respect to $\rho$. Then $\widetilde{G}$ has a presentation whose generators are the union of the generators of $\widetilde{N}$ and the generators of $H$, and whose relations are the union of the relations of $\widetilde{N}$ and the relations $\lambda^{\iv{\varepsilon_t}}=\rho(\varepsilon_t).\lambda$ for all $\lambda\in \Lambda$.
%\begin{itemize}
%	\item Generators: $\ts_i$, $1\leq i\leq \aleph-1$, $\tt_r$, $1\leq r\leq 2g+b-1$, $\varepsilon_r$, $1\leq r\leq 2g+b-1$.
%	\item Relations:
%	\[\begin{array}{ll}
%	\Co(\ts_i,\ts_j)&  \text{if $|i-j|>1$}\\
%	\Br(\ts_i,\ts_{i+1})&\text{if $i\neq \aleph-1$}\\
%	\Co(\tt_r,\ts_i)&\text{if $i>2$}\\
%	\Br(\tt_r,\tx)&\\
%	\Br(\tt_r,\widetilde{y})&\\
%	\Co({\tt_r}^{\ty},{\tt_s}^{\tx})&\text{if $s<r$ and $s\notin\dgen$}\\
%	\Co({\tt_r}^{\iv{\ty}},{\tt_s}^{\tx})&\text{if $s<r$ and $s\in\dgen$}\\
%	\ts_1^{\iv{\varepsilon_t}}=\tt_t&\\
	%\ts_i^{\iv{\varepsilon_t}}=\ts_i&\text{if $i>1$}\\
%	\tt_r^{\iv{\varepsilon_t}}=\ts_1^{\tt_t}&\text{if $r=t$}\\
%	\tt_r^{\iv{\varepsilon_t}}=\tt_r^{\tx \tt_t\ts_2\tt_t}&\text{if $r<t$ and $r\notin\dgen$}\\
%	\tt_r^{\iv{\varepsilon_t}}=\tt_r^{\tx\iv{\tt_t}\ts_2\tt_t}&\text{if $r<t$ and $r\in\dgen$}\\
%	\tt_r^{\iv{\varepsilon_t}}=\tt_r^{\iv{\tx}\iv{\tt_t}\ts_2\tt_t}&\text{if $r>t$ and $t\notin\dgen$}\\
%	\tt_r^{\iv{\varepsilon_t}}=\tt_r^{\iv{\tx}\tt_t\ts_2\tt_t}&\text{if $r>t$ and $t\in\dgen$}
%	\end{array}\]
%	where $1\leq i,j\leq \aleph-1$ and $1\leq r,s,t\leq 2g+b-1$.
%\end{itemize}
%Due to the relations \eqref{nr:01}-\eqref{nr:07}By Lemma~\ref{lem:pre} and Lemma~\ref{lem:relforbt}, i

Note that we have proved the relations in Theorem~\ref{thm:pre} and gave the explicit formulas for the action of $H$ on $\Br(\surfO)$ in Lemma~\ref{lem:relforbt}. Comparing these with the presentation of $\widetilde{N}$ and the action of $H$ on $\widetilde{N}$, it follows that there is a homomorphism $\Phi$ from $\widetilde{G}$ to $\SBG(\surfO)$ sending $\ts_i$ to $\sigma$, sending $\tt_r$ to $\tau_r$ and sending $\varepsilon_r$ to $\varepsilon_r$. By the relations in Proposition~\ref{lem:sbg}, $\ker\Phi$ is generated by $\ts_i\ts_j\iv{\ts_i}\iv{\ts_j}$ for $|i-j|>1$, $\ts_i\ts_j\ts_i\iv{\ts_j}\iv{\ts_i}\iv{\ts_j}$ for $|i-j|=1$, $\ts_i\varepsilon_r\iv{\ts_i}\iv{\varepsilon_r}$ for $i\neq 1$, $\iv{\varepsilon_r}\iv{\ts_1}\varepsilon_r\ts_1\varepsilon_r\ts_1\iv{\varepsilon_r}\iv{\ts_1}$ for all $r$, $\varepsilon_s\ts_1\varepsilon_r\ts_1\iv{\varepsilon_s}\iv{\ts_1}\iv{\varepsilon_r}\iv{\ts_1}$ for $s<r$ with $s\notin\dgen$, and $\varepsilon_s\ts_1\varepsilon_r\iv{\ts_1}\iv{\varepsilon_s}\iv{\ts_1}\iv{\varepsilon_r}\iv{\ts_1}$ for $s<r$ with $s\in\dgen$.  We calculate these elements in the following.
By \eqref{def:01} and \eqref{def:02}, $\ts_i\ts_j\iv{\ts_i}\iv{\ts_j}$ (for $|i-j|>1$) and $\ts_i\ts_j\ts_i\iv{\ts_j}\iv{\ts_i}\iv{\ts_j}$ (for $|i-j|=1$) are the identity. For $i\neq 1$, we have \[\ts_i\varepsilon_r\iv{\ts_i}\iv{\varepsilon_r}=\ts_i\left(\rho(\varepsilon_r).\iv{\ts_i}\right)=\ts_i\iv{\ts_i}=1.\]
For all $r$, we have \[\iv{\varepsilon_r}\iv{\ts_1}\varepsilon_r\ts_1\varepsilon_r\ts_1\iv{\varepsilon_r}\iv{\ts_1}=\left(\rho(\iv{\varepsilon_r}).\iv{\ts_1}\right)\ts_1\left(\rho(\varepsilon_r).\ts_1\right)\iv{\ts_1}=\iv{\tt_t}^{\iv{\ts_1}}\ts_1\tt_t\iv{\ts_1}=1.\]
For $s<r$ with $s\notin\dgen$, we have  \[\varepsilon_s\ts_1\varepsilon_r\ts_1\iv{\varepsilon_s}\iv{\ts_1}\iv{\varepsilon_r}\iv{\ts_1}=\left(\rho(\varepsilon_s).\ts_1\right)\left(\rho(\varepsilon_s\varepsilon_r).\ts_1\right)\left(\rho(\varepsilon_s\varepsilon_r\iv{\varepsilon_s}).\iv{\ts_1}\right)\left(\rho(\varepsilon_s\varepsilon_r\iv{\varepsilon_s}\iv{\varepsilon_r}).\iv{\ts_1}\right)\varepsilon_s\varepsilon_r\iv{\varepsilon_s}\iv{\varepsilon_r}.\]
For $s<r$ with $s\in\dgen$, we have \[\varepsilon_s\ts_1\varepsilon_r\iv{\ts_1}\iv{\varepsilon_s}\iv{\ts_1}\iv{\varepsilon_r}\iv{\ts_1}=\left(\rho(\varepsilon_s).\ts_1\right)\left(\rho(\varepsilon_s\varepsilon_r).\iv{\ts_1}\right)\left(\rho(\varepsilon_s\varepsilon_r\iv{\varepsilon_s}).\iv{\ts_1}\right)\left(\rho(\varepsilon_s\varepsilon_r\iv{\varepsilon_s}\iv{\varepsilon_r}).\iv{\ts_1}\right)\varepsilon_s\varepsilon_r\iv{\varepsilon_s}\iv{\varepsilon_r}.\]
For  $1\leq s<r\leq 2g+b-1$, we set
\[n_{s,t}=\begin{cases}
\left(\rho(\varepsilon_s).\ts_1\right)\left(\rho(\varepsilon_s\varepsilon_r).\ts_1\right)\left(\rho(\varepsilon_s\varepsilon_r\iv{\varepsilon_s}).\iv{\ts_1}\right)\left(\rho(\varepsilon_s\varepsilon_r\iv{\varepsilon_s}\iv{\varepsilon_r}).\iv{\ts_1}\right) & \text{if $s\notin\dgen$;}\\
\left(\rho(\varepsilon_s).\ts_1\right)\left(\rho(\varepsilon_s\varepsilon_r).\iv{\ts_1}\right)\left(\rho(\varepsilon_s\varepsilon_r\iv{\varepsilon_s}).\iv{\ts_1}\right)\left(\rho(\varepsilon_s\varepsilon_r\iv{\varepsilon_s}\iv{\varepsilon_r}).\iv{\ts_1}\right) & \text{if $s\in\dgen$.}
\end{cases}\]
Then we have that $\ker\Phi$ is generated by $n_{s,t}\varepsilon_s\varepsilon_r\iv{\varepsilon_s}\iv{\varepsilon_r}$ for $1\leq s<r\leq 2g+b-1$. Note that $n_{s,t}$ is in $\widetilde{N}$ and $\varepsilon_s\varepsilon_r\iv{\varepsilon_s}\iv{\varepsilon_r}$, $1\leq s<r\leq 2g+b-1$, freely generate a subgroup of $H$. %$\langle\varepsilon_s\varepsilon_r\iv{\varepsilon_s}\iv{\varepsilon_r}\mid 1\leq s<r\leq 2g+b-1\rangle$ of $H$ is freely generated by its generators, we have that
It follows that $\ker\Phi\cap\widetilde{N}=\emptyset$. Hence the homomorphism $\Phi$ restricting on $\widetilde{N}$ is injective. Therefore we get the required isomorphism.
\end{proof}

\begin{remark}
By the above proof, the surface braid group $\SBG(\surfO)$ is not the (inner) semidirect product of $\BT(\surfO)$ and $H$, except that $\surfO$ is a decorating disk or a decorating annulus. For a decorating disk, $H$ is trivial; for a decorating annulus, this is the case in \cite{KP}.
\end{remark}

By Lemma~\ref{lem:key}, to complete the proof of Theorem~\ref{thm:pre}, we only need to show that $\rho(\varepsilon_t)$ and $\rho(\iv{\varepsilon_t})$ induce endomorphisms of $\widetilde{N}$ for every $1\leq t\leq 2g+b-1$, that is, to show that the images of generating relations \eqref{def:01}--\eqref{def:07} for $\widetilde{N}$ under the action of $\rho(\varepsilon_t)$ and $\rho(\iv{\varepsilon_t})$ respectively still hold in $\widetilde{N}$. Explicitly, we need to show that the following hold: for $1\leq i,j\leq \aleph-1$ and $1\leq r,s\leq 2g+b-1$,
%To show that $\varphi(\varepsilon_t)$ induces an endomorphism of $\widetilde{N}$, we shall prove the following relations.
\begin{align}
&\Co(\rho(\varepsilon_t).\ts_i,\rho(\varepsilon_t).\ts_j)\label{ac:01}
&&\text{if $|i-j|>1$}\\
&\Br(\rho(\varepsilon_t).\ts_i,\rho(\varepsilon_t).\ts_{j})\label{ac:02}
&&\text{if $|i-j|=1$}\\
&\Co(\rho(\varepsilon_t).\tt_r,\rho(\varepsilon_t).\ts_i)\label{ac:03}
&&\text{if $i>2$}\\
&\Br(\rho(\varepsilon_t).\tt_r,\rho(\varepsilon_t).\tx)\label{ac:04}&&\\
&\Br(\rho(\varepsilon_t).\tt_r,\rho(\varepsilon_t).\ty)\label{ac:05}&&\\
&\Co({(\rho(\varepsilon_t).\tt_r)}^{(\rho(\varepsilon_t).\ty)},{(\rho(\varepsilon_t).\tt_s)}^{(\rho(\varepsilon_t).\tx)})\label{ac:06}
&&\text{if $s<r$ and $s\notin\dgen$}\\
&\Co({(\rho(\varepsilon_t).\tt_r)}^{\iv{(\rho(\varepsilon_t).\ty)}},{(\rho(\varepsilon_t).\tt_s)}^{(\rho(\varepsilon_t).\tx)})\label{ac:07}
&&\text{if $s<r$ and $s\in\dgen$}
\end{align}
%where $1\leq i,j\leq \aleph-1$ and $1\leq r,s\leq 2g+b-1$.
and
%Similarly, to show that $\varphi(\iv{\varepsilon_t})$ induces an endomorphism of $\widetilde{N}$, we shall prove the following relations.
\begin{align}
&\Co(\rho(\iv{\varepsilon_t}).\ts_i,\rho(\iv{\varepsilon_t}).\ts_j)\label{iac:01}
&&\text{if $|i-j|>1$}\\
&\Br(\rho(\iv{\varepsilon_t}).\ts_i,\rho(\iv{\varepsilon_t}).\ts_{j})\label{iac:02}
&&\text{if $|i-j|=1$}\\
&\Co(\rho(\iv{\varepsilon_t}).\tt_r,\rho(\iv{\varepsilon_t}).\ts_i)\label{iac:03}
&&\text{if $i>2$}\\
&\Br(\rho(\iv{\varepsilon_t}).\tt_r,\rho(\iv{\varepsilon_t}).\tx)\label{iac:04}&&\\
&\Br(\rho(\iv{\varepsilon_t}).\tt_r,\rho(\iv{\varepsilon_t}).\ty)\label{iac:05}&&\\
&\Co({(\rho(\iv{\varepsilon_t}).\tt_r)}^{(\rho(\iv{\varepsilon_t}).\ty)},{(\rho(\iv{\varepsilon_t}).\tt_s)}^{(\rho(\iv{\varepsilon_t}).\tx)})\label{iac:06}
&&\text{if $s<r$ and $s\notin\dgen$}\\
&\Co({(\rho(\iv{\varepsilon_t}).\tt_r)}^{\iv{(\rho(\iv{\varepsilon_t}).\ty)}},{(\rho(\iv{\varepsilon_t}).\tt_s)}^{(\rho(\iv{\varepsilon_t}).\tx)})\label{iac:07}
&&\text{if $s<r$ and $s\in\dgen$}
\end{align}
where
\begin{align}
&\rho(\varepsilon_t).\tx:=\left(\rho(\varepsilon_t).\ts_2\right)^{\left(\rho(\varepsilon_t).\ts_1\right)}=\ts_2^{\tt_t},\label{atx}\\
&\rho(\varepsilon_t).\ty:=\left(\rho(\varepsilon_t).\ts_2\right)^{\left(\rho(\varepsilon_t).\ts_3\right)}=\ts_2^{\ts_3}=\ty,\label{aty}\\
&\rho(\iv{\varepsilon_t}).\tx:=\left(\rho(\iv{\varepsilon_t}).\ts_2\right)^{\left(\rho(\iv{\varepsilon_t}).\ts_1\right)}=\ts_2^{\left(\tt_t^{\iv{\ts_1}}\right)}=\tt_t^{\iv{\ts_1}\iv{\ts_2}},\label{atx2}\\
&\rho(\iv{\varepsilon_t}).\ty:=\left(\rho(\iv{\varepsilon_t}).\ts_2\right)^{\left(\rho(\iv{\varepsilon_t}).\ts_3\right)}=\ts_2^{\ts_3}=\ty,\label{aty2}
\end{align}

We show these relations in Appendix~\ref{app:lemmas}. Thus we finish the proof.

%A presentation of a group is called positive if all relations are of the form $u=v$, where $u$ and $v$ are nonempty positive words of the generators. Although the presentation in the above theorem is not positive, it is easy to deduce a positive presentation from that.

%\begin{remark}
%Easy calculations show that the above theorem implies a positive finite presentation of $\BT(\surfO)$ as follows.
%\begin{itemize}
%	\item Generators: $\sigma_i$, $1\leq i\leq \aleph-1$, $\tau_r$, $1\leq r\leq 2g+b-1$.
%	\item Relations:
%	\[\begin{array}{ll}
%	\Co(\sigma_i,\sigma_j)&  \text{if $|i-j|>1$}\\
%	\Br(\sigma_i,\sigma_{i+1})&\text{if $i\neq \aleph-1$}\\
%	\Co(\tau_r,\sigma_i)&\text{if $i>2$}\\
%	\sigma_1\tau_r\sigma_2\sigma_1\tau_r\sigma_2=\sigma_2\sigma_1\tau_r\sigma_2\sigma_1\tau_r&\\
%	\Br(\tau_r,\sigma_2)&\\
%	\Co({\tau_r}^y,{\tau_s}^x)&\text{if $s<r$ and $s\notin\dgen$}\\
%	\Co({\tau_r}^{\iv{y}},{\tau_s}^x)&\text{if $s<r$ and $s\in\dgen$}
%	\end{array}\]
%	where $1\leq i,j\leq \aleph-1$ and $1\leq r,s\leq 2g+b-1$.
%\end{itemize}
%\end{remark}

%=========================================================
\section{Finite presentations for spherical twist groups}\label{sec:st}
%=========================================================

In this section, we introduce the braid group associated to a quiver with potential arising from a marked surface and show that mutations of quivers with potential induce isomorphisms of the associated braid groups. Then we show that such braid group is isomorphic to the braid twist group of the corresponding decorated marked surface. Thus, we get finite presentations of braid/spherical twist groups via quivers with potential.

%=========================================================
\subsection{Braid groups associated to quivers with potential}
%=========================================================

Let $(Q,W)$ be a quiver with potential arising from a triangulated marked surface, that is, $(Q,W)=(Q_\TT,W_\TT)$ for a triangulation $\TT$ of a decorated marked surface $\surfo$. By the construction (see Section~\ref{sec:MS}), the quiver with potential $(Q,W)$ has the following properties.
\begin{itemize}
	\item For each vertex of $Q$, there are	at most two arrows starting at it and at most two arrows ending at it. In particular, there are at most two arrows between two vertices.
	\item There are no 2-cycles in $Q$. So if there are two arrows between two vertices, then they have the same directions, which are called a double arrow.
	%\item There are at most two arrows between two vertices.	
	\item Any two 3-cycles in $W$ do not share an arrow.
	\item If there is a 3-cycle between vertices $a,b,c$, then there is at most one double arrow between them and there is exactly one 3-cycle between them contributing a term in $W$, that is, the full subquiver with potential between these vertices is the first one or the second one in Figure~\ref{fig:local}, where the gray triangle means that its three edges form a 3-cycle in $W$.
\end{itemize}

\begin{definition}\label{def:BrQP}
Let $(Q,W)$ be a quiver with potential arising from a triangulated marked surface. We define the associated braid group $\Br(Q,W)$ by the following presentation:
\begin{itemize}
	\item Generators: vertices of $Q$.
	\item Relations (see Section~\ref{sec:notation} for the notation):
	\numbers
	\item $\Co(a,b)$, if there is no arrow between $a$ and $b$.
	\item $\Br(a,b)$, if there is exactly one arrow between $a$ and $b$.
	\item $\Co(a^b,c)$, if there is one 3-cycle between $a$, $b$ and $c$ which contributes a term in $W$ and there are no double arrows between them, see the first full subquiver with potential in Figure~\ref{fig:local}.
	\item $\Br(a^b,c)$, if there is a 3-cycle between $a$, $b$ and $c$ which contributes a term in $W$ and there is a double arrow between $b$ and $c$, see the second full subquiver with potential in Figure~\ref{fig:local}.
	\item $\Co(c^{ae},b)$, if there is one 3-cycle between $a$, $b$ and $c$ and one 3-cycle between $e$, $b$ and $c$, which contribute terms in $W$, and there are no arrows between $a$ and $e$, see the third full subquiver with potential in Figure~\ref{fig:local}.
	%the full subquiver between $a,e,b,c$ is
	%the left quiver in \eqref{eq:mut3} and one 3-cycle
	%contribute terms in $W_\TT$.%, in such a way that the arrows $b\to c$ in these two 3-cycles are different.
	\item $\Br(c^{ae},b)$ and $\Br(c^{ea},b)$,
	if there is one 3-cycle between $a$, $b$ and $c$ and one 3-cycle between $e$, $b$ and $c$, which contribute terms in $W$, and there is an arrow between $a$ and $e$, see the fourth full subquiver with potential in Figure~\ref{fig:local}.
	\item $\Co(e,f^{abc})$, if in the previous case, additionally there is a 3-cycle between $a$, $e$ and $f$ which contributes a term in $W$, see the last full subquiver with potential in Figure~\ref{fig:local}.
	%as in the left quiver in \eqref{eq:mut5}, which
	%also contributes a term in $W_\TT$, then there is an additional relation $\Co(e,f^{abc})$.
	\ends
\end{itemize}
%the generators are identified with the vertices of $Q_\TT$, and thus with the open arcs in $\TT$, and the relations are

\end{definition}

\begin{figure}
\begin{tikzpicture}[scale=0.5,
arrow/.style={->,>=stealth},
equalto/.style={double,double distance=2pt},
mapto/.style={|->}]
\node (x1) at (0,2){};
\node (x2) at (2,-1){};
\node (x3) at (-2,-1){};
\draw[white, fill=gray!11] (2,-1)to(-2,-1)to(0,2)to(2,-1);
\node at (x1){$a$};
\node at (x2){$b$};
\node at (x3){$c$};
\foreach \n/\m in {1/2,2/3,3/1}
{\draw[arrow]
	(x\n) to (x\m);}
\end{tikzpicture}
\begin{tikzpicture}[scale=0.5,
arrow/.style={->,>=stealth},
equalto/.style={double,double distance=2pt},
mapto/.style={|->}]
\node (x2) at (0,2){};
\node (x1) at (2,-1){};
\node (x3) at (-2,-1){};
\draw[white, fill=gray!11] (2,-.8)to(-2,-.8)to(0,2)--cycle;
\node at (x1){$b$};
\node at (x2){$a$};
\node at (x3){$c$};
\foreach \n/\m in {3/2,2/1}
{\draw[arrow]
	(x\n) to (x\m);}
\draw[arrow] (x1.150) to (x3.30);
\draw[arrow] (x1.-150)to (x3.-30);
\end{tikzpicture}
\begin{tikzpicture}[xscale=0.5,yscale=.5,
arrow/.style={->,>=stealth},
equalto/.style={double,double distance=2pt},
mapto/.style={|->}]
\node (x4) at (2,2){};
\node (x1) at (2,-1){};
\node (x3) at (-2,-1){};
\node (x2) at (-2,2){};
\draw[white, fill=gray!11] (2,-1.2)to(-2.3,-1.2)to(2,2)--cycle;
\draw[white, fill=gray!11] (1.85,-.82)to(-2,-.82)to(-2,2)--cycle;
\node at (x1){$b$};
\node at (x2){$a$};
\node at (x3){$c$};
\node at (x4){$e$};
%\draw[white, fill=gray!11] (2+8,-.8)to(-2+8,-.8)to(2+8,2)--cycle;
\foreach \n/\m in {3/2,2/1,3/4,4/1}
{\draw[arrow]
	(x\n) to (x\m);}
\draw[arrow] (x1.150) to (x3.30);
\draw[arrow] (x1.-150)to (x3.-30);
\end{tikzpicture}
\begin{tikzpicture}[xscale=0.5,yscale=.5,
arrow/.style={->,>=stealth},
equalto/.style={double,double distance=2pt},
mapto/.style={|->}]
\node (x4) at (2,2){};
\node (x1) at (2,-1){};
\node (x3) at (-2,-1){};
\node (x2) at (-2,2){};
\draw[white, fill=gray!11] (2,-.8)to(-2,-.8)to(2,2)--cycle;
\draw[white, fill=gray!11] (2,-1.2)to(-2,-1.2)to(-2,2)--cycle;
\node at (x1){$b$};
\node at (x2){$a$};
\node at (x3){$c$};
\node at (x4){$e$};
\foreach \n/\m in {3/2,2/1,3/4,4/1,2/4}
{\draw[arrow]
	(x\n) to (x\m);}
\draw[arrow] (x1.150) to (x3.30);
\draw[arrow] (x1.-150)to (x3.-30);
\end{tikzpicture}
\begin{tikzpicture}[xscale=0.5,yscale=.5,
arrow/.style={->,>=stealth},
equalto/.style={double,double distance=2pt},
mapto/.style={|->}]
\node (x4) at (2,2){};
\node (x1) at (2,-1){};
\node (x3) at (-2,-1){};
\node (x2) at (-2,2){};
\node (x5) at (0,3){};
\draw[white, fill=gray!11] (2,2)to(-2,2)to(0,3)--cycle;
\draw[white, fill=gray!11] (2,-1)to(-2,-1)to(2,2)--cycle;
\draw[white, fill=gray!11] (2,-1)to(-2,-1)to(-2,2)--cycle;
\node at (x1){$b$};
\node at (x2){$a$};
\node at (x3){$c$};
\node at (x4){$e$};
\node at (x5){$f$};
\foreach \n/\m in {3/2,2/1,3/4,4/1,5/2,4/5,2/4}
{\draw[arrow]
	(x\n) to (x\m);}
\draw[arrow] (x1.150) to (x3.30);
\draw[arrow] (x1.-150)to (x3.-30);
\end{tikzpicture}
\caption{Cases of full subquiver of $(Q,W)$}
\label{fig:local}
\end{figure}

%\newpage

%The number of relations above is not minimal in general, see Lemma~\ref{lem:a}, Lemma~\ref{lem:mut3} and Lemma~\ref{lem:mut5} for minimal relations in some special cases.
The above presentation will become much simpler if there are no double arrows in the quiver $Q$. Note that excluding the case that $\surf$ is an annulus with one marked point on each of its boundary components or a torus with only one marked point, there always exists a triangulation $\TT$ such that there are no double arrows in the associated quiver $Q_\TT$ (see \cite[Lemma~3.12]{QQ}).

\begin{proposition}\label{prop:simple}
	Let $(Q,W)$ be a quiver with potential arising from a triangulated marked surface. If there are no double arrows in $Q$, then the associated braid group $\Br(Q,W)$ has the following presentation:
	\begin{itemize}
		\item Generators: vertices of $Q$.
		\item Relations:
		\numbers
		\item[$1^\circ$.] $\Co(a,b)$ if there is no arrow between $a$ and $b$.
		\item[$2^\circ$.] $\Br(a,b)$ if there is exactly one arrow between $a$ and $b$.
		\item[${3^\circ}'$.] $\Tr(a,b,c)$ if there is a 3-cycle between $a$, $b$ and $c$, which contributes a term in $W$.
		\ends
	\end{itemize}
\end{proposition}

\begin{proof}
	Since there are no double arrows in $Q$, the relations $4^\circ$--$7^\circ$ in Definition~\ref{def:BrQP} does not occur. So we only need to show that the relation $3^\circ$ is equivalent to ${3^\circ}'$ under the relations $1^\circ$ and $2^\circ$. This equivalence follows from the following easy equivalences:
	\[%abca=bcab\Longleftrightarrow
	abca=bcab\Longleftrightarrow\iv{b}abc=cab\iv{a}
	\Longleftrightarrow a^bc=cb^{\iv{a}}\xLongleftrightarrow{}\Co(a^b,c)\]
	where the last equivalence follows from $b^{\iv{a}}=a^b$ by $\Br(a,b)$.
\end{proof}

\begin{remark}
Sometime we prefer this triangle relation $\Tr(a,b,c)$,
since it can be generalized to the cyclic relation, which may be used for the quivers with potential which contains cycles of lengths at least 4 in the potential (cf. \cite{GM} and \cite[Definition~10.1]{QQ}).
%Note that if there are no double arrows in the quiver $Q_\TT$, then the presentation of the braid group $\Br(Q_\TT,W_\TT)$ becomes:

%only the first three types of relations (commutation/braid/triangle) occur. %When the marked surface $\surf$ is a disk, the quiver $Q_\TT$ does not have double arrows for any triangulation $\TT$. But in general, this is not true.
\end{remark}

The relations in Definition~\ref{def:BrQP} are not minimal. See the following two simple cases, where less relations are enough.

\begin{lemma}\label{lem:a}
	Let $(Q,W)$ be the first quiver with potential in Figure~\ref{fig:local}. Then $\Br(Q,W)$ admits the following presentation.
	\begin{itemize}
		\item Generators: $a,b,c$.
		\item Relations:
		\[\Br(a,b),\ \Br(a,c),\ \Co(a^b,c),\]
	\end{itemize}
	where $\Co(a^b,c)$ can be replaced by $\Co(c^a,b)$.
\end{lemma}

\begin{proof}
	By the braid relation $\Br(a,b)$, we have $a^b=b^{\iv{a}}$. Then $\Co(a^b,c)$ is equivalent to $\Co(b^{\iv{a}},c)$. The latter is equivalent to $\Co(b,c^a)$ via conjugation by $a$. Thus, $\Co(a^b,c)$ can be replaced by $\Co(c^a,b)$.
	
	We need to show the relations $\Br(b,c)$ and $\Co(a,b^c)$. Starting from $\Br(a,c)$, conjugated by $\iv{a}^b$ and using the commutation relation $\Co(a^b,c)$, we have $\Br(a^{\left(\iv{a}^b\right)},c)$. Then using the braid relation $\Br(a,b)$, we get $\Br(b,c)$. To show the other one, starting from $\Co(a^b,c)$, conjugated by $\iv{b}$ and using the braid relation $\Br(b,c)$, we are done.
	
	%the implication $\Co(a^b,c)\Longrightarrow$
	%First, we have
	%\[\Co(a^b,c)
	%\xLongleftrightarrow{\Br(a,b)}  \Co(b^{\iv{a}},c)
	%\Longleftrightarrow  \Co(b,c^a).\]
	%Moreover,
	%\[\Br(a,c)\xLongrightarrow[\Co(a^b,c)]{?^{\iv{a}^b}} \Br(a^{\left(\iv{a}^b\right)},c)\xLongrightarrow{\Br(a,b)}\Br(b,c) \]
	%and
	%\[\Co(a^b,c)
	%\Longrightarrow \Co(a,c^{\iv{b}})
	%\xLongrightarrow{\Br(b,c)} \Co(a,b^c)\]
	%which complete the proof.
\end{proof}

\begin{lemma}\label{lem:mut3}
	Let $(Q,W)$ be the third quiver with potential in Figure~\ref{fig:local}. Then $\Br(Q,W)$ admits the following presentation.
	\begin{itemize}
		\item Generators: $a,b,c,e$.
		\item Relations:
		\[\Br(a,b),\ \Br(a,c),\ \Br(e,b),\ \Br(e,c), \ \Co(a,e), \ \Co(c^{ae},b).\]
	\end{itemize}
\end{lemma}
\begin{proof}
	We shall prove that the relations $\Br(a^b,c)$ and $\Br(e^b,c)$ hold. Indeed, %we have the following implications, where the relations used for each implication are labeled.
	starting from $\Br(e,b)$, conjugated by $c^{ae}$ and using the commutation relation $\Co(c^{ae},b)$, we get $\Br(e^{\left(c^{ae}\right)},b)$. By the commutation relation $\Co(a,e)$, we have $\Br(e^{cea},b)$, which implies $\Br(c^a,b)$ by using the braid relation $\Br(e,c)$. Conjugated by $\iv{a}$ and using the braid relation $\Br(a,b)$, we deduce $\Br(c,a^b)$ as required. The relation $\Br(c,e^b)$ can be showed similarly.
	%\[\Br(e,b)
	%\xLongrightarrow{\Co(c^{ae},b)} \Br(e^{\left(c^{ae}\right)},b)
	%\xLongrightarrow{\Co(a,e)} \Br(e^{cea},b)
	%\xLongrightarrow{\Br(e,c)} \Br(c^a,b)
	%\xLongrightarrow{\Br(a,b)} \Br(c,a^b).
	%\]
	%and
	%\[\Br(a,b)
	%\xLongrightarrow{\Co(c^{ae},b)} \Br(a^{\left(c^{ae}\right)},b)
	%\xLongrightarrow{\Co(a,e)} \Br(a^{cae},b)
	%\xLongrightarrow{\Br(a,c)} \Br(c^e,b)
	%\xLongrightarrow{\Br(e,b)} \Br(c,e^b).
	%\]
\end{proof}

Another case for which less relations are enough is the last case in Figure~\ref{fig:local}; see Lemma~\ref{lem:mut5}, whose proof is a little complicated.

\subsection{Mutations and isomorphisms}

In this subsection, we show that the presentations in Definition~\ref{def:BrQP} are compatible with mutation of quivers with potential. Recall from Section~\ref{sec:MS} the notion of mutation of quivers with potential.

%Let $\TT$ be a triangulation of $\surfo$ and $\TT'$ the backward flip of $\TT$ at $\gamma\in\TT$, i.e. $\TT'=\mu_\gamma^\flat(\TT)$ (cf. Figure~\ref{fig:ex0}). By definition, $\TT$ is the forward flip of $\TT'$ at $\gamma$, i.e. $\TT=\mu_\gamma(\TT')$. For any open arc $\alpha\in\TT$, denote by $\alpha'$ the corresponding arc in $\TT'$.

%Each closed arc $\alpha\in\TT^\ast$ corresponds to a closed arc $\alpha'\in\TT'^\ast$ by a Whitehead move, cf. \cite{QQ}. Denote by $\gamma\in\TT^\ast$ the dual of $\eta$. Then the quiver with potential $(Q_{\TT'},W_{\TT'})$ is the mutation of $(Q_\TT,W_\TT)$ at the vertex $\gamma$.
%Note that $\mu_\gamma^\flat(\TT)$ and $\mu_\gamma^\sharp(\TT)$ give the same quiver with potential, although they are different triangulation.
%The generators of $\Br(Q_\TT,W_\TT)$ will be denoted by $\alpha$ for $\alpha \in \TT$ and
%the generators of $\Br(Q_{\TT'},W_{\TT'})$ will be denoted by $\alpha'$ for $\alpha \in {\TT'}$.

\begin{proposition}\label{pp:invariant}
Let $(Q,W)$ be a quiver with potential arising from a triangulated marked surface and $(Q',W')=\mu_\gamma(Q,W)$ the mutation of $(Q,W)$ at a vertex $\gamma$. Denote by $\alpha'$ the vertex of $Q'$ corresponding to a vertex $\alpha$ of $Q$. There are mutually inverse canonical isomorphisms of groups
\[\begin{array}{cl}
\theta_{\gamma}^\flat\colon&\Br(Q,W)\xrightarrow{\cong}\Br(Q',W')\\
\theta_{\gamma'}^\sharp\colon&\Br(Q',W')\xrightarrow{\cong}\Br(Q,W)
\end{array}
\]
%between the braid groups associated to the quivers with potential $(Q_\TT,W_\TT)$ and $(Q_{\TT'},W_{\TT'})$,
satisfying
\begin{gather}
\begin{array}{ccl}\label{eq:5.1}
\theta_{\gamma}^\flat(\alpha)&=&\begin{cases}
(\alpha')^{\gamma'} & \text{if there are arrows from $\gamma$ to $\alpha$ in $Q$,}\\
\alpha' & \text{otherwise,}
\end{cases}\end{array}\\
\begin{array}{ccl}\label{eq:5.2}
\theta_{\gamma'}^\sharp(\alpha')&=&\begin{cases}
(\alpha)^{\underline{\gamma}} & \text{if there are arrows from $\alpha'$ to $\gamma'$ in $Q'$,}\\
\alpha & \text{otherwise.}
\end{cases}\end{array}
\end{gather}
\end{proposition}

%\newpage

\begin{proof}
The formulas \eqref{eq:5.1} and \eqref{eq:5.2} defines two homomorphisms $\theta_{\gamma}^\sharp$ and $\theta_{\gamma'}^\flat$ between the groups freely generated by the vertices of $Q$ and $Q'$, respectively. Moreover, $\theta_{\gamma'}^\sharp$ is the composition of $\theta_{\gamma'}^\flat$ with the conjugation by $\iv{\gamma}$, and $\theta_{\gamma}^\flat$ and $\theta_{\gamma'}^\sharp$ are mutually inverse.
%Note that if there is exactly one arrow from $\gamma$ to $\alpha$,
%we have $\Br(\gamma,\alpha)$ and $\theta^{-1}(\alpha')=\alpha^{\iv{\gamma}}=\gamma^{\alpha}$.
So we only need to show $\theta=\theta_{\gamma}^\flat$ preserves the relations in the presentation of $\Br(Q,W)$ in Definition~\ref{def:BrQP}.

First, we consider several special (local) cases, where we show that the images of the relations in $\Br(Q',W')$ under $\theta^{-1}$ are equivalent to the relations in $\Br(Q,W)$. In each case, the left quiver with potential is (a full sub-quiver of) $(Q,W)$ and the right quiver with potential is (the corresponding full sub-quiver of) $(Q',W')$, but whose vertices are indexed by their images under $\theta^{-1}$. The relations common to both sides are listed on the top and an equivalence may use common relations and equivalences above it. We only give proofs in Appendix~\ref{app:cal} for non-easily checked equivalences.

\begin{description}
\item[Case (I)] Consider the mutation in \eqref{eq:mut1}. Lemma~\ref{lem:a} ensures that the relations on the left side are enough.
\end{description}
%In the right quiver, $a=\theta^{-1}(a')$, $c=\theta^{-1}(c')$ and $d=\theta^{-1}(b')$ (and similar for the later cases).
\begin{gather}\label{eq:mut1}
\begin{tikzpicture}[scale=0.5,
  arrow/.style={->,>=stealth},
  equalto/.style={double,double distance=2pt},
  mapto/.style={|->}]
\node (x1) at (0,2){};
\node (x2) at (2,-1){};
\node (x3) at (-2,-1){};
\draw[white, fill=gray!11] (2,-1)to(-2,-1)to(0,2)to(2,-1);
  \node at (x1){$a$};
  \node at (x2){$b$};
  \node at (x3){$c$};
  \foreach \n/\m in {1/2,2/3,3/1}
    {\draw[arrow]
     (x\n) to (x\m);}
\draw(4,0)node[above]{$\mu_a$}node[below]{\tiny{$d=b^{\iv{a}}$}};
\draw[arrow,decorate, decoration={snake}](3,0)to(5,0);
\node (x1) at (0+8,2){};
\node (x2) at (2+8,-1){};
\node (x3) at (-2+8,-1){};
  \node at (x1){$a$};
  \node at (x2){$d$};
  \node at (x3){$c$};
  \foreach \n/\m in {2/1,1/3}
    {\draw[arrow]
     (x\n) to (x\m);}
\draw(16,1)node{
$\begin{array}{ccl}
&\Br(a,c)&\\
\Br(a,b)&\Longleftrightarrow&\Br(a,d)\\
\Co(a^b,c)&\Longleftrightarrow&\Co(c,d)\\
\end{array}$
};
\end{tikzpicture}
\end{gather}
%Note that in \textbf{Case (I)}, we list a minimal set of relations. Lemma~\ref{lem:a} ensures that they are enough.
%One also see that the relations in Definition~\ref{def:BrQP} is clearly not minimal by all means.

\begin{description}
\item[Case (II.1/2)] Consider the mutations in \eqref{eq:mut2} and \eqref{eq:mut2'} respectively.
\end{description}
%It is straightforward to check that Proposition~\ref{pp:invariant} holds.
%Note Lemma~\ref{lem:a} ensures that these relations are enough.
\begin{gather}\label{eq:mut2}
\begin{tikzpicture}[scale=0.5,
  arrow/.style={->,>=stealth},
  equalto/.style={double,double distance=2pt},
  mapto/.style={|->}]
\node (x2) at (0,2){};
\node (x1) at (2,-1){};
\node (x3) at (-2,-1){};
\draw[white, fill=gray!11] (2,-.8)to(-2,-.8)to(0,2)--cycle;
  \node at (x1){$b$};
  \node at (x2){$a$};
  \node at (x3){$c$};
  \foreach \n/\m in {3/2,2/1}
    {\draw[arrow]
     (x\n) to (x\m);}
\draw[arrow] (x1.150) to (x3.30);
\draw[arrow] (x1.-150)to (x3.-30);
\draw(4,0)node[above]{$\mu_a$}node[below]{\tiny{$d=b^{\iv{a}}$}};
\draw[arrow,decorate, decoration={snake}](3,0)to(5,0);
\node (x1) at (0+8,2){};
\node (x2) at (2+8,-1){};
\node (x3) at (-2+8,-1){};
  \node at (x1){$a$};
  \node at (x2){$d$};
  \node at (x3){$c$};
  \foreach \n/\m in {2/1,1/3,2/3}
    {\draw[arrow]
     (x\n) to (x\m);}
\draw(16,1)node{
$\begin{array}{ccl}
&\Br(a,c)&\\
\Br(a,b)&\Longleftrightarrow&\Br(a,d)\\
\Br(a^b,c)&\Longleftrightarrow&\Br(d,c)\\
\end{array}$
};
\end{tikzpicture}\\
\label{eq:mut2'}
\begin{tikzpicture}[scale=0.5,
  arrow/.style={->,>=stealth},
  equalto/.style={double,double distance=2pt},
  mapto/.style={|->}]
\node (x2) at (0,2){};
\node (x1) at (2,-1){};
\node (x3) at (-2,-1){};
\draw[white, fill=gray!11] (2,-.8)to(-2,-.8)to(0,2)--cycle;
\draw[white, fill=gray!11] (2+8,-.8)to(-2+8,-.8)to(0+8,2)--cycle;
  \node at (x1){$b$};
  \node at (x2){$a$};
  \node at (x3){$c$};
  \foreach \n/\m in {3/2,2/1}
    {\draw[arrow]
     (x\n) to (x\m);}
\draw[arrow] (x1.150) to (x3.30);
\draw[arrow] (x1.-150)to (x3.-30);
\draw(4,0)node[above]{$\mu_b$}node[below]{\tiny{$l=c^{\iv{b}}$}};
\draw[arrow,decorate, decoration={snake}](3,0)to(5,0);
\node (x1) at (0+8,2){};
\node (x2) at (2+8,-1){};
\node (x3) at (-2+8,-1){};
  \node at (x1){$a$};
  \node at (x2){$b$};
  \node at (x3){$l$};
  \foreach \n/\m in {2/1,1/3}
    {\draw[arrow]
     (x\n) to (x\m);}
\draw[arrow] (x3.30)to(x2.150);
\draw[arrow] (x3.-30)to(x2.-150);
\draw(16,1)node{
$\begin{array}{ccl}
&\Br(a,b)&\\
\Br(a^b,c)&\Longleftrightarrow&\Br(a,l)\\
\Br(a,c)&\Longleftrightarrow&\Br(a^l,b)\\
\end{array}$
};
\end{tikzpicture}
\end{gather}

\begin{description}
\item[Case (III.1/2)] Consider the mutations in \eqref{eq:mut3} and \eqref{eq:mut3'} respectively.
\end{description}
For the mutation in \eqref{eq:mut3},
Lemma~\ref{lem:mut3} and Lemma~\ref{lem:a} ensure the relations on the left side and on the right side are enough, respectively.
\begin{gather}\label{eq:mut3}
\begin{tikzpicture}[xscale=0.5,yscale=.5,
  arrow/.style={->,>=stealth},
  equalto/.style={double,double distance=2pt},
  mapto/.style={|->}]
\node (x4) at (2,2){};
\node (x1) at (2,-1){};
\node (x3) at (-2,-1){};
\node (x2) at (-2,2){};
\draw[white, fill=gray!11] (2,-.8)to(-2,-.8)to(2,2)--cycle;
\draw[white, fill=gray!11] (2,-1.2)to(-2,-1.2)to(-2,2)--cycle;
  \node at (x1){$b$};
  \node at (x2){$a$};
  \node at (x3){$c$};
  \node at (x4){$e$};
\draw[white, fill=gray!11] (2+8,-.8)to(-2+8,-.8)to(2+8,2)--cycle;
  \foreach \n/\m in {3/2,2/1,3/4,4/1}
    {\draw[arrow]
     (x\n) to (x\m);}
\draw[arrow] (x1.150) to (x3.30);
\draw[arrow] (x1.-150)to (x3.-30);
\draw(4,0)node[above]{$\mu_a^\flat$}node[below]{\tiny{$d=b^{\iv{a}}$}};
\draw[arrow,decorate, decoration={snake}](3,0)to(5,0);
\node (x4) at (2+8,2){};
\node (x2) at (2+8,-1){};
\node (x3) at (-2+8,-1){};
\node (x1) at (-2+8,2){};
  \node at (x1){$a$};
  \node at (x2){$d$};
  \node at (x3){$c$};
  \node at (x4){$e$};
  \foreach \n/\m in {2/1,1/3,2/3,3/4,4/2}
    {\draw[arrow]
     (x\n) to (x\m);}
\draw(17,.5)node{
$\begin{array}{ccl}
&\Co(a,e)\\&\Br(a,c)\\&\Br(e,c)\\
\Br(e,b)&\Longleftrightarrow&\Br(e,d)\\
\Br(a,b)&\Longleftrightarrow&\Br(a,d)\\
\Co(c^{ae},b)&\Longleftrightarrow&\Co(e^d,c)\\
\end{array}$};
\end{tikzpicture}
\end{gather}
For the mutation in \eqref{eq:mut3'}, we show the last three equivalences in Lemma~\ref{lem:mut6}.
\begin{gather}\label{eq:mut3'}
\begin{tikzpicture}[xscale=0.5,yscale=.5,
  arrow/.style={->,>=stealth},
  equalto/.style={double,double distance=2pt},
  mapto/.style={|->}]
\node (x4) at (2,2){};
\node (x1) at (2,-1){};
\node (x3) at (-2,-1){};
\node (x2) at (-2,2){};
\draw[white, fill=gray!11] (2,-.8)to(-2,-.8)to(2,2)--cycle;
\draw[white, fill=gray!11] (2,-1.2)to(-2,-1.2)to(-2,2)--cycle;
  \node at (x1){$b$};
  \node at (x2){$a$};
  \node at (x3){$c$};
  \node at (x4){$e$};
  \foreach \n/\m in {3/2,2/1,3/4,4/1}
    {\draw[arrow]
     (x\n) to (x\m);}
\draw[arrow] (x1.150) to (x3.30);
\draw[arrow] (x1.-150)to (x3.-30);
\draw(4,0)node[above]{$\mu_b$}node[below]{\tiny{$l=c^{\iv{b}}$}};
\draw[arrow,decorate, decoration={snake}](3,0)to(5,0);
\node (x4) at (2+8,2){};
\node (x1) at (2+8,-1){};
\node (x3) at (-2+8,-1){};
\node (x2) at (-2+8,2){};
\node (x5) at (0+8,3){};
\draw[white, fill=gray!11] (2+8,-1)to(-2+8,-1)to(2+8,2)--cycle;
\draw[white, fill=gray!11] (2+8,-1)to(-2+8,-1)to(-2+8,2)--cycle;
  \node at (x1){$b$};
  \node at (x2){$a$};
  \node at (x3){$l$};
  \node at (x4){$e$};
  \foreach \n/\m in {3/2,2/1,3/4,4/1}
    {\draw[<-,>=stealth]
     (x\n) to (x\m);}
\draw[<-,>=stealth] (x1.150) to (x3.30);
\draw[<-,>=stealth] (x1.-150)to (x3.-30);
\draw(17,1)node{
$\begin{array}{ccl}
%\Co(f,b)&&\Co(f,c)\\
%\Br(a,f)&&\Br(e,f)\\
&\Br(a,b)\\&\Co(a,e)\\&\Br(b,e)\\
\Br(a^b,c)&\Longleftrightarrow&\Br(a,l)\\
\Br(e^b,c)&\Longleftrightarrow&\Br(e,l)\\
\Br(a,c)&\Longleftrightarrow&\Br(a^l,b)\\
\Br(e,c)&\Longleftrightarrow&\Br(e^l,b)\\
\Co(c^{ae},b)&\Longleftrightarrow&\Co(b^{ae},l)
\end{array}$};
\end{tikzpicture}
\end{gather}

%\begin{lemma}
%If $\Co(a,e),\Br(a,b)$ and $\Br(e,b)$ hold,
%then, $\Co(c^a,e^b)\Longleftrightarrow\Co(c^e,a^b)$.
%\end{lemma}
%\begin{proof}
%Since, $\Br(a,b)\Leftrightarrow a^b=b^{\iv{a}}$, $\Br(e,b)\Leftrightarrow e^b=b^{\iv{e}}$, we have
%\[\Co(c^a,e^b)\Leftrightarrow
%\Co(c^a,b^{\iv{e}})\Longleftrightarrow
%\Co(c^{ae},b)\xLongleftrightarrow{\Co(a,e)}
%\Co(c^{ea},b)
%\Longleftrightarrow\Co(c^e,b^{\iv{a}})
%\Leftrightarrow\Co(c^e,a^b).\]
%\end{proof}

\begin{description}
\item[Case (IV.1/2)] Consider the mutations in \eqref{eq:mut4} and \eqref{eq:mut4'} respectively.
\end{description}
For the mutation in \eqref{eq:mut4}, we show the last three equivalences in Lemma~\ref{lem:mut4}.
\begin{gather}\label{eq:mut4}
\begin{tikzpicture}[xscale=0.5,yscale=.5,
  arrow/.style={->,>=stealth},
  equalto/.style={double,double distance=2pt},
  mapto/.style={|->}]
\node (x4) at (2,2){};
\node (x1) at (2,-1){};
\node (x3) at (-2,-1){};
\node (x2) at (-2,2){};
\draw[white, fill=gray!11] (2,-1)to(-2,-1)to(2,2)--cycle;
\draw[white, fill=gray!11] (2,-1)to(-2,-1)to(-2,2)--cycle;
  \node at (x1){$b$};
  \node at (x2){$a$};
  \node at (x3){$c$};
  \node at (x4){$e$};
  \foreach \n/\m in {3/2,2/1,3/4,4/1,2/4}
    {\draw[arrow]
     (x\n) to (x\m);}
\draw[arrow] (x1.150) to (x3.30);
\draw[arrow] (x1.-150)to (x3.-30);
\draw(4,0)node[above]{$\mu_e$}node[below]{\tiny{$h=b^{\iv{e}}$}};
\draw[arrow,decorate, decoration={snake}](3,0)to(5,0);
\node (x4) at (2+8,2){};
\node (x1) at (2+8,-1){};
\node (x3) at (-2+8,-1){};
\node (x2) at (-2+8,2){};
\draw[white, fill=gray!11] (2+8,-1+.2)to(-2+8,2+.2)to(2+8,2)--cycle;
\draw[white, fill=gray!11] (2+8-.2,-1)to(-2+8,-1)to(-2+8,2-.2)--cycle;
  \node at (x1){$h$};
  \node at (x2){$a$};
  \node at (x3){$c$};
  \node at (x4){$e$};
  \foreach \n/\m in {1/3,1/4,3/2,4/2,4/3}
    {\draw[arrow]
     (x\n) to (x\m);}
\draw[arrow] (x2.-15) to (x1.105);
\draw[arrow] (x2.-60)to (x1.165);
\draw(17,1)node{
$\begin{array}{ccl}
&\Br(a,e)\\&\Br(a,c)\\&\Br(e,c)\\
\Br(e,b)&\Longleftrightarrow&\Br(e,h)\\
\Br(a,b)&\Longleftrightarrow&\Br(e^a,h)\\
\Br(e^b,c)&\Longleftrightarrow&\Br(h,c)\\
\Br(a^b,c)&\Longleftrightarrow&\Br(h^{ec},a)\\
\Br(c^{ae},b)&\Longleftrightarrow&\Br(c^a,h)\\
\Br(c^{ea},b)&\Longleftrightarrow&\Br(h^{ce},a)\\
\end{array}$};
%\draw(17,-1)node[below,gray!80]{$\Ss(a,b,c)\Longleftrightarrow\Br(c,d)\;\;$
\end{tikzpicture}
\end{gather}
For the mutation in \eqref{eq:mut4'}, we show the last two equivalences in Lemma~\ref{lem:mut4'}.
\begin{gather}\label{eq:mut4'}
\begin{tikzpicture}[xscale=0.5,yscale=.5,
  arrow/.style={->,>=stealth},
  equalto/.style={double,double distance=2pt},
  mapto/.style={|->}]
\node (x4) at (2,2){};
\node (x1) at (2,-1){};
\node (x3) at (-2,-1){};
\node (x2) at (-2,2){};
\draw[white, fill=gray!11] (2,-1)to(-2,-1)to(2,2)--cycle;
\draw[white, fill=gray!11] (2,-1)to(-2,-1)to(-2,2)--cycle;
  \node at (x1){$b$};
  \node at (x2){$a$};
  \node at (x3){$c$};
  \node at (x4){$e$};
  \foreach \n/\m in {3/2,2/1,3/4,4/1,2/4}
    {\draw[arrow]
     (x\n) to (x\m);}
\draw[arrow] (x1.150) to (x3.30);
\draw[arrow] (x1.-150)to (x3.-30);
\draw(4,0)node[above]{$\mu_b$}node[below]{\tiny{$l=c^{\iv{b}}$}};
\draw[arrow,decorate, decoration={snake}](3,0)to(5,0);
\node (x4) at (2+8,2){};
\node (x1) at (2+8,-1){};
\node (x3) at (-2+8,-1){};
\node (x2) at (-2+8,2){};
\node (x5) at (0+8,3){};
\draw[white, fill=gray!11] (2+8,-1)to(-2+8,-1)to(2+8,2)--cycle;
\draw[white, fill=gray!11] (2+8,-1)to(-2+8,-1)to(-2+8,2)--cycle;
  \node at (x1){$b$};
  \node at (x2){$a$};
  \node at (x3){$l$};
  \node at (x4){$e$};
  \foreach \n/\m in {3/2,2/1,3/4,4/1}
    {\draw[<-,>=stealth]
     (x\n) to (x\m);}
\draw[<-,>=stealth] (x1.150) to (x3.30);
\draw[<-,>=stealth] (x1.-150)to (x3.-30);
\foreach \n/\m in {4/2}
    {\draw[arrow] (x\m)--(x\n);}
\draw(17,1)node{
$\begin{array}{ccl}
&\Br(a,b)\\&\Br(a,e)\\&\Br(e,b)\\
\Br(a^b,c)&\Longleftrightarrow&\Br(a,l)\\
\Br(a,c)&\Longleftrightarrow&\Br(a^l,b)\\
\Br(e^b,c)&\Longleftrightarrow&\Br(e,l)\\
\Br(e,c)&\Longleftrightarrow&\Br(e^l,b)\\
\Br(c^{ae},b)&\Longleftrightarrow&\Br(b^{ae},l)\\
\Br(c^{ea},b)&\Longleftrightarrow&\Br(b^{ea},l)\\
\end{array}$};
\end{tikzpicture}
\end{gather}

\begin{description}
\item[Case (V.1$\sim$4)]
Consider the mutations in \eqref{eq:mut5}, \eqref{eq:mut7}
\eqref{eq:mut8} and \eqref{eq:mut9} respectively.
\end{description}
For the mutation in \eqref{eq:mut5}, Lemma~\ref{lem:mut5}, together with Lemma~\ref{lem:a}, ensures that the relations on both sides are enough. We show the last equivalence in Lemma~\ref{lem:mut5'}.
\begin{gather}\label{eq:mut5}
\begin{tikzpicture}[xscale=0.5,yscale=.5,
  arrow/.style={->,>=stealth},
  equalto/.style={double,double distance=2pt},
  mapto/.style={|->}]
\node (x4) at (2,2){};
\node (x1) at (2,-1){};
\node (x3) at (-2,-1){};
\node (x2) at (-2,2){};
\node (x5) at (0,3){};
\draw[white, fill=gray!11] (2,2)to(-2,2)to(0,3)--cycle;
\draw[white, fill=gray!11] (2,-1)to(-2,-1)to(2,2)--cycle;
\draw[white, fill=gray!11] (2,-1)to(-2,-1)to(-2,2)--cycle;
  \node at (x1){$b$};
  \node at (x2){$a$};
  \node at (x3){$c$};
  \node at (x4){$e$};
  \node at (x5){$f$};
  \foreach \n/\m in {3/2,2/1,3/4,4/1,5/2,4/5,2/4}
    {\draw[arrow]
     (x\n) to (x\m);}
\draw[arrow] (x1.150) to (x3.30);
\draw[arrow] (x1.-150)to (x3.-30);
\draw(4,0)node[above]{$\mu_f$}node[below]{\tiny{$g=a^{\iv{f}}$}};
\draw[arrow,decorate, decoration={snake}](3,0)to(5,0);
\node (x4) at (2+8,2){};
\node (x1) at (2+8,-1){};
\node (x3) at (-2+8,-1){};
\node (x2) at (-2+8,2){};
\node (x5) at (0+8,3){};
\draw[white, fill=gray!11] (2+8,-1)to(-2+8,-1)to(2+8,2)--cycle;
\draw[white, fill=gray!11] (2+8,-1)to(-2+8,-1)to(-2+8,2)--cycle;
  \node at (x1){$b$};
  \node at (x2){$g$};
  \node at (x3){$c$};
  \node at (x4){$e$};
  \node at (x5){$f$};
  \foreach \n/\m in {3/2,2/1,3/4,4/1,5/4,2/5}
    {\draw[arrow]
     (x\n) to (x\m);}
\draw[arrow] (x1.150) to (x3.30);
\draw[arrow] (x1.-150)to (x3.-30);
\draw(17,1)node{
$\begin{array}{ccl}
&\Br(f,e)\\&\Br(e,b)\\&\Br(e,c)\\
&\Co(f,b)\\&\Co(f,c)\\
%\Br(e,b)&\Longleftrightarrow&\Br(e,h)\\
\Br(f,a)&\Longleftrightarrow&\Br(f,g)\\
\Br(a,b)&\Longleftrightarrow&\Br(g,b)\\
\Br(a,c)&\Longleftrightarrow&\Br(g,c)\\
\Co(f^a,e)&\Longleftrightarrow&\Co(g,e)\\
\Co(e,f^{abc})&\Longleftrightarrow&\Co(c^{ge},b)\\
\end{array}$};
\end{tikzpicture}\end{gather}
For the mutation in \eqref{eq:mut7},
the equivalences are from \eqref{eq:mut4'}.
\begin{gather}\label{eq:mut7}
\begin{tikzpicture}[xscale=0.5,yscale=.5,
  arrow/.style={->,>=stealth},
  equalto/.style={double,double distance=2pt},
  mapto/.style={|->}]
\node (x4) at (2,2){};
\node (x1) at (2,-1){};
\node (x3) at (-2,-1){};
\node (x2) at (-2,2){};
\draw[white, fill=gray!11] (2,2)to(-2,2)to(0,3)--cycle;
\draw[white, fill=gray!11] (2+8,2)to(-2+8,2)to(0+8,3)--cycle;
\draw[white, fill=gray!11] (2,-.8)to(-2,-.8)to(2,2)--cycle;
\draw[white, fill=gray!11] (2,-1.2)to(-2,-1.2)to(-2,2)--cycle;
  \node at (x1){$b$};
  \node at (x2){$a$};
  \node at (x3){$c$};
  \node at (x4){$e$};
  \foreach \n/\m in {3/2,2/1,3/4,4/1}
    {\draw[arrow]
     (x\n) to (x\m);}
\draw[arrow] (x1.150) to (x3.30);
\draw[arrow] (x1.-150)to (x3.-30);
\draw(0,3)node(x5){$f$};
\foreach \n/\m in {4/2,5/4,2/5}
    {\draw[arrow] (x\m)--(x\n);}
\draw(4,0)node[above]{$\mu_b$}node[below]{\tiny{$l=c^{\iv{b}}$}};
\draw[arrow,decorate, decoration={snake}](3,0)to(5,0);
\node (x4) at (2+8,2){};
\node (x1) at (2+8,-1){};
\node (x3) at (-2+8,-1){};
\node (x2) at (-2+8,2){};
\node (x5) at (0+8,3){};
\draw[white, fill=gray!11] (2+8,-1)to(-2+8,-1)to(2+8,2)--cycle;
\draw[white, fill=gray!11] (2+8,-1)to(-2+8,-1)to(-2+8,2)--cycle;
  \node at (x1){$b$};
  \node at (x2){$a$};
  \node at (x3){$l$};
  \node at (x4){$e$};
  \foreach \n/\m in {3/2,2/1,3/4,4/1}
    {\draw[<-,>=stealth]
     (x\n) to (x\m);}
\draw[<-,>=stealth] (x1.150) to (x3.30);
\draw[<-,>=stealth] (x1.-150)to (x3.-30);
\draw(0+8,3)node(x5){$f$};
\foreach \n/\m in {4/2,5/4,2/5}
    {\draw[arrow] (x\m)--(x\n);}
%\draw(17,1)node{
%$\begin{array}{ccl}
%\Br(a,f)&\Co(f^a,e)&\Br(e,f)\\
%\Br(b,a)&\Br(a,e)&\Br(b,e)\\
%\Co(f,b)&&\Co(f,c)\\
%\Br(a^b,c)&\Longleftrightarrow&\Br(a,l)\\
%\Br(e^b,c)&\Longleftrightarrow&\Br(e,l)\\
%\Br(a,c)&\Longleftrightarrow&\Br(a^l,b)\\
%\Br(e,c)&\Longleftrightarrow&\Br(e^l,b)\\
%\Co(e,f^{abc})&\Longleftrightarrow&\Co(e,f^{alb})
%\end{array}$};
\end{tikzpicture}
\end{gather}
For the mutation in \eqref{eq:mut8}, we show the last equivalence in Lemma~\ref{lem:mut8}.
\begin{gather}\label{eq:mut8}
\begin{tikzpicture}[xscale=0.5,yscale=.5,
  arrow/.style={->,>=stealth},
  equalto/.style={double,double distance=2pt},
  mapto/.style={|->}]
\node (x4) at (2,2){};
\node (x1) at (2,-1){};
\node (x3) at (-2,-1){};
\node (x2) at (-2,2){};
\node (x5) at (0,3){};
\draw[white, fill=gray!11] (2,2)to(-2,2)to(0,3)--cycle;a
\draw[white, fill=gray!11] (2,-.8)to(-2,-.8)to(2,2)--cycle;
\draw[white, fill=gray!11] (2,-1.2)to(-2,-1.2)to(-2,2)--cycle;
  \node at (x1){$b$};
  \node at (x2){$a$};
  \node at (x3){$c$};
  \node at (x4){$e$};
  \node at (x5){$f$};
  \foreach \n/\m in {3/2,2/1,3/4,4/1,4/5,5/2}
    {\draw[arrow]
     (x\n) to (x\m);}
  \foreach \n/\m in {3/2,2/1,3/4,4/1,2/4}
    {\draw[arrow]
     (x\n) to (x\m);}
\draw[arrow] (x1.150) to (x3.30);
\draw[arrow] (x1.-150)to (x3.-30);
\draw(4,0)node[above]{$\mu_e$}
node[below]{\tiny{$\begin{array}{ccl}&h=b^{\iv{e}}&\\&k=f^{\iv{e}}&\end{array}$}};
\draw[arrow,decorate, decoration={snake}](3,0)to(5,0);
\node (x4) at (2+8,2){};
\node (x1) at (2+8,-1){};
\node (x3) at (-2+8,-1){};
\node (x2) at (-2+8,2){};
\draw[white, fill=gray!11] (2+8,-1+.2)to(-2+8,2+.2)to(2+8,2)--cycle;
\draw[white, fill=gray!11] (2+8-.2,-1)to(-2+8,-1)to(-2+8,2-.2)--cycle;
\draw[white, fill=gray!11] (8,3)to(-2+8,-1)to(10,2)--cycle;
  \node at (x1){$h$};
  \node at (x2){$a$};
  \node at (x3){$c$};
  \node at (x4){$e$};
  \foreach \n/\m in {1/3,1/4,3/2,4/2,4/3}
    {\draw[arrow]
     (x\n) to (x\m);}
\draw[arrow] (x2.-15) to (x1.105);
\draw[arrow] (x2.-60)to (x1.165);
\draw(17,1)node{
$\begin{array}{ccl}
&\Br(a,e)\\&\Br(a,c)\\&\Br(e,c)\\
\Br(e,b)&\Longleftrightarrow&\Br(e,h)\\
\Br(a,b)&\Longleftrightarrow&\Br(e^a,h)\\
\Br(e^b,c)&\Longleftrightarrow&\Br(h,c)\\
\Br(a^b,c)&\Longleftrightarrow&\Br(h^{ec},a)\\
\Br(e,f)&\Longleftrightarrow&\Br(e,k)\\
\Co(b,f)&\Longleftrightarrow&\Co(h,k)\\
\Co(c,f)&\Longleftrightarrow&\Co(e^c,k)\\
\Co(e^f,a)&\Longleftrightarrow&\Co(k,a)\\
\Co(e,f^{abc})&\Longleftrightarrow&\Co(c,k^{eah})\\
\end{array}$};
\node (x5) at (8,3){};
  \node at (x5){$k$};
  \foreach \n/\m in {1/3,1/4,3/2,4/2,3/5,5/4}
    {\draw[arrow]
     (x\n) to (x\m);}
%\draw(17,-1)node[below,gray!80]{$\Ss(a,b,c)\Longleftrightarrow\Br(c,d)\;\;$
\end{tikzpicture}
\end{gather}
For the mutation in \eqref{eq:mut9},
this is a combination of \eqref{eq:mut3} for $(a,b,c,e)$
and \eqref{eq:mut2} for $(f,b,e)$.
\begin{gather}\label{eq:mut9}
\begin{tikzpicture}[xscale=0.5,yscale=.5,
  arrow/.style={->,>=stealth},
  equalto/.style={double,double distance=2pt},
  mapto/.style={|->}]
\node (x4) at (2+8,2){};
\node (x1) at (2+8,-1){};
\node (x3) at (-2+8,-1){};
\node (x2) at (-2+8,2){};
\node (x5) at (0+8,3){};
\draw[white, fill=gray!11] (2+8,-1)to(-2+8,-1)to(2+8,2)--cycle;
\draw[white, fill=gray!11] (2+8,-1)to(-2+8,-1)to(-2+8,2)--cycle;
  \node at (x1){$b$};
  \node at (x2){$a$};
  \node at (x3){$c$};
  \node at (x4){$e$};
  \node at (x5){$f$};
  \foreach \n/\m in {3/2,2/1,3/4,4/1,5/4,2/5}
    {\draw[arrow]
     (x\n) to (x\m);}
\draw[arrow] (x1.150) to (x3.30);
\draw[arrow] (x1.-150)to (x3.-30);
\draw(4+8,0)node[above]{$\mu_e$}node[below]{\tiny{$h=b^{\iv{e}}$}};
\draw[arrow,decorate, decoration={snake}](3+8,0)to(5+8,0);
\node (x4) at (2+8+8,2){};
\node (x1) at (2+8+8,-1){};
\node (x3) at (-2+8+8,-1){};
\node (x2) at (-2+8+8,2){};
\node (x5) at (0+8+8,3){};
\draw[white, fill=gray!11] (2+8+8,-1)to(16,3)to(2+8+8,2)--cycle;
\draw[white, fill=gray!11] (2+8+8,-1)to(-2+8+8,-1)to(-2+8+8,2)--cycle;
  \node at (x1){$h$};
  \node at (x2){$a$};
  \node at (x3){$c$};
  \node at (x4){$e$};
  \node at (x5){$f$};
  \foreach \n/\m in {3/2,2/1,1/4,4/3,4/5,2/5,1/3,5/1}
    {\draw[arrow]
     (x\n) to (x\m);}
\end{tikzpicture}
\end{gather}

Next, we consider general (global) cases. We claim that for any relation $R$ in $\Br(Q,W)$,
$\theta(R)$ holds in $\Br(Q',W')$.
This will complete the proof. There are the following cases depending on the type of $R$ in Definition~\ref{def:BrQP}. Note that the numbers of arrows between two fixed vertices in $(Q,W)$ and $(Q',W')$ differ at most one.
\numbers
\item $R=\Co(c,d)$ is of type $1^\circ$.
If there is exactly one arrow between $c',d'$,
then locally the mutation is the composition of the inverse of $\mu_a$ in \eqref{eq:mut1} with the conjugation by $a$ and hence $\theta(R)$ holds. If there is no arrow between $c',d'$, then we have $\Co(c',d')$ and there are the following cases:
    \begin{itemize}
      \item Both $\theta(t)=t'$ for $t=c,d$ and then $\Co(\theta(c),\theta(d))$ holds.
      \item Both $\theta(t)=(t')^{\gamma'}$ for $t=c,d$ and then $\Co(\theta(c),\theta(d))$ holds.
      \item $\theta(c)=c'$ but $\theta(d)=(d')^{\gamma'}$.
      So there are arrows from $\gamma$ to $d$ but no arrows from $\gamma$ to $c$ in $Q$.
        As there no arrows from $c'$ to $d'$ in $Q'$, there is no arrow from $c$ to $\gamma$
        in $Q$. So we have $\Co(c,\gamma)$ and then $\Co(c',d')$ implies $\Co(c',(d')^{\gamma'})$, which is $\Co(\theta(c),\theta(d))$.
        %\[\Co(c',d')\xLongrightarrow{\Co(c,\gamma)} \Co(c',(d')^{\gamma'})\Longleftrightarrow\Co(\theta(c),\theta(d)).\]
    \end{itemize}
\item $R=\Br(b,c)$ is of type $2^\circ$.
Then there is exactly one arrow between $b$ and $c$ in $Q$.
If the number of arrows between $b'$ and $c'$ is zero or two in $Q'$, then there is exactly one path between $b$ and $c$ of length 2 and through $\gamma$. Hence the mutation is locally $\mu_a$ in \eqref{eq:mut1} or the composition of the inverse of $\mu_a$ in \eqref{eq:mut2} (via $\{b,c\}$ equaling $\{d,c\}$ there) with the conjugation by $a$. So $\theta(R)$ holds. If there is exactly one arrow between $b'$ and $c'$ in $Q'$, then we have $\Br(b',c')$ and there are the following cases:
    \begin{itemize}
      %\item If $\gamma$ is $c$ or $b$, it follows directly that $\theta(R)$ holds.
      \item Both $\theta(t)=t'$ for $t=b,c$ and then $\Co(\theta(c),\theta(d))$ holds.
      \item Both $\theta(t)=(t')^{\gamma'}$ for $t=b,c$ and then $\Co(\theta(c),\theta(d))$ holds.
      \item $\theta(c)=c'$ but $\theta(b)=(b')^{\gamma'}$.
      So there are arrows from $\gamma$ to $b$ but no arrows from $\gamma$ to $c$ in $Q$. If there are no arrows from $c$ to $\gamma$ in $Q'$, similarly as above, we deduce that  $\theta(R)$ holds. If there are arrows from $c$ to $\gamma$, then the arrow between $b$ and $c$ is from $b$ to $c$ and exactly one of the following occurs
      \begin{itemize}
      \item there are two arrows from $\gamma$ to $b$;
      \item there are two arrows from $c$ to $\gamma$.
      \end{itemize}
      Then the mutation is locally $\mu_a$ in \eqref{eq:mut2'} (via $(b,c)$ equaling $(c,a)$ there) or the composition of the inverse of $\mu_a$ in \eqref{eq:mut2'} (via $(b,c)$ equaling $(a,l)$ there) with the conjugation by $b$. Hence $\theta(R)$ holds.
    \end{itemize}
\item $R=\Co(a^b,c)$ is of type $3^\circ$.
Then the full subquiver between $a,b,c$ in $Q$ is a triangle $L$ which contributes a term in the potential $W$.
If $\gamma$ is one of $a,b,c$, without loss of generality, assuming $\gamma=a$,
then the mutation locally is $\mu_a$ in \eqref{eq:mut1} and so $\theta(R)$ holds.
If $\gamma$ is different from any of $a,b,c$, then there are the following cases.
\begin{itemize}
\item There are no arrows from any of $a,b,c$ to $\gamma$, or there are no arrows from $\gamma$ to any of $a,b,c$. Then the full subquiver with potential between $a',b',c'$ is the same as $a,b,c$. It is straightforward to check that $\theta(R)$ holds.
\item There is an arrow from $\gamma$ to $a$ and an arrow from $b$ or $c$ to $\gamma$. Note that if there is an arrow from $b$ to $\gamma$, then there is a 3-cycle between $a,b,\gamma$. So the arrow from $a$ to $b$ has to contribute two terms in $W$, a contradiction. Hence the arrow to $\gamma$ is from $c$. Then the mutation is the composition of the inverse of $\mu_a$ in \eqref{eq:mut3} (via $(a,b,c,\gamma)$ equaling $(c,e,d,a)$ there) with the conjugation by $\gamma$. Hence $\theta(R)$ holds.
\end{itemize}

%If the number of arrows between any two of $a',b',c'$ in $Q_{\TT'}$ is one, then for each of $a,b,c$, either there is an arrow from $\gamma$ to it, or there are no arrows between $\gamma$ and it. It is straightforward to check that $\theta(R)$ holds.
%Otherwise, say the number of arrows between $a,b$ changes. Since the arrow between $a,b$ in $Q_\TT$ can only contribute once (i.e. $L$) in the potential $W_\TT$ and $\gamma\neq c$, this arrow remains in $W_\TT$. Then the mutation must be $\mu_a^\sharp$ in \textbf{Case III}, with $(a,b,c)$ equaling $(d,c,e)$ there, and $\theta(R)$ holds.

\item $R$ is of type $4^\circ-7^\circ$,
then locally, the full subquiver of $R$ in $Q$, denoted by $Q_R$,
is a subquiver of the left quiver $\widetilde{Q}$ in
\eqref{eq:mut4},\eqref{eq:mut5} or \eqref{eq:mut9}.

If $\gamma$ is one of the vertices in $\widetilde{Q}$,
then the possible mutations, up to reversing all arrows,
are the mutations in \textbf{Case (III)-(V)} or their inverses composited by some conjugations. Hence $\theta(R)$ holds in this case.

If $\gamma$ is not one of the vertices in $\widetilde{Q}$, then the full subquiver between the corresponding vertices in $Q'$ is the same as $Q_R$.
There are two cases. First, there is no arrow between $\gamma$
and a vertex in $Q_R$ except for at most one vertex.
As $\Co(t',\gamma')\Leftrightarrow(t')^{\gamma'}=t'$,
we have either $\theta(t)=t'$ or $\theta(t)=(t')^{\gamma'}$ for all $t\in Q_R$,
which implies that $\theta(R)$ holds.
Second, $\gamma$ has arrows from/to at least two vertices in $Q_R$.
Then $Q_R$ is isomorphic to one of the quivers in \eqref{eq:mut3'} or \eqref{eq:mut4'} and $\gamma$ must have arrows both to or both from vertices $a,e$.
This forces that there is no arrow between $a$ and $e$, i.e.
$Q_R$ is a quiver in \eqref{eq:mut3}, say the left one, and $R=\Co(c^{ae},b)$.
Then either $\theta(t)=t'$ or $\theta(t)=(t')^{\gamma'}$ for both $t=a,e$.
Moreover, $\theta(t)=t'$ for $t=b,c$.
With $\Co(\gamma',b')$ and $\Co(\gamma',c')$, it is straightforward to check $\theta(R)$ holds.
\ends
\end{proof}

%=========================================================
\subsection{The main result}
%=========================================================

We shall state and prove our main result in the paper: finite presentations of braid/spherical twist groups via quivers with potential. Recall (Lemma~\ref{lem:4.2}) that for a decorated marked surface $\surfo$, its braid twist group $\BT(\surfo)=\BT(\TT)$ admits a set $\TT^*$ of generators for any triangulation $\TT$ of $\surfo$, where $\TT^\ast$ is the set of duals of arcs in $\TT$ (see Section~\ref{sec:MS}).
% i.e. $\BT(\surfo)$ admits a set $\TT^*$ of generators. %, w.r.t. $\TT$.
We proceed to prove that the braid twist group $\BT(\surfo)$ is isomorphic to the braid group $\Br(Q_\TT,W_\TT)$, and thus has the presentation in Definition~\ref{def:BrQP}, i.e. a finite presentation via quiver with potential.
We shall need the following two lemmas.

\begin{lemma}\label{lem:N1}
Suppose that we have a group generated by $a$, $b$ and $\eta_i$, $l\leq i\leq r$ where $l,r$ are integers with $l<-1$ and $r>1$, subject to the relations given by the following  configuration of arcs
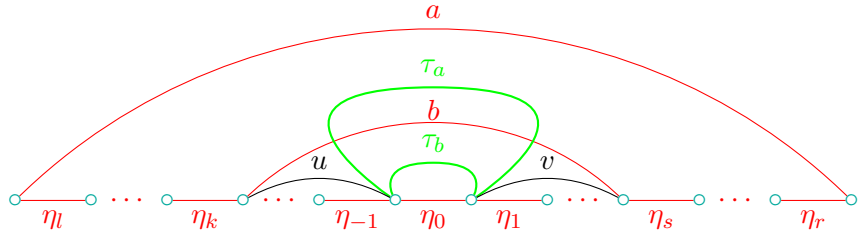
\[\begin{tikzpicture}[xscale=1,yscale=1]
\draw[red](3,0)to[bend left=45](8,0)
(0,0)to[bend left=45](11,0)
(5.5,1.2)node{$b$}(5.5,2.5)node{$a$};
\draw[green, thick]plot [smooth,tension=3] coordinates {(5,0)(5.5,.5)(6,0)};
\draw[green, thick]plot [smooth,tension=10] coordinates {(5,0)(5.5,1.5)(6,0)};
\draw[green, thick](5.5,.5)node[above]{$\tau_b$}(5.5,1.5)node[above]{$\tau_a$};;
\draw[]
(3,0)to[bend left=30](5,0)
(6,0)to[bend left=30](8,0)
(4,.5)node{$u$}
(7,.5)node{$v$};
\foreach \j/\k in {0/1,2/3,4/5,5/6,6/7,8/9,10/11}
{\draw[red](\j,0)--(\k,0);}
\foreach \j in {0,...,11}
{\draw (\j,0) node[white](x\j){$\bullet$}node[Emerald](x\j){$\circ$};}
\foreach \j in {1.5,3.5,9.5,7.5}
{\draw[red](\j,0)node{$\cdots$};}
\draw(5.5,0)node[red,below]{$\eta_0$}
(6.5,0)node[red,below]{$\eta_1$}
(4.5,0)node[red,below]{$\eta_{-1}$}
(2.5,0)node[red,below]{$\eta_k$}
(.5,0)node[red,below]{$\eta_l$}
(8.5,0)node[red,below]{$\eta_s$}
(10.5,0)node[red,below]{$\eta_r$};
\end{tikzpicture}\]
where
\begin{itemize}
	\item $\Co(\beta,\gamma)$ holds if the arcs (labeled by) $\beta$ and $\gamma$ are disjoint.
	\item $\Br(\beta,\gamma)$ holds if the arcs $\beta$ and $\gamma$ are disjoint except sharing an endpoint.
	\item $\Tr(b,\eta_{k+1},\eta_k)$ and $\Tr(b,\eta_s,\eta_{s-1})$ hold for some fixed integers $k$ and $s$ satisfying $l\leq k<-1$ and $1<s\leq r$.
\end{itemize}
%we have the following configuration of closed arcs with corresponding variables for some integers $l\leq k<-1,1<s\leq r$.

%Suppose that the following relations hold (for $\beta,\gamma\in\{\eta_j,a,b\mid l\leq j\leq r\}$):

Let $\tau_a=a^{(\eta_l \cdots \eta_{-1})(\iv{\eta_r}\cdots \iv{\eta_1})}$,
$\tau_b=b^{(\eta_{k+1} \cdots \eta_{-1})(\iv{\eta_{s-1}}\cdots \iv{\eta_1})}$
(cf. the green arcs in the figure above).
Then we have the relation
\begin{gather}\label{eq:n1}
    \Co(\tau_a^{\eta_1},\tau_b^{\eta_{-1}})
\end{gather}
\end{lemma}
\begin{proof}
Let
\[  u=\eta_{k+1}^{\eta_{k+2} \cdots \eta_{-1}},\quad
    v=\eta_{1}^{\eta_{2} \cdots \eta_{r-1}}\]
which would correspond to the violet arcs in the configuration above.
It is straightforward to see that
$\Co(a,u),\Co(a,v)$ holds.
Conjugating \eqref{eq:n1} with
$\eta_2\cdots\eta_r\iv{\eta_{-1}}\cdots\iv{\eta_l}$, we obtain $\Co(a,b^{u\iv{v}})$,
which holds as $a$ commutes with all of $b,u,v$.
\end{proof}

Similarly, we have the following.
\begin{lemma}\label{lem:N0}
Suppose that we have a group generated by $a$, $b$ and $\eta_i$, $l\leq i\leq r$ where $l,r$ are integers with $l<-1$ and $r>1$, subject to the relations given by the following configuration of arcs
%Suppose that we have the following configuration of closed arcs
%with corresponding variables for some integers $l\leq k<-1,1<s\leq r$.
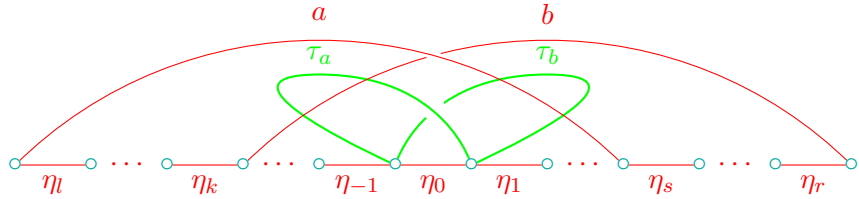
\[\begin{tikzpicture}[xscale=1,yscale=1]
\draw[green,thick]plot [smooth,tension=6] coordinates {(5,0)(7,1.2)(6,0)};
\draw[white, fill=white](5.5,.75)circle(.15);
\draw[green,thick]plot [smooth,tension=6] coordinates {(5,0)(4,1.2)(6,0)};
\draw[green,thick](4,1.2)node[above]{$\tau_a$}(7,1.2)node[above]{$\tau_b$};;
\draw[red](3,0)to[bend left=45](11,0)
    (4,2)node{$a$}(7,2)node{$b$};
\draw[white](5.5,1.45)node{\Large{$\bullet$}};
\draw[red](0,0)to[bend left=45](8,0);
\foreach \j/\k in {0/1,2/3,4/5,5/6,6/7,8/9,10/11}
     {\draw[red](\j,0)--(\k,0);}
\foreach \j in {0,...,11}
    {\draw (\j,0) node[white](x\j){$\bullet$}node[Emerald](x\j){$\circ$};}
\foreach \j in {1.5,3.5,9.5,7.5}
     {\draw[red](\j,0)node{$\cdots$};}
\draw(5.5,0)node[red,below]{$\eta_0$}
    (6.5,0)node[red,below]{$\eta_1$}
    (4.5,0)node[red,below]{$\eta_{-1}$}
    (2.5,0)node[red,below]{$\eta_k$}
    (.5,0)node[red,below]{$\eta_l$}
    (8.5,0)node[red,below]{$\eta_s$}
    (10.5,0)node[red,below]{$\eta_r$};
\end{tikzpicture}\]
%Suppose that the following relations hold (for $\beta,\gamma\in\{\eta_j,a,b\mid l\leq j\leq r\}$):
where
\begin{itemize}
	\item $\Co(\beta,\gamma)$ holds if the arcs (labeled by) $\beta$ and $\gamma$ are disjoint (note that $a$ and $b$ are disjoint here).
	\item $\Br(\beta,\gamma)$ holds if the arcs $\beta$ and $\gamma$ are disjoint except sharing an endpoint.
    \item $\Tr(b,\eta_{k+1},\eta_k),\Tr(a,\eta_s,\eta_{s-1})$ hold for some fixed integers $k$ and $s$ satisfying $l\leq k<-1$ and $1<s\leq r$.
\end{itemize}
Let $\tau_b=b^{(\eta_{k+1} \cdots \eta_{-1})(\iv{\eta_s}\cdots \iv{\eta_1})}$,
$\tau_a=a^{(\eta_{l} \cdots \eta_{-1})(\iv{\eta_{s-1}}\cdots \iv{\eta_1})}$ and
then we have the relation
\begin{gather}\label{eq:n0}
    \Co(\tau_a^{\iv{\eta_1}},\tau_b^{\eta_{-1}})
\end{gather}
\end{lemma}

Now we have the following main result.

\begin{theorem}[Presentations for braid twist groups]\label{thm:1}
Let $\TT$ be a triangulation of $\surfo$ and
$(Q_\TT,W_\TT)$ the associated quiver with potential.
Then there is a canonical isomorphism
\begin{gather}\label{eq:main}
    \kappa_\TT\colon\Br(Q_\TT,W_\TT) \to \BT(\surfo)%\xlongequal{\text{Lemma~\ref{lem:4.2}}}\BT(\surfo),
\end{gather}
sending the generators $\gamma$ to the braid twists $\gamma^*\in\TT^*$. Thus, $\BT(\surfo)$ admits the finite presentation in Definition~\ref{def:BrQP} via $(Q_\TT,W_\TT)$.
Moreover, it is compatible with the canonical isomorphisms
in Proposition~\ref{pp:invariant}, in the sense that the diagram
\begin{gather}\label{eq:diag}\xymatrix@C=3pc{
    \Br(Q_\TT,W_\TT) \ar[rd]^{\kappa_\TT}\ar[dd]^{\theta_{\gamma}^{?}} \\
    &\BT(\surfo)\\
    \Br(Q_{\TT'},W_{\TT'}) \ar[ru]^{\kappa_{\TT'}}
}\end{gather}
commutes, where $\TT'$ is the backward and forward flip of $\TT$ at $\gamma$, for $?=\flat$ and $\sharp$, respectively.
\end{theorem}
%Note that under \eqref{eq:main}, we obtain a finite presentation (as in Definition~\ref{def:BrQP}) of $\BT(\surfo)$ for each triangulation $\TT$.
\begin{figure}
\begin{tikzpicture}[xscale=.35,yscale=.35]
\foreach \j in {-4,...,3}
    {\draw[thick,Cyan]plot[smooth,tension=.7] coordinates
    {(0,-14)(-\j,-7)(90+30*\j:14)};}
\draw[thick,Cyan]plot[smooth,tension=.7] coordinates
    {(0,-14)(4,-9)(-60:14)};
\draw[thick,fill=gray!14]
    (0,-14).. controls +(55:4) and +(5:4) ..(0,-14);
\draw[thick,fill=gray!14]
    (-120:14).. controls +(0:2.5) and +(50:2.5) ..(-120:14);
\draw[thick,fill=gray!14]
    (-150:14).. controls +(30-50:2.5) and +(90-50:2.5) ..(-150:14);

\draw[thick,Cyan]
    (-150:14).. controls +(60:5) and +(105:5) ..(-90:14);
\draw[thick,Cyan]
    (-150:14).. controls +(-30:3) and +(110:7) ..(-90:14);
\draw[thick,Cyan]
    (-120:14).. controls +(90:3) and +(125:5) ..(-90:14);

\draw[thick,Cyan] (75:15.5) node {$\overrightarrow{A_1}$};
\draw[thick,Cyan] (15:15.5) node {$\overrightarrow{A_g}$};
\draw[thick,Cyan] (-45:15.5) node {$\overrightarrow{B_g}$};
\draw[thick,Cyan] (-75:15.5) node {$\overrightarrow{B_1}$};
\draw[thick,Cyan] (-135:15.5) node {$\overleftarrow{A_g}$};
\draw[thick,Cyan] (165:15.5) node {$\overleftarrow{B_g}$};
\draw[thick,Cyan] (135:15.5) node {$\overleftarrow{A_1}$};
\draw[thick,Cyan] (105:15.5) node {$\overleftarrow{B_1}$};
\draw[thick,Cyan](0,0)circle(14);
\draw[ultra thick, blue] (-90:14) arc(-90:-180:14);
\draw[ultra thick, blue] (-30:14) arc(-30:60:14);
\foreach \j/\x/\y in {2,...,11}
    {\draw (-105-30*\j:12.5) coordinate (t\j);}
\foreach \j/\x/\y in {2,...,10}
    {\draw[red] (t\j)edge(-135-30*\j:12.5)edge(-105-30*\j:14);}
    {\draw[red] (t11)--(-105-330:14);}
\draw[red] (-105:12.5) coordinate (t12)--(-105:14);
\draw[red] (-130:12.25) coordinate (t0)--(-130:14);
\draw (-140:11) coordinate (t1);
\foreach \j/\i in {2/1,1/0,0/12}
    {\draw[red](t\i)--(t\j);}

\foreach \j in {0,...,12} {\draw(t\j)node[white]{$\bullet$}node[Emerald]{$\circ$};}
\draw(t5)node[Emerald,below]{\tiny{$Z_1$}}
    (t6)node[Emerald,below]{\tiny{$Z_2$}}
    (t7)node[Emerald,below]{\tiny{$Z_4\cdots Z_{g+1}$}}
    (t8)node[Emerald,left]{\tiny{$Z_{g+2}$}}
%    (t9)node[Emerald,left]{\tiny{$Z_6$}}
    (t4)node[Emerald,below right]{\tiny{$Z_3$}}
    (t3)node[Emerald,right]{\tiny{$Z_{\aleph}$}}
    (t2)node[Emerald,right]{\tiny{$Z_{\aleph-1}$}};
\draw (90:11.5)node[red,right]{$\sigma_1$}
(120:11.9)node[red,right]{$x$}
(65:11.2)node[red,right]{$y$}; % mark generators
\foreach \j in {1,...,12}
    {\draw[Cyan](30*\j:14)node{$\bullet$};}

\draw[white,fill=white](-140:14)circle(.7)
    node[Emerald,rotate=-40]{{$\cdots$}};
\draw[white,fill=white](-10:14)circle(.7)
    node[Emerald,rotate=80]{{$\cdots$}};
 \foreach \j in {-1,2,-5}
    {\draw[white,fill=gray!0] (90+30*\j:14) circle (1);
    \draw[Cyan] (90+30*\j:14)node[rotate=30*\j]{\Large{$\cdots$}};}

\foreach \j in{1,12,11}
    {\draw(t\j)node[Emerald]{$\bullet$};}%split!
\draw(-78:12)node{\small{$\partial_1$}};
\draw(-112:13)node{\small{$\partial_{\mathrm{b}}$}};
\draw(-145:12)node{\small{$\partial_2$}};
\draw(-90:14)node[below]{$^M$};
\end{tikzpicture}\qquad
\begin{tikzpicture}[rotate=-90,xscale=.6,yscale=.35]
\draw[thick,Cyan,dashed](0,0)node{$\bullet$}(0,5)node{$\bullet$}
    edge[bend left=10](.2,2)edge[bend left=-10](-.2,2)node[black,right]{$_M$};
\draw[thick,Cyan](0,0)edge[bend left=90](0,5)(0,0)edge[bend left=-90](0,5);
\draw[thick,Cyan](0,5)edge[bend right=-20](0.5,1.5)edge[bend right=20](-0.5,1.5);
\draw[thick,fill=gray!14]plot[smooth,tension=1] coordinates
    {(0,0)(.5,1.5)(0,2)(-.5,1.5)(0,0)};
\draw[red,thick](2,4)to(.9,3)(-.9,3)to(-2,4);\draw[red,dashed,thick](.9,3)to(-.9,3);
\draw%(0,3.2)node[white]{$\bullet$}node[Emerald,rotate=90]{$\cdots$}
    (0.5,1.5)node[Cyan]{$\bullet$}
    (.9,3)node[white]{$\bullet$}node[Emerald]{$\circ$}
    (-0.5,1.5)node[Cyan]{$\bullet$}
    (-.9,3)node[white]{$\bullet$}node[Emerald]{$\circ$};

\draw[thick,Cyan](0-6,0)node{$\bullet$}(0-6,5)node{$\bullet$}node[black,right]{$_M$};
\draw[thick,Cyan](0-6,0)edge[bend left=90](0-6,5)edge[bend right=90](0-6,5);
\draw[thick,fill=gray!14]plot[smooth,tension=1] coordinates
    {(0-6,0)(.5-6,1.5)(0-6,2)(-.5-6,1.5)(0-6,0)};
\draw[red,thick](2-6,4)to(0-6,3)to(-2-6,4);
\draw(0-6,3)node[Emerald]{$\bullet$}
    (5,0)node{}(-3,2.5)node[rotate=90]{\Huge{$=$}};
\end{tikzpicture}
  \caption{A triangulation of $\surfo$}
  \label{fig:The triangulation}
\end{figure}
\begin{proof}
First, choose the triangulation $\TT_0$ of $\surfo$ as shown in Figure~\ref{fig:The triangulation}. Assume that $\aleph\geq 4$. Then there are no double arrows in the quiver $Q_{\TT_0}$. So the braid group $\Br(Q_{\TT_0},W_{\TT_0})$ admits the presentation in Proposition~\ref{prop:simple}.
%, where the surface $\surfo$ is the same as Figure~\ref{fig:poly}.
%Then only the first three types of relations (in Definition~\ref{def:BrQP}) occur in $\Br(Q_{\TT_0},W_{\TT_0})$.
Hence, by Lemma~\ref{lem:btrel}, the group homomorphism $\kappa_{\TT_0}$ is well-defined.
By Lemma~\ref{lem:4.2}, $\kappa_{\TT_0}$ is surjective.
To show that it is injective, we only need to show that
\[\{\kappa_{\TT_0}(R)\mid \text{$R$ a relation in $\Br(Q_{\TT_0},W_{\TT_0})$}\}\]
generates all the relations of $\BT(\surfo)$.
By Corollary~\ref{rmk:n=4}, we just need to consider the relations in Theorem~\ref{thm:pre}, cf. Figure~\ref{fig:BT}.
%, where $\sigma_1=Z_1Z_2$, $x=Z_1Z_3$ and $y=Z_2Z_4$.
It is easy to check the first five relations, while the last two
follows from Lemma~\ref{lem:N1} and Lemma~\ref{lem:N0}, respectively,
taking $\eta_0=\kappa_{\TT_0}(\sigma_1)$, $\eta_{-1}=\kappa_{\TT_0}(x)$, $\eta_1=\kappa_{\TT_0}(y)$, $\tau_a=\tau_r$ and $\tau_b=\tau_t$ as in Figure~\ref{fig: }. Note that the two cases are determined by the relative position of $\tau_r$ and $\tau_s$. For the case $\aleph\leq 3$, we have that either $g=1, b=1$, $|\M|=1$, or $g=0, b\leq 2$. A direct calculation can give the bijection $\kappa_{\TT_0}$.

Next, using commutativity in diagram \eqref{eq:diag} as definition,
we obtain an isomorphism $\kappa_{\TT_1}$ for any (forward/backward) flip $\TT_1$ of $\TT_0$.
It is straightforward to check that this isomorphism satisfies the condition that sending
$\gamma\in\TT_1$ to the corresponding braid twist $\eta=\gamma^*\in\TT_1^\ast$, comparing \eqref{eq:5.1}, \eqref{eq:5.2} and Figure~\ref{fig:ex0}.
Hence, any mutate sequence $p$ from $\TT_0$ to $\TT$ induces
an isomorphism $\kappa_\TT$ satisfying the condition above.
This implies that $\kappa_\TT$ does not depend on $p$.
Hence, the theorem holds for
any triangulation in the connected component of the exchange graph $\EG(\surfo)$
(of triangulations with flips, cf. \cite[Definition~3.4]{QQ}) containing $\TT_0$.
By \cite[Remark~3.10]{QQ}, all components of the exchange graph $\EG(\surfo)$
are identical, hence the theorem holds in general.
\end{proof}

\begin{figure}
\begin{tikzpicture}[scale=.4]
\draw[thick,green,font=\scriptsize](140:14)to(-2,3)(2,3)to(60:14)
    (3,7)node{$\tau_a$};
\draw[thick,red,font=\scriptsize]
    (142:14)to[bend left=30](-6,-1)(6,-1)to[bend left=30](58:14)
    (5.5,6)node{$a$};

\draw[thick,dashed,Cyan](-20:14) arc(-20:200:14);;
\draw[ultra thick] (50:14) arc(50:70:14);
\draw[ultra thick] (130:14) arc(130:150:14);

\draw[red](6,0-1)to(2,4-1)to(-2,4-1)to(-6,0-1);
\draw[dashed,white](6,0-1)to(3,3-1) (-6,0-1)to(-3,3-1);
\foreach \a/\b in {2/4,3/3,-3/3,-2/4,-6/0,6/0}{
\draw[red,font=\scriptsize]
    (\a,\b-1)node[white]{$\bullet$}node{$\circ$};
}
\draw[red,font=\scriptsize]
    (-3.5,4-1)node[below]{$_{\eta_{-1}=x\;}$}
    (3.5,4-1)node[below]{$_{\eta_{1}=y}$}
    (0,4-1)node[above]{$_{\eta_0=\sigma_1}$};
\draw[red!20,font=\scriptsize]
    (2,3)node[below]{$_{Z_1}$}
    (-2,3)node[below]{$\;_{Z_0}$}
    (-3,2)node[below]{$\;_{Z_{3}}$}
    (3,2)node[below]{$_{Z_{4}}$};
\end{tikzpicture}
  \caption{$\tau_a=\tau_r$ (similarly for $\tau_b=\tau_s$)}
  \label{fig: }
\end{figure}

%Recall that we have the following result from the prequel.

By Theorem~\ref{thmQQ}, we obtain the following corollary, which provides finite presentations for spherical twist groups.

\begin{corollary}[Presentations for spherical twist groups]\label{Cor:1}
Let $(Q,W)$ be a quiver with potential arising from a triangulated marked surface and $\Gamma$ the associated Ginzburg dg algebra. % associated to a triangulation $\TT$ of a decorated marked surface $\surfo$.
Then the spherical twist group $\ST(\Gamma)$ admits the finite presentation given in Definition~\ref{def:BrQP} via the following isomorphism of groups
\begin{gather}\label{2}
\begin{array}{ccc}
    \ST(\Gamma)&\xlongrightarrow{\cong}&\Br(Q,W)\\
    \phi_{S_\gamma}&\mapsto&\gamma
\end{array}\end{gather}
%admits a finite presentation given in Definition~\ref{def:BrQP},
where $S_\gamma$ is the simple $\Gamma(Q,W)$-module corresponds to a vertex $\gamma$ of $Q$.
\end{corollary}

We also have the following infinite presentations for these twist groups.
\begin{theorem}[Infinite presentations]\label{thm:2}
The braid twist group $\BT(\surfo)$ admits the following presentation:
\begin{itemize}
  \item Generators: $B_\eta$, $\eta\in\cA(\surfo)$;
  \item Relations: all
  conjugation relations, i.e. %for any simple closed arcs $\alpha,\beta$,
   \[B_{B_\alpha(\beta)}=B_\beta^{\iv{B_\alpha}},\ \forall\ \alpha,\beta\in\cA(\surfo).\]
\end{itemize}
Let $(Q,W)$ be a quiver with potential arising from a triangulated marked surface and $\Gamma$ the associated Ginzburg dg algebra. Then the spherical twist group $\ST(\Gamma)$ admits the following presentation
\begin{itemize}
  \item Generators: $\phi_X$, $X\in\Sph(\Gamma)$;
  \item Relations: all conjugation relations, %i.e. for any (reachable) spherical objects $S,X$,
  %all commutation/braid relations, i.e.
    %\begin{itemize}
    %  \item $\Co(\phi_S,\phi_X)$ if $\dim\Hom^\bullet(S,X)=0$;
    %  \item $\Br(\phi_S,\phi_X)$ if $\dim\Hom^\bullet(S,X)=1$.
    %  \item $\phi_{\phi_S(X)}=\phi_X^{\iv{\phi_S}}$
    %\end{itemize}
\[\phi_{\phi_X(Y)}={\phi_Y}^{\iv{\phi_X}},\ \forall\ X,Y\in\Sph(\Gamma).\]
\end{itemize}
\end{theorem}
\begin{proof}
The presentation for $\BT(\surfo)$ follows from the following three facts.
\begin{enumerate}
	\item The relations in Definition~\ref{def:BrQP}
	are either commutation or braid between braid twists of conjugations of simple closed arcs,
    which are still closed arcs.
	\item Commutation/braid relations of two simple closed arcs are given by conjugation relations. More precisely, if two closed arcs $\alpha$ and $\beta$ are disjoint, then $B_{\alpha}(\beta)=\beta$, which implies the following implications:
	%in the commutation case, we have
	\[B_{B_\alpha(\beta)}=B_\beta^{\iv{B_\alpha}} \Longrightarrow B_\beta=B_\beta^{\iv{B_\alpha}} \Longrightarrow \Co(B_\alpha,B_\beta);\]
	and if two closed arcs $\alpha$ and $\beta$ are disjoint except share a common endpoint, then ${B_\alpha(\beta)}={\iv{B_{\beta}}(\alpha)}$, which implies the following implications:
	\[\left. \begin{matrix}
	B_{B_\alpha(\beta)}=B_\beta^{\iv{B_\alpha}}\\
	B_{\iv{B_{\beta}}(\alpha)}=B_\alpha^{B_\beta}
	\end{matrix}\right\} \Longrightarrow B_\beta^{\iv{B_\alpha}}=B_\alpha^{B_\beta} \Longrightarrow \Br(B_\alpha,B_\beta).\]
	\item Any simple closed arcs can be obtained from the generators in Definition~\ref{def:BrQP} by a finite sequence of conjugations of arcs, cf. \cite[Proposition~4.3]{QQ} and \cite[Proposition~2.3]{QZ2}.
\end{enumerate}
The presentation for $\ST(\Gamma)$ follows from the isomorphism \eqref{1}
and \cite[Proposition~4.6]{QZ2} (cf. also \cite[Corollary~6.5]{QQ}) that the action of a spherical twist is compatible with the action of  the corresponding braid twist.
\end{proof}
%=========================================================
%\subsection{Positive presentations}
%=========================================================
%\[\Cc(a,e;b,c)\colon eabceab=beabcea\]
%\[\CC(a,e;b,c)\colon eabceab=beabcea\]
%\[\NY(a,e;b,c;f)\colon eabceab=beabcea\]

%\begin{lemma}
%If $\Co(a,e),\Br(a,b),\Br(a,c),\Br(e,b)$ and $\Br(e,c)$ hold,
%then,
%\begin{gather}\label{eq:c1}
%    \Co(c^a,e^b)\Longleftrightarrow\Cc(a,e;b,c) ,\\
%\label{eq:c2}
%    \Co(c^e,a^b)\Longleftrightarrow\Cc(e,a;b,c) .
%\end{gather}
%\end{lemma}
%\begin{proof}
%First,
%\begin{gather}\label{eq:a1}aeb\iv{ea}
%\xlongequal{\Br(a,e)}a\iv{b}eb\iv{a}
%\xlongequal{\Co(a,e)}a\iv{ba}eab\iv{a}
%\xlongequal{\Br(a,b)}\iv{ba}be\iv{b}ab.
%\end{gather}
%Then
%$$\iv{a}caeb\iv{e}=eb\iv{ea}ca
%\xLongleftrightarrow{\Co(a,e)} caeb\iv{aec}=aeb\iv{ea}
%\xLongleftrightarrow{\eqref{eq:a1}} caeb\iv{aec}=\iv{bae}beab
%\Longleftrightarrow \Cc(a,e;b,c)$$
%i.e. \eqref{eq:c1}.
%Similarly we have \eqref{eq:c2}.
%\end{proof}

%\newpage

%=========================================================
\appendix

%=========================================================
\section{Proof of the relations \eqref{ac:01}--\eqref{iac:07} in Section~\ref{sec:bt}}\label{app:lemmas}
%=========================================================

Set
\begin{align}\label{eq:v}
v_{t,r}=\begin{cases}
\tx&\text{if $r=t$}\\
{\tt_r}^{\tx\tt_t}&\text{if $r<t$ and $r\notin\dgen$}\\
{\tt_r}^{\tx\iv{\tt_t}}&\text{if $r<t$ and $r\in\dgen$}\\
{\tt_r}^{\iv{\tx}\iv{\tt_t}}&\text{if $r>t$ and $t\notin\dgen$}\\
{\tt_r}^{\iv{\tx}\tt_t}&\text{if $r>t$ and $t\in\dgen$}
\end{cases}
\end{align}
and
\begin{align}\label{eq:v'}
v'_{t,r}=\begin{cases}
\ty&\text{if $r=t$}\\
{\tt_r}^{\iv{\ty}\iv{\tt_t}}&\text{if $r<t$ and $r\notin\dgen$}\\
{\tt_r}^{\ty\iv{\tt_t}}&\text{if $r<t$ and $r\in\dgen$}\\
{\tt_r}^{\ty\tt_t}&\text{if $r>t$ and $t\notin\dgen$}\\
{\tt_r}^{\iv{\ty}\tt_t}&\text{if $r>t$ and $t\in\dgen$}
\end{cases}
\end{align}
Using these notation, the equalities \eqref{eq:ss03} and \eqref{eq:ss06} become
\begin{align}
\rho(\varepsilon_t).\tt_r={v_{t,r}}^{\ts_2\tt_t}\label{eq:ss03new}
\end{align}
and
\begin{align}
\rho(\iv{\varepsilon_t}).\tt_r={v'_{t,r}}^{\iv{\ts_3}\iv{\ts_1}\iv{\ts_2}},\label{eq:ss06new}
\end{align}
respectively. Here, checking the equality \eqref{eq:ss03new} is straightforward. To show the equality \eqref{eq:ss06new}, for $r=t$, we have \[\rho(\iv{\varepsilon_t}).\tt_t=\ty^{\iv{\ts_3}\iv{\ts_1}\iv{\ts_2}}\xlongequal{\eqref{eq:tx ty}}\ts_2^{\iv{\ts_1}\iv{\ts_2}}\xlongequal{\eqref{def:02}}\ts_1;\]
for $r<t$ and $r\notin\dgen$, we have \[\rho(\iv{\varepsilon_t}).\tt_r={\tt_r}^{\iv{\ty}\iv{\tt_t}\iv{\ts_3}\iv{\ts_1}\iv{\ts_2}}\xlongequal{\eqref{def:03}}{\tt_r}^{\iv{\ty}\iv{\ts_3}\iv{\tt_t}\iv{\ts_1}\iv{\ts_2}}\xlongequal{
		\eqref{eq:tx ty}
	}{\tt_r}^{\iv{\ts_3}\iv{\ts_2}\iv{\tt_t}\iv{\ts_1}\iv{\ts_2}}\xlongequal{\eqref{def:03}}{\tt_r}^{\iv{\ts_2}\iv{\tt_t}\iv{\ts_1}\iv{\ts_2}};\]
	the other cases can be proved similarly.
	%For $r<t$ and $r\in\dgen$, we have
	%\[{\tt_r}^{\ty\iv{\tt_t}\iv{\ts_3}\iv{\ts_1}\iv{\ts_2}}\xlongequal{\eqref{def:03}}{\tt_r}^{\ty\iv{\ts_3}\iv{\tt_t}\iv{\ts_1}\iv{\ts_2}}\xlongequal{
	%\eqref{eq:tx ty}
	%}{\tt_r}^{\iv{\ts_3}\ts_2\iv{\tt_t}\iv{\ts_1}\iv{\ts_2}}\xlongequal{\eqref{def:03}}{\tt_r}^{\ts_2\iv{\tt_t}\iv{\ts_1}\iv{\ts_2}}.\]
	%For $r>t$ and $t\notin\dgen$, we have
	%\[{\tt_r}^{\ty\tt_t\iv{\ts_3}\iv{\ts_1}\iv{\ts_2}}\xlongequal{\eqref{def:03}}{\tt_r}^{\ty\iv{\ts_3}\tt_t\iv{\ts_1}\iv{\ts_2}}\xlongequal{
	%\eqref{eq:tx ty}
	%}{\tt_r}^{\iv{\ts_3}\ts_2\tt_t\iv{\ts_1}\iv{\ts_2}}\xlongequal{\eqref{def:03}}{\tt_r}^{\ts_2\tt_t\iv{\ts_1}\iv{\ts_2}};\]
	%For $r>t$ and $t\in\dgen$, we have
	%\[{\tt_r}^{\iv{\ty}\tt_t\iv{\ts_3}\iv{\ts_1}\iv{\ts_2}}\xlongequal{\eqref{def:03}}{\tt_r}^{\iv{\ty}\iv{\ts_3}\tt_t\iv{\ts_1}\iv{\ts_2}}\xlongequal{
	%\eqref{eq:tx ty}
	%}{\tt_r}^{\iv{\ts_3}\iv{\ts_2}\tt_t\iv{\ts_1}\iv{\ts_2}}\xlongequal{\eqref{def:03}}{\tt_r}^{\iv{\ts_2}\tt_t\iv{\ts_1}\iv{\ts_2}}.\]

The following two (classes of) relations will be useful.

\begin{lemma}
The following hold:
\begin{align}
\Co(v_{t,r},\ty)\label{pre:01}\\
\Co(v'_{t,r},\tx)\label{pre:02}
\end{align}
for any $1\leq r,t\leq 2g+b-1$.
\end{lemma}

\begin{proof}
We only prove \eqref{pre:01}; the proof of \eqref{pre:02} is similar. For $r=t$, \eqref{pre:01} is $\Co(\tx,\ty)$ \eqref{c1:02}. For $r<t$ and $r\notin\dgen$, by \eqref{def:06}, we have the commutative relation $\Co({\tt_r}^{\tx},{\tt_t}^{\ty})$. Then we have the following implications.
\[\Co({\tt_r}^{\tx},{\tt_t}^{\ty})\Longrightarrow\Co({\tt_r}^{\tx},{\ty}^{\iv{\tt_t}})\Longrightarrow\Co({\tt_r}^{\tx\tt_t},\ty)\Longrightarrow\Co(v_{t,r},\ty)\]
where the first implication uses the braid relation $\Br(\tt_t,\ty)$ by \eqref{def:05}, the second implication is taking the conjugation by $\tt_t$, the last implication uses $v_{t,r}={\tt_r}^{\tx\tt_t}$ \eqref{eq:v}. The case $r<t$ and $r\in\dgen$ can be proved similarly.
%we have $v_{t,r}={\tt_r}^{\tx\tt_t}$. By
%\[
%\eqref{def:06}\Longrightarrow\Co({\tt_r}^{\tx},{\tt_t}^{\ty})\xLongrightarrow{\eqref{def:05}}\Co({\tt_r}^{\tx},{\ty}^{\iv{\tt_t}})\Longrightarrow\Co({\tt_r}^{\tx\tt_t},\ty)\Longrightarrow\eqref{pre:01};
%\]
%, we have $v_{t,r}={\tt_r}^{\tx\iv{\tt_t}}$ and
%\[\eqref{def:07}\Longrightarrow\Co({\tt_r}^{\tx},{\tt_t}^{\iv{\ty}})\xLongrightarrow{\eqref{def:05}
%}\Co({\tt_r}^{\tx},{\ty}^{\tt_t})\Longrightarrow\Co({\tt_r}^{\tx\iv{\tt_t}},\ty)\Longrightarrow\eqref{pre:01};\]

For $r>t$ and $t\notin\dgen$, %we have $v_{t,r}={\tt_r}^{\iv{\tx}\iv{\tt_t}}$ and
we have the commutation relation $\Co({\tt_r}^{\ty},{\tt_t}^{\tx})$ by \eqref{def:06}. Then we have the following implications.
\[\Co({\tt_r}^{\ty},{\tt_t}^{\tx})\Longrightarrow\Co({\tt_r}^{\iv{\tx}},{\tt_t}^{\iv{\ty}})\Longrightarrow\Co({\tt_r}^{\tx},{\ty}^{{\tt_t}})\Longrightarrow\Co({\tt_r}^{\iv{\tx}\iv{\tt_t}},\ty)\Longrightarrow\Co(v_{t,r},\ty)\]
where the first implication uses the commutation relation $\Co(\tx,\ty)$ \eqref{c1:02}, the second implication uses the braid relation $\Br(\tt_t,\ty)$ by \eqref{def:05}, the third implication is taking the conjugated by $\iv{\tt_t}$, and the last equality uses $v_{t,r}={\tt_r}^{\iv{\tx}\iv{\tt_t}}$ \eqref{eq:v}.
%\[
%\eqref{def:06}\Longrightarrow\Co({\tt_r}^{\ty},{\tt_t}^{\tx})\xLongrightarrow{\eqref{c1:02}}\Co({\tt_r}^{\iv{\tx}},{\tt_t}^{\iv{\ty}})\xLongrightarrow{\eqref{def:05}}\Co({\tt_r}^{\iv{\tx}},{\ty}^{\tt_t})\Longrightarrow\Co({\tt_r}^{\iv{\tx}\iv{\tt_t}},\ty)\Longrightarrow\eqref{pre:01};
%\]
The case $r>t$ and $t\in\dgen$ can be proved similarly.%, we have $v_{t,r}={\tt_r}^{\iv{\tx}\tt_t}$ and

\end{proof}

\begin{lemma}\label{lem:b2}
The relations \eqref{ac:01}, \eqref{ac:02}, \eqref{ac:03}, \eqref{iac:01}, \eqref{iac:02} and \eqref{iac:03} hold.
\end{lemma}

\begin{proof}
Using \eqref{eq:ss02},\eqref{eq:ss03}, \eqref{def:01}, \eqref{def:02}, \eqref{def:03}, \eqref{def:04} and \eqref{def:05}, it is easy to check the relations \eqref{ac:01}, \eqref{ac:02}, \eqref{ac:03} for $i>3$, \eqref{iac:01}, \eqref{iac:02}, and \eqref{iac:03} for $i>3$. %The relation \eqref{iac:03} for $i=4$ follows from $\rho(\iv{\varepsilon}).\ts_4=\ts_4$ and Lemma~\ref{lem:ac}.
The relation \eqref{ac:03} for $i=3$, conjugated by $\iv{\tt_t}\iv{\ts_2}$, becomes \eqref{pre:01}. The relation \eqref{iac:03} for $i=3$, conjugated by $\ts_2\ts_1\ts_3$, becomes \eqref{pre:02}.

\end{proof}

\begin{lemma}\label{lem:b3}
The relations \eqref{ac:04}, \eqref{ac:05}, \eqref{ac:06} and \eqref{ac:07} hold if the following relations hold, respectively.
\begin{align}
&\Br(v_{t,r},\ts_3)\label{acc:04}&&\\
&\Br(v_{t,r},\tt_t)\label{acc:05}&&\\
&\Co({v_{t,r}}^{\tt_t},{v_{t,s}}^{\ts_3})\label{acc:06}
&&\text{if $s<r$ and $s\notin\dgen$}\\
&\Co({v_{t,r}}^{\iv{\tt_t}},{v_{t,s}}^{\ts_3})\label{acc:07}
&&\text{if $s<r$ and $s\in\dgen$}
\end{align}
The relations \eqref{iac:04}, \eqref{iac:05}, \eqref{iac:06} and \eqref{iac:07} hold if the following relations hold, respectively.
\begin{align}
&\Br(v'_{t,r},\tt_t)\label{iacc:04}&&\\
&\Br(v'_{t,r},\ts_3)\label{iacc:05}&&\\
&\Co({v'_{t,r}}^{\ts_3},{v'_{t,s}}^{\tt_t})\label{iacc:06}
&&\text{if $s<r$ and $s\notin\dgen$}\\
&\Co({v'_{t,r}}^{\iv{\ts_3}},{v'_{t,s}}^{\tt_t})\label{iacc:07}
&&\text{if $s<r$ and $s\in\dgen$}
\end{align}
\end{lemma}

\begin{proof}
We only prove the first assertion, since the second can be proved similarly. Conjugated by ${\iv{\tt_t}\iv{\ts_2}\iv{\ty}}$, we have the following easy calculations:
\[\left(\rho(\varepsilon_t).\tx\right)^{\iv{\tt_t}\iv{\ts_2}\iv{\ty}}=\ts_2^{\tt_t \iv{\tt_t}\iv{\ts_2}\iv{\ty}}=\ts_2^{\iv{\ty}}=\ts_2^{\left(\iv{\ts_2}^{\ts_3}\right)}=\ts_3,\]
where the first equality uses $\rho(\varepsilon_t).\tx=\ts_2^{\tt_t}$ by \eqref{atx}, the third equality uses $\ty={\ts_2}^{\ts_3}$ \eqref{eq:tx ty}, and the last equality uses the braid relation $\Br(\ts_2,\ts_3)$ by \eqref{def:02};
\[\left(\rho(\varepsilon_t).\ty\right)^{\iv{\tt_t}\iv{\ts_2}\iv{\ty}}={\ty}^{\iv{\tt_t}\iv{\ts_2}\iv{\ty}}={\tt_t}^{y\iv{\ts_2}\iv{\ty}}={\tt_t}^{\ts_2\ts_3\iv{\ts_2}\iv{\ts_2}\ts_2\iv{\ts_3}\iv{\ts_2}}={\tt_t}^{\iv{\ts_3}}\tt_t,\]
where the first equality uses $\rho(\varepsilon_t).\ty=\ty$ by \eqref{aty}, the second equality uses the braid relation $\Br(\tt_t,\ty)$ by \eqref{def:05}, the third equality uses $\ty=\ts_3^{\iv{\ts_2}}$ \eqref{eq:tx ty}, the fourth equality uses the braid relation $\Br(\ts_2,\ts_3)$ by \eqref{def:02}, and the last equality uses the commutation relation $\Co(\tt_t,\ts_3)$ by \eqref{def:03};
\[\left(\rho(\varepsilon_t).\ts_3\right)^{\iv{\tt_t}\iv{\ts_2}\iv{\ty}}=\ts_3^{\iv{\tt_t}\iv{\ts_2}\iv{\ty}}=\ts_3^{\iv{\ts_2}\iv{\ty}}\ty,\]
where the first equality uses $\rho(\varepsilon_t).\ts_3=\ts_3$ \eqref{eq:ss02}, the second equality uses the commutation relation $\Co(\tt_t,\ts_3)$ \eqref{def:03}, and the last equality uses $\ty=\ts_3^{\iv{\ts_2}}$ \eqref{eq:tx ty};
\[\left(\rho(\varepsilon_t).\widetilde{\tau}_r\right)^{\iv{\tt_t}\iv{\ts_2}\iv{\ty}}=v_{t,r}^{\ts_2\tt_t\iv{\tt_t}\iv{\ts_2}\iv{\ty}}=v_{t,r}^{\ty}=v_{t,r}\]
where the first equality uses $\rho(\varepsilon_t).\widetilde{\tau}_r=v_{t,r}^{\ts_2\tt_t}$ \eqref{eq:ss03new}, and the last equality uses the commutation relation $\Co(v_{t,r},\ty)$.

Then the relations \eqref{ac:04}, \eqref{ac:05}, \eqref{ac:06} and \eqref{ac:07}, after conjugated by ${\iv{\tt_t}\iv{\ts_2}\iv{\ty}}$, become \eqref{acc:04}, \eqref{acc:05}, \eqref{acc:06} and \eqref{acc:07}, respectively.
%Using conjugation by $h:=\iv{\tt_t}\iv{\ts_2}\iv{\ty}$, we have

%Using conjugation by $h':=\ts_2\ts_1\ts_3$, we have
%\[\left(\rho(\iv{\varepsilon_t}).\tx\right)^{h'}\xlongequal{\eqref{atx2}}\tt_t^{\iv{\ts_1}\iv{\ts_2}h'}=\tt_t^{\ts_3}\xlongequal{\eqref{def:03}}\tt_t\]
%\[\left(\rho(\iv{\varepsilon_t}).\ty\right)^{h'}\xlongequal{\eqref{aty2}}\ty^{\ts_2\ts_1\ts_3}\xlongequal{\eqref{eq:tx ty}}\ts_3^{\ts_1\ts_3}\xlongequal{\eqref{def:01}}\ts_3\]
%\[\left(\rho(\iv{\varepsilon_t}).\ts_3\right)^{h'}\xlongequal{\eqref{eq:ss05}}\ts_3^{\ts_2\ts_1\ts_3}\xlongequal{\eqref{def:01}}\ts_3^{\ts_2\ts_3\ts_1}\xlongequal{\eqref{def:02}}\ts_2^{\ts_1}\xlongequal{\eqref{eq:tx}}\tx\]
%\[\left(\rho(\iv{\varepsilon_t}).\tt_r\right)^{h'}\xlongequal{\eqref{eq:ss06}}v'_{t,r}\]
%Then the relations \eqref{iac:04}, \eqref{iac:05}, \eqref{iac:06} and \eqref{iac:07} become \eqref{iacc:04}, \eqref{iacc:05}, \eqref{iacc:06} and \eqref{iacc:07}, respectively.
\end{proof}

\begin{lemma}\label{lem:b4}
The relations \eqref{acc:04}, \eqref{acc:05}, \eqref{iacc:04} and \eqref{iacc:05} hold.
\end{lemma}

\begin{proof}

We only show the relations \eqref{acc:04} and \eqref{acc:05}; the other two relations can be shown similarly. To show \eqref{acc:04}, note that by the braid relation $\Br(\tt_r,\tx)$ \eqref{def:04}, the equality \eqref{eq:v} becomes
\[
v_{t,r}=\begin{cases}
\tx&\text{if $r=t$}\\
{\tx}^{\iv{\tt_r}\tt_t}&\text{if $r<t$ and $r\notin\dgen$}\\
{\tx}^{\iv{\tt_r}\iv{\tt_t}}&\text{if $r<t$ and $r\in\dgen$}\\
{\tx}^{\tt_r\iv{\tt_t}}&\text{if $r>t$ and $t\notin\dgen$}\\
{\tx}^{\tt_r\tt_t}&\text{if $r>t$ and $t\in\dgen$}
\end{cases}
\]
So $v_{t,r}=\tx^{A}$, where $A$ satisfies the commutation relation $\Co(A,\ts_3)$ because of the commutation relations $\Co(\tt_t,\ts_3)$ and $\Co(\tt_r,\ts_3)$ by \eqref{def:03}. Then we have
\[{v_{t,r}}^{\ts_3}=\tx^{A\ts_3}=\tx^{\ts_3 A}={\ts_3}^{\iv{\tx}A}={\ts_3}^{\iv{A}\iv{\tx}A}={\ts_3}^{\iv{v_{t,r}}}
\]
which implies \eqref{acc:04}, where the second and the fourth equalities uses the commutation relation $\Co(A,\ts_3)$, and the third equality uses the braid relation $\Br(\tx,\ts_3)$ \eqref{c1:03}.

To show \eqref{acc:05}, for $t=r$, \eqref{acc:05} is $\Br(\tx,\tt_r)$ \eqref{def:04} by $v_{t,r}=\tx$ \eqref{eq:v}. For $r<t$ and $r\notin\dgen$, by \eqref{c1:05}, we have $\Br(\tt_r^{\tx},\tt_t)$. Then we have the following easy implications:
\[\Br(\tt_r^{\tx},\tt_t)\Longrightarrow\tt_r^{\tx}=\tt_t^{\left(\tt_r^{\tx}\right)\tt_t}\Longrightarrow{\left(\tt_r^{\tx\tt_t}\right)}^{\iv{\tt_t}}=\tt_t^{\left(\tt_r^{\tx\tt_t}\right)}\Longrightarrow\Br(\tt_r^{\tx\tt_t},\tt_t)\]
which implies \eqref{acc:05} because $v_{t,r}=\tt_r^{\tx\tt_t}$ in this case by \eqref{eq:v}. The other cases can be proved similarly.

\end{proof}

\begin{lemma}\label{lem:2cases}
The relations \eqref{acc:06}, \eqref{acc:07}, \eqref{iacc:06} and \eqref{iacc:07} hold if $t=r$ or $t=s$.
\end{lemma}

\begin{proof}
We only show the relation \eqref{acc:06}; the other relations can be shown similarly.

To show \eqref{acc:06}, if $t=r$, we have $v_{t,r}=\tx$ and $v_{t,s}={\tt_s}^{\tx\tt_t}$. By \eqref{def:03}, we have the commutation relations $\Co({\ts_3},{\tt_s})$ and $\Co({\ts_3},{\tt_t})$. So we have the following implications:
\[
\Co({\ts_3},{\tt_s})\Longrightarrow\Co({\ts_3}^{\tx},{\tt_s}^{\tx})\Longrightarrow\Co({\tx}^{\iv{\ts_3}},{\tt_s}^{\tx})\Longrightarrow\Co({\tx}^{\tt_t},{\tt_s}^{\tx\ts_3\tt_t})\\
\Longrightarrow\Co({\tx}^{\tt_t},\left({\tt_s}^{\tx\tt_t}\right)^{\ts_3})
\]
which is \eqref{acc:06}, where the first implication is taking the conjugation by $\tx$, the second implication uses the braid relation $\Br(\tx,\ts_3)$ \eqref{c1:03}, the third implication is taking the conjugation by ${\ts_3\tt_t}$, the fourth implication uses the commutation relation $\Co({\ts_3},{\tt_t})$. The case $t=s$ can be proved similarly.

\end{proof}

Note that in the case when $2g+b-1\leq2$, the assumption in Lemma~\ref{lem:2cases} holds automatically. It follows that all the relations \eqref{ac:01}--\eqref{iac:07} hold in this case. Hence by Lemma~\ref{lem:key}, we have the following result.

\begin{lemma}\label{lem:precase}
If %$\aleph\geq 4$ and
$2g+b-1\leq2$, then the braid twist group $\BT(\surfO)$ has the presentation in Theorem~\ref{thm:pre}.
%there is an group isomorphism from $\widetilde{N}$ to $\BT(\surfO)$, sending $\ts_i$ to $\sigma_i$ and sending $\tt_r$ to $\tau_r$.
\end{lemma}

We denote by $\surfO^{g,b}$ a decorated surface $\surfO$ with genus $g$ and with $b$ boundary components. So both $\surfO^{0,3}$ and $\surfO^{1,1}$ satisfy $2g+b-1=2$, i.e. they have generators $\sigma_i$, $1\leq  i\leq \aleph-1$, $\tau_1$ and $\tau_2$, where $\tau_1\notin\dgen$ for $\surfO^{0,3}$ while $\tau_1\in\dgen$ for $\surfO^{1,1}$. They can hence give a model for the relations between at most two $\tau_r$'s in the following sense.

%There are two cases for a decorated surface $\surfO$ satisfying $2g+b-1=2$ has two possible cases: either $g=1$ and $b=1$, or $g=0$ and $b=3$. We denote by the surface

%\begin{proof}
%This follows directly from Lemma~\ref{lem:b2}$\sim$\ref{lem:2cases}
%and Lemma~\ref{lem:key}.
%\end{proof}

\begin{lemma}\label{lem:emb}
For any $1\leq s<r\leq 2g+b-1$, there is a group homomorphism
\[\begin{array}{rl}
\psi_{\{s,r\}}:\BT(\surfO^{0,3})\to\widetilde{N}&\text{if $s\notin\dgen$}\\
\psi_{\{s,r\}}:\BT(\surfO^{1,1})\to\widetilde{N}&\text{if $s\in\dgen$}
\end{array}\]
sending $\sigma_i$ to $\ts_i$, sending $\tau_1$ to $\tt_s$ and sending $\tau_2$ to $\tt_r$.
\end{lemma}

\begin{proof}
This follows directly from Lemma~\ref{lem:precase} and the presentation of $\widetilde{N}$.
\end{proof}

Now we can prove the relations in Lemma~\ref{lem:2cases} for the general case.

\begin{lemma}\label{lem:last4}
The relations \eqref{acc:06}, \eqref{acc:07}, \eqref{iacc:06} and \eqref{iacc:07} hold.
\end{lemma}

\begin{proof}
By Lemma~\ref{lem:precase}, we may assume $\aleph\geq 5$. We only show the relation \eqref{acc:06}; the other relations can be showed similarly.
Due to Lemma~\ref{lem:2cases}, the case when $t= r$ or $t= s$ has been proved. Since the assumption for \eqref{acc:06} is $s<r$ and $s\notin\dgen$, we still need to consider the following cases: (1) $t<s<r$ and $t\notin \dgen$, (2) $s<t<r$ and $t\notin\dgen$, (3) $s<r<t$ and $r\notin\dgen$, (4) $t<s<r$ and $t\in\dgen$, (5) $s<t<r$ with $t\in\dgen$, (6) $s<r<t$ with $r\in\dgen$. We only prove the case (1), since the other cases can be proved similarly.

For case (1), we have $t<s<r$, $s\notin\dgen$ and $t\notin \dgen$. In the surface $\surfO^{0,3}$, using Lemma~\ref{rmk:conj} and Lemma~\ref{lem:btrel}, one can easily deduce the relations ${\tau_2}^{\iv{x}\iv{\tau_1}}={\left(\tau_1^{\iv{y}\sigma_4}\right)}^{\iv{\left({\tau_2}^{\iv{x}\iv{y}\sigma_4}\right)}}$, $\Co({\tau_2}^{\iv{x}\iv{\sigma_3}},{\tau_1}^{\iv{x}\iv{y}\sigma_4})$ and $\Co({\tau_2}^{\iv{x}\iv{\sigma_3}},{\tau_1}^{\iv{y}\sigma_4})$. By Lemma~\ref{lem:emb}, these three relations can be transfered to the following relations in $\widetilde{N}$: ${\tt_s}^{\iv{\tx}\iv{\tt_t}}={\left(\tt_t^{\iv{\ty}\ts_4}\right)}^{\iv{\left({\tt_s}^{\iv{\tx}\iv{\ty}\ts_4}\right)}}$, $\Co({\tt_r}^{\iv{\tx}\iv{\ts_3}},{\tt_s}^{\iv{\tx}\iv{\ty}\ts_4})$ and $\Co({\tt_r}^{\iv{\tx}\iv{\ts_3}},{\tt_t}^{\iv{\ty}\ts_4})$. So we have the relation $\Co({\tt_r}^{\iv{\tx}\iv{\ts_3}},{\tt_s}^{\iv{\tx}\iv{\tt_t}})$. Conjugated by $\ts_3$, we get $\Co({\tt_r}^{\iv{\tx}},{\tt_s}^{\iv{\tx}\iv{\tt_t}\ts_3})$, which implies $\Co({v_{t,r}}^{\tt_t},{v_{t,s}}^{\ts_3})$ \eqref{acc:06} because in this case  $v_{t,r}={\tt_r}^{\iv{\tx}\iv{\tt_t}}$ and $v_{t,s}={\tt_s}^{\iv{\tx}\iv{\tt_t}}$ by \eqref{eq:v}.

\end{proof}

%=========================================================
\section{Calculations for the proof of Proposition~\ref{pp:invariant}}\label{app:cal}
%=========================================================

In the following proofs, the relations and conjugations used for an equivalence, implication or equality will be labeled there.

\begin{lemma}\label{lem:mut6}
If $l=c^{\iv{b}}$, $\Co(a,e), \Br(a,b)$ and $\Br(b,e)$ hold,
then \[\Co(c^{ae},b)\Longleftrightarrow\Co(b^{ae},l).\]
\end{lemma}
\begin{proof}
We have
\[b^{\iv{ea}}\xlongequal{\Co(a,e)}b^{\iv{ae}}\xlongequal{\Br(a,b)}a^{b\iv{e}}\xlongequal{\Co(a,e)}a^{eb\iv{e}}\xlongequal{\Br(b,e)}a^{\iv{b}eb}\xlongequal{\Br(a,b)}b^{aeb}\]
and hence
\[\Co(c^{ae},b)\xLongleftrightarrow{?^{\iv{e}\iv{a}}}
\Co(c,b^{\iv{ea}})\xLongleftrightarrow{\text{above}}\Co(c,b^{aeb})
\xLongleftrightarrow{?^{\iv{b}}}\Co(c^{\iv{b}},b^{ae})
\xLongleftrightarrow{l=c^{\iv{b}}}\Co(l,b^{ae}).
\]
\end{proof}

\begin{lemma}\label{lem:mut4}
If $h=b^{\iv{e}},\Br(a,b),\Br(b,e),\Br(a,c),\Br(a,e),\Br(e,c),\Br(e^b,c)$ hold,
then
\[\Br(a^b,c)\Longleftrightarrow\Br(h^{ec},a),\quad
\Br(c^{ae},b)\Longleftrightarrow\Br(c^a,h),\quad
\Br(c^{ea},b)\Longleftrightarrow\Br(h^{ce},a)\\
.\]
\end{lemma}
\begin{proof}
For the first one, we have
\[
\Br(h^{ec},a)\xLongleftrightarrow{h=b^{\iv{e}}}\Br(b^c,a)\xLongleftrightarrow[\Br(a,c)]{?^{\iv{c}}}\Br(b,c^a)\xLongleftrightarrow[\Br(a,b)]{?^{\iv{a}}}\Br(a^b,c).
\]
For the second one, we have
$
\Br(c^a,h)\xLongleftrightarrow[?^e]{h=b^{\iv{e}}}\Br(c^{ae},b).
$
For the third one,
we have $\Br(e^b,c)\xLongrightarrow[?^e]{\Br(b,e)}\Br(b,c^e)$ and hence we have
\[
\Br(h^{ce},a)\xLongleftrightarrow{h=b^{\iv{e}}}\Br(b^{\left(c^e\right)},a)\xLongleftrightarrow[?^{b}]{\Br(b,c^e)}\Br(c^e,a^b)\xLongleftrightarrow[?^a]{\Br(a,b)}\Br(c^{ea},b).
\]
\end{proof}

\begin{lemma}\label{lem:mut4'}
In \eqref{eq:mut4'}, the last two equivalences
follow from the relations above them.
\end{lemma}
\begin{proof}
We only show the first one, since the second one is similar.
\[\Br(b^{ae},l)\xLeftrightarrow{?^{\iv{e}}}\Br(b^a,l^{\iv{e}})
\xLongleftrightarrow[\Br(a,b)]{\Br(e,l)}\Br(a^{\iv{b}},e^l)
\xLongleftrightarrow[?^{b}]{l=c^{\iv{b}}}\Br(a,e^{bc})
\xLongleftrightarrow[\Br(a,c)]{\Br(e,b)}\Br(c^a,b^{\iv{e}})
\xLeftrightarrow{?^{e}}\Br(c^{ae},b)
\]
%and
%\[\Br(b^{ea},l)\Leftrightarrow\Br(b^e,l^{\iv{a}})
%\xLongleftrightarrow[\Br(e,b)]{\Br(a,l)}\Br(e^{\iv{b}},a^l)
%\xLongleftrightarrow{l=c^{\iv{b}}}\Br(e,a^{bc})
%\xLongleftrightarrow[\Br(e,c)]{\Br(a,b)}\Br(c^e,b^{\iv{a}})
%\Leftrightarrow\Br(c^{ea},b).
%\]
\end{proof}

\begin{lemma}\label{lem:mut5'}
If $g=a^{\iv{f}},\Br(e,c), \Br(g,b),\Co(g,e)$ and $\Br(f,a)$ hold,
then we have
\[\Co(e,f^{abc})\Longleftrightarrow\Co(c^{ge},b).\]
\end{lemma}
\begin{proof}
Note that $g=a^{\iv{f}}$ and $\Br(f,a)$ imply $g=f^a$. Then we have
\[
\Co(e,f^{abc})\xLongleftrightarrow{g=f^a}\Co(e,g^{bc})\xLongleftrightarrow[\Br(e,c)]{?^{\iv{c}}}\Co(c^e,g^{b})\xLongleftrightarrow[?^g]{\Br(g,b)}\Co(c^{eg},b)\xLongleftrightarrow{\Co(g,e)}\Co(c^{ge},b).
\]
\end{proof}

\begin{lemma}\label{lem:mut5}
	Let $(Q,W)$ be the last quiver with potential in Figure~\ref{fig:local}. Then $\Br(Q,W)$ admits the following presentation.
	\begin{itemize}
		\item Generators: $a,b,c,e,f$.
		\item Relations: the relations in Lemma~\ref{lem:a} for the 3-cycle between $\{f,a,e\}$ and
		\[\Br(a,b),\ \Br(a,c),\ \Br(e,b),\ \Br(e,c),\ \Co(f,b),\ \Co(f,c),\ \Co(e,f^{abc}	).\]
	\end{itemize}
\end{lemma}

\begin{proof}
	We need to prove that the relations $\Br(a^b,c)$, $\Br(e^b,c)$, $\Br(c^{ae},b)$ and $\Br(c^{ea},b)$ hold.
	
	To show the first two relations, set $g=a^{\iv{f}}\xlongequal{\Br(f,a)}f^a$. We have $\Co(f^a,e)\xLongrightarrow{g=f^a}\Co(g,e)$ and $\Br(a,b)\xLongrightarrow{\Co(f,b)}\Br(a,b^f)\xLongrightarrow[g=a^{\iv{f}}]{?^{\iv{f}}}\Br(g,b).$
	So by Lemma~\ref{lem:mut5'}, $\Co(c^{ge},b)$ holds. Then by Lemma~\ref{lem:mut3}, we have $\Br(g^b,c)$ and $\Br(e^b,c)$. Hence we have
	\[\Br(g^b,c)\xLongrightarrow{g=a^{\iv{f}}}\Br(a^{\iv{f}b},c)\xLongrightarrow[\Co(f,c)]{\Co(f,b)} \Br(a^b,c).\]
	
	To show the third relation, we have \[\Br(a,c)\xLongrightarrow{\Co(f,c)}\Br(a,c^f)\xLongrightarrow[g=a^{\iv{f}}]{?^{\iv{f}}}\Br(g,c).\]
	So we have \[\Co(c^{ge},b)\xLongrightarrow[\Br(b,e)]{?^{\iv{e}}}\Co(c^g,e^{b})\xLongrightarrow[?^{c}]{\Br(g,c)}\Co(g,e^{bc})\xLongrightarrow{g=f^a}\Co(f^a,e^{bc})\] and then
	\begin{align*}\Br(f,e)\xLongrightarrow[\Co(f,b)]{\Co(f,c)}\Br(f^{\iv{c}\iv{b}},e)\xLongrightarrow{?^{bc}} \Br(f,e^{bc})\xLongrightarrow{\Co(f^a,e^{bc})}\Br(f^{\left(\iv{f}^a\right)},e^{bc})\xLongrightarrow{\Br(f,a)} \Br(a,e^{bc})\\\xLongrightarrow[\Br(a,c)]{?^{\iv{c}}}\Br(c^a,e^{b})\xLongrightarrow[?^{e}]{\Br(e,b)}\Br(c^{ae},b).\end{align*}
	The fourth relation can be proved similarly.
	%, we have $\Co(e,f^{abc})\xLongrightarrow{?^{\iv{c}\iv{b}}}\Co(e^{\iv{c}\iv{b}},f^a)$ and then
	%\[\Br(f,e)\xLongrightarrow[\Co(f,c)]{\Co(f,b)} \Br(f,e^{\iv{c}\iv{b}})\xLongrightarrow{\Co(e^{\iv{c}\iv{b}},f^a)}\Br(f^{\left(\iv{f}^a\right)},e^{\iv{c}\iv{b}})\xLongrightarrow{\Br(f,a)} \Br(a,e^{\iv{c}\iv{b}})\xLongrightarrow[\Br(e,c)]{\Br(a,b)}\Br(c^{ea},b).\]
\end{proof}

\begin{lemma}\label{lem:mut8}
If $f=e^{\iv{k}}, b=e^{\iv{h}}$, $\Br(e,c)$, $\Br(e,h)$, $\Co(k,a)$, $\Br(e,k)$, $\Br(a,c)$, and $\Br(e,a)$ hold,
then we have $\Co(e,f^{abc})\Longleftrightarrow\Co(c,k^{eah})$.
\end{lemma}
\begin{proof}
First, we have \[\Co(k,a)\xLongrightarrow{?^e}\Co(k^e,a^e)\xLongrightarrow[\Br(e,a)]{\Br(e,k)}\Co(e^{\iv{k}},e^{\iv{a}})\xLongrightarrow{?^{a}}\Co(e^{\iv{k}a},e)\xLongrightarrow{f=e^{\iv{k}}}\Co(f^a,e).\]
So we have
\[
\Co(e,f^{abc})\xLongleftrightarrow[\Br(e,c)]{?^{\iv{c}}}\Co(c^e,f^{ab})
\xLongleftrightarrow{\Co(f^a,e)}\Co(c^e,f^{aeb})
\xLongleftrightarrow[f=e^{\iv{k}}]{b=e^{\iv{h}}}\Co(c,e^{\iv{k}aee^{\iv{h}}\iv{e}})
\xLongleftrightarrow[\Br(e,k)]{\Br(e,h)}\Co(c,k^{eah}).
\]
\end{proof}

%=========================================================
%\section{Two decorating points case}\label{sec:1}
%=========================================================

%=========================================================
%\clearpage

%=========================================================

\begin{thebibliography}{99}
%=========================================================
% Notation
\newcommand{\au}[1]{\textrm{#1},}
\newcommand{\ti}[1]{\textrm{#1},}
\newcommand{\jo}[1]{\textit{#1}}
\newcommand{\vo}[1]{\textbf{#1}}
\newcommand{\yr}[1]{(#1)}
\newcommand{\pp}[2]{#1--#2.}
\newcommand{\arxiv}[1]{\href{http://arxiv.org/abs/#1}{arXiv:#1}}

%\bibitem{}
%  \au{}
%  \ti{}
%  \jo{} \vo{} \yr{} \pp{}{}
% \arxiv{}

\bibitem{A}
  \au{E.~Artin}
  \ti{Theorie der Z\"{o}pfe}
  \jo{Abh. Math. Sem. Univ. Hamburg} \vo{4} \yr{1925} \pp{47}{72}

\bibitem{Be}
  \au{P.~Bellingeri}
  \ti{On presentation of surface braid groups}
  \jo{J. Algebra} \vo{274} \yr{2004} \pp{543}{563}
  (\href{https://arxiv.org/abs/math/0110129}{arxiv:math/0110129})

\bibitem{BG}
  \au{P.~Bellingeri \and E.~Godelle}
  \ti{Positive presentations of surface braid groups}
  \jo{J. Knot Theory Ramifications} \vo{16} \yr{2007} \pp{1219}{1233}
  (\href{https://arxiv.org/abs/math/0503658}{arXiv:math/0503658})

\bibitem{BB}
  \au{J.~Birman \and T.~Brendle}
  \ti{Braids: A survey}
  \jo{Handbook of knot theory}, 19--103, Elsevier B. V., Amsterdam, 2005.
  (\href{https://arxiv.org/abs/math/0409205}{arxiv:math/0409205})

\bibitem{BS}
  \au{T.~Bridgeland \and I.~Smith}
  \ti{Quadratic differentials as stability conditions}
  \jo{Publ. Math. Inst. Hautes \'{E}tudes Sci.}  \vo{121} \yr{2015} \pp{155}{278}
  (\arxiv{1302.7030})

\bibitem{DWZ}
  \au{H.~Derksen, J.~Weyman \and A.~Zelevinsky}
  \ti{Quivers with potentials and their representations. I. Mutations}
  \jo{Selecta Math. (N.S.)}
  \vo{14} \yr{2008} \pp{59}{119}
  (\arxiv{0704.0649})

\bibitem{FM}
  \au{B.~Farb \and D.~Margalit}
  \ti{A primer on mapping class groups}
  Princeton Mathematical Series, 49.
  \jo{Princeton University Press, Princeton, NJ,} 2012.

\bibitem{FST}
  \au{S.~Fomin, M.~Shapiro \and D.~Thurston}
  \ti{Cluster algebras and triangulated surfaces, part I: Cluster complexes}
  \jo{Acta Math.} \vo{201} \yr{2008} \pp{83}{146}
  (\arxiv{math/0608367})

\bibitem{GMN}
  \au{D.~Gaiotto,  G.~Moore \and A.~Neitzke}
  \ti{Wall-crossing, Hitchin systems and the WKB approximation}
  \jo{Adv. Math.} \vo{234} \yr{2013} \pp{239}{403}
  (\arxiv{0907.3987})

\bibitem{G}
  \au{V.~Ginzburg}
  \ti{Calabi-Yau algebras},
  \href{https://arxiv.org/abs/math/0612139}{arXiv:math/0612139}.

\bibitem{GM}
  \au{J.~Grant \and  R.~Marsh}
  \ti{Braid groups and quiver mutation}
  \jo{Pacific J. Math.} \vo{290} \yr{2017} \pp{77}{116}
  (\arxiv{1408.5276})

\bibitem{GJP}
  \au{J.~Guaschi \and D.~Juan-Pineda}
  \ti{A survey of surface braid groups and the lower algebraic K-theory of their group rings}
  \jo{Handbook of Group actions} \vo{II 32}
  \yr{2015} \pp{23}{75}
  International Press of Boston Inc. ,
  Advanced Lectures in Mathematics
  (\arxiv{1302.6536})

\bibitem{LF}
  \au{D.~Labardini-Fragoso}
  \ti{Quivers with potentials associated to triangulated surfaces}
  \jo{Proc. Lond. Math. Soc.} \vo{98} \yr{2009} \pp{797}{839}
  (\arxiv{0803.1328})

\bibitem{KT}
  \au{C.~Kassel \and V.~Turaev}
  \ti{Braid Groups}
  Graduate Texts in Mathematics, \vo{247},
  \jo{Springer} 2008.

\bibitem{Ke}
  \au{B.~Keller}
  Deformed Calabi-Yau completions.
  With an appendix by Michel Van den Bergh.
  \jo{J. Reine Angew. Math.} \vo{654} \yr{2011} \pp{125}{180}
  (\arxiv{0908.3499})

\bibitem{KY}
  \au{B.~Keller \and D.~Yang}
  \ti{Derived equivalences from mutations of quivers with potential}
  \jo{Adv. Math.} \vo{226} \yr{2011} \pp{2118}{2168}
  (\arxiv{0906.0761})

\bibitem{KP}
  \au{R.~P.~Kent IV \and D.~Peifer}
  \ti{A geometric and algebraic description of annular braid groups}
  \jo{Internat. J. Algebra Comput.} \vo{12} \yr{2002} \pp{85}{97}

\bibitem{KQ}
  \au{A.~King \and Y.~Qiu}
  \ti{Exchange graphs and Ext quivers}
  \jo{Adv. Math.} \vo{285} \yr{2015} \pp{1106}{1154}
  (\arxiv{1109.2924})

\bibitem{KQ1}
  \au{A.~King \and Y.~Qiu}
  \ti{Cluster exchange groupoids and framed quadratic differentials} to appear in \jo{Invent. Math.}
  \arxiv{1805.00030}.

\bibitem{KS}
M.~Khovanov \and P.~Seidel.
Quivers, Floer cohomology, and braid group actions.
\emph{J. Amer. Math. Soc. }15 (2002), no. 1, 203--271.
(\arxiv{math/0006056})

\bibitem{QQ}
  \au{Y.~Qiu}
  \ti{Decorated marked surfaces: Spherical twists versus braid twists}
  \jo{Math. Ann.} \vo{365} \yr{2016} \pp{595}{633}
  (\arxiv{1407.0806})

\bibitem{Qs}
  \au{Y.~Qiu}
  \ti{The braid group for a quiver with superpotential}
  \jo{Sci. China Math.} \vo{62} \yr{2019} \pp{1241}{1256}
  (\arxiv{1712.09585})

\bibitem{QW}
  \au{Y.~Qiu \and J.~Woolf}
  \ti{Contractible stability spaces and faithful braid group actions}
  \jo{Geom. Topol.} \vo{22} \yr{2018} \pp{3701}{3760}
  (\arxiv{1407.5986})

\bibitem{QZ2}
  \au{Y.~Qiu \and Y.~Zhou}
  \ti{Decorated marked surfaces II: Intersection numbers and dimensions of Homs}
  \jo{Trans. Amer. Math. Soc.} \vo{372} \yr{2019} \pp{635}{660}
  (\arxiv{1411.4003})

\bibitem{ST}
  \au{P.~Seidel \and R.~Thomas}
  \ti{Braid group actions on derived categories of coherent sheaves}
  \jo{Duke Math. J.} \vo{108} \yr{2001} \pp{37}{108}
  (\arxiv{math/0001043})

\bibitem{S}
  \au{I.~Smith}
  \ti{Quiver algebras and Fukaya categories}
  \jo{Geom. Topol.} \vo{19} \yr{2015} \pp{2557}{2617}
  (\arxiv{1309.0452})

\end{thebibliography}
\end{document}